\documentclass{article}

\usepackage[top=1in,bottom=1in,left=1in,right=1in]{geometry}

\usepackage{authblk}
\usepackage[utf8]{inputenc} 
\usepackage[T1]{fontenc}    
\usepackage{hyperref}       
\usepackage{url}            
\usepackage{booktabs}       
\usepackage{amsfonts}       
\usepackage{nicefrac}       
\usepackage{microtype}      

\usepackage{caption}
\usepackage{subcaption}
\usepackage{amsfonts, amsmath, amssymb, amsthm}
\usepackage{dsfont}
\usepackage{mathrsfs}
\usepackage{verbatim}
\usepackage{graphicx}
\usepackage{blindtext}
\usepackage{color, colortbl}
\usepackage[table]{xcolor}
\usepackage{algorithm}
\usepackage{enumitem}
\usepackage[noend]{algpseudocode}

\newcommand{\pa}[1]{\left(#1\right)}

\newcommand{\ac}[1]{\left\{#1\right\}}
\newcommand{\cX}{\mathcal{X}}

\newcommand{\cA}{\mathcal{A}}
\newcommand{\cN}{\mathcal{N}}
\newcommand{\cM}{\mathcal{M}}

\newcommand{\cP}{\mathcal{P}}

\newcommand{\cC}{\mathcal{C}}
\newcommand{\cF}{\mathcal{F}}

\newcommand{\cbF}{\overline{\mathcal{F}}}
\newcommand{\R}{\mathbb{R}}

\newcommand{\htheta}{\mbox{${\widehat \theta}$}}
\newcommand{\hgamma}{\mbox{${\widehat \gamma}$}}
\newcommand{\homega}{\mbox{${\widehat \omega}$}}

\newcommand{\ceil}[1]{\left\lceil #1 \right\rceil}

\DeclareMathOperator*{\argmax}{argmax}
\DeclareMathOperator*{\argmin}{argmin}
\DeclareMathOperator*{\minimize}{minimize}

\DeclareMathOperator{\supp}{supp}
\DeclareMathOperator{\Image}{Range}

\newtheorem{thm}{Theorem}

\newtheorem{lem}{Lemma}

\newtheorem{defi}{Definition}
\newtheorem{remark}{Remark}

\title{The price of unfairness in linear bandits with biased feedback}

\author[1]{Solenne Gaucher \thanks{solenne.gaucher@math.u-psud.fr}}
\author[2]{Alexandra Carpentier}
\author[1]{Christophe Giraud}
\affil[1]{Laboratoire de Mathématiques d'Orsay, Université Paris-Saclay}
\affil[2]{University of Potsdam}

\begin{document}

\maketitle

\begin{abstract}
In this paper, we study the problem of fair sequential decision making with biased linear bandit feedback. At each round, a player selects an action described by a covariate and by a sensitive attribute. The perceived reward is a linear combination of the covariates of the chosen action, but the player only observes a biased evaluation of this reward, depending on the sensitive attribute. To characterize the difficulty of this problem, we design a phased elimination algorithm that corrects the unfair evaluations, and establish upper bounds on its regret. We show that the worst-case regret is smaller than $\mathcal{O}(\kappa_*^{1/3}\log(T)^{1/3}T^{2/3})$, where $\kappa_*$ is an explicit geometrical constant characterizing the difficulty of bias estimation. We prove lower bounds on the worst-case regret for some sets of actions showing that this rate is tight up to a possible sub-logarithmic factor. We also derive gap-dependent upper bounds on the regret, and  matching lower bounds for some problem instance.
Interestingly, these results reveal a transition between a regime where the problem is as difficult as its unbiased counterpart, and a regime where it can be much harder.
\end{abstract}

\section{Introduction}
Artificial intelligence is increasingly used in a wide range of decision making scenarii with higher and higher stakes, with application in online advertisement \cite{MLAdd}, credit \cite{credit}, health care \cite{healthcare}, education \cite{education} and job interviews \cite{hiring}, in the hope of improving accuracy and efficiency. Recent works have shown that the decisions made by algorithms can be dangerously biased against certain categories of people, and have endeavored to mitigate this behavior \cite{fairHiring, mortgage, chawla2021individual, BiasInML}. Studies have underlined that the main cause of algorithmic unfairness is the presence of bias in the training set \cite{BiasInML}, which led to the development of methods aiming to guarantee the fairness of the algorithms. This paper, in lines with these works, addresses the problem of online decision making under biased feedback.

Linear bandits have become a very popular tool in online decision making problems, when side information on the actions is available in the form of covariates. In the present paper, we consider a variant of this problem, where the agent only has access to an unfair assessment of the action taken, that is systematically biased against a group of actions. For example, examiners may be prejudiced against people from a minority group, and give them lower grades; similarly, algorithms trained on biased data may produce unfair assessments of the credit risk of individuals belonging to a minority group. Note that not correcting biased evaluation can have adverse effects for all parties: on the one hand, actions disadvantaged by the evaluation mechanism will be unfairly discriminated against; on the other hand, the agent may spend his budget on an unfairly advantaged action that is actually sub-optimal. The problem of sequential decision making under biased feedback can be formalized as follows.

\paragraph{Biased linear bandit problem} A player is presented with a set of $k$ distinct actions characterized by covariates $x \in \mathcal{X} \subset \mathbb{R}^{d}$, and by known sensitive attributes $z_x \in \{-1,1\}$ indicating the group of the action. At each round $t\leq T$, the player chooses the action $x_t$ and receives an unobserved reward $x_t^{\top}\gamma^*$, where $\gamma^* \in \mathbb{R}^{d}$ is the regression parameter specifying the true value of the action. The regret of the player is given by
\begin{equation}\label{eq:regret}
    R_T = \mathbb{E}\Big[\sum_{t\leq T}(x^*-x_{t})^{\top}\gamma^*\Big], \quad \text{ where } \quad x^* \in \argmax_{x \in \mathcal{X}} x^{\top}\gamma^*.
\end{equation} By contrast to the classical linear bandit, the player does not observe a noisy version of the unbiased reward $x_t^{\top}\gamma^*$. Instead, she observes an unfair evaluation $y_t$ of the value of the action $x_t^{\top}\gamma^*$, given by the following biased linear model:
$$y_t = x_t^{\top}\gamma^* + z_{x_t}\omega^* + \xi_t$$ where $\xi_t \overset{i.i.d}{\sim} \mathcal{N}(0,1)$ is a noise term. The evaluation are systematically biased against a certain group: this unequal treatment of the groups is captured by the bias parameter $\omega^* \in \mathbb{R}$.

\paragraph{Preliminary discussion}
The biased linear bandit is a variant of the linear bandit. By contrast, in the classical linear bandit model, the agent observes a noisy version of the reward. Obviously, applying directly an algorithm designed for linear bandit to biased linear bandits without correcting the evaluations would lead to a linear regret if the evaluation mechanism is prejudiced against the group of the best action in terms of reward, and if the best action in terms of feedback belongs to the advantaged group. To avoid this pitfall, one must estimate the bias in order to correct the evaluations. This implies a change in the exploration-exploitation trade-off, as exploration becomes more expensive. Indeed, in classical bandit problems, one can compare the rewards of two actions by repeatedly sampling them - or, to put it differently, one can find the best action by sampling only those actions that seem optimal. This does not hold in the biased linear bandit: if, at some point, the set of potentially optimal actions contains representatives from both groups, and does not span $\mathbb{R}^d$, one is forced to sample sub-optimal actions to estimate the bias and improve the estimation of the unbiased rewards. For this reason, classical algorithm for linear bandit that only sample actions considered as potentially optimal, such as OFUL \cite{NIPS2011_e1d5be1c} or Phase Elimination \cite{BanditBook}, can suffer linear regret. This underlines the necessity to ensure sufficient estimation of the bias parameter, even when it implies sampling sub-optimal actions.

\subsection{Related work}
Fairness in bandit problems has mostly been studied from the perspective of fair budget allocation between actions. This problem is motivated by the fact that classical bandit algorithms select sub-optimal actions only a vanishing fraction of the time, which may be undesirable in many situations. To mitigate this problem and guarantee diversity in the actions selected, some papers \cite{celis2018algorithmic,Patil2020,ClaureCMJN20,DiversityBandits2020,wang2021fairness} have proposed new algorithms ensuring fairness of the selection frequency of each action. The framework studied in this paper is different: we consider here that the mechanism for observing the rewards is unfair, and we aim at correcting it in order to maximize a (fair) true cumulative reward.

The biased linear model has been studied in the batch setting in \cite{chzhen2020minimax}, where the authors investigate the optimal trade-off between minimax risk and Demographic Parity. Detection of systematic bias, interpreted as a treatment effect, has been investigated in a batch setting in \cite{FairTreatmentEffect19}. In \cite{barik2021fair}, the authors consider a similar model, with unobserved sensitive attribute $z$ and known bias parameter $\omega^*$, under additional assumption that the sensitive attribute $z$ is independent from the covariate $x$. By contrast, we show that bias estimation is one of the main difficulties of the biased bandit problem. 

The linear bandit with biased feedback can be viewed as a stochastic partial monitoring game. With the terminology of partial monitoring, the biased problem considered in the present paper is globally observable but not locally observable: in this case, the optimal worst-case regret rate typically increases as $\tilde{O}(T^{2/3})$. This regret rate is for example achieved in the related problem of partial linear monitoring with linear feedback and linear reward using an Information Directed Sampling algorithm \cite{PartialFeedbackCOLT2020}. However, the dependence of the regret on the geometry of the action set and on the dimension $d$ remains in most cases an open question \cite{PartialFeedbackICML2014, PartialFeedbackNIPS2016,PartialFeedbackCOLT2020}. In this paper, we characterize the geometry of the biased linear bandit problem, and we investigate dependence of the regret on the gaps.

\subsection{Contribution and outline}

In this paper, we introduce the linear bandit problem with biased feedback. We design a new algorithm based on optimal design for this problem. We derive an upper bound on the worst case regret of this algorithm of order $\kappa_*^{1/3}\log(T)^{1/3}T^{2/3}$ for large $T$, where $\kappa_*$ is an explicit constant depending on the geometry of the action set. We provide matching lower bounds on some problem instances, showing that the constant $\kappa_*$ characterizes the difficulty of the action set.  Note that this regret is higher than the classical rates of order $\tilde{O}(dT^{1/2})$ obtained for $d$-dimensional linear bandits: this increase corresponds to the price to pay for debiasing the unfair evaluations.

We also characterize the gap-depend regret,  showing that it is of order $\left(\nicefrac{d}{\Delta_{\min}}\vee \nicefrac{\kappa(\Delta)}{\Delta_{\neq}^2}\right)\log(T)$, where $\Delta_{\min}$ is the minimum gap, $\Delta_{\neq}$ is the gap between the best actions of the two groups, and $\kappa(\Delta)$ corresponds to the minimum regret to pay for estimating the bias with a given variance. This bound underlines the relative difficulties of the $d$-dimensional linear bandit and of the bias estimation. When $\nicefrac{d}{\Delta_{\min}}\geq \nicefrac{\kappa(\Delta)}{\Delta_{\neq}^2}$, i.e. when one group contains all near-optimal actions, the difficulty is dominated by that of the corresponding linear bandit problem. When both groups contain near-optimal actions, and $\nicefrac{d}{\Delta_{\min}}\leq \nicefrac{\kappa(\Delta)}{\Delta_{\neq}^2}$, the regret corresponds to the price of debiasing the rewards.

The rest of the paper is organized as follows. In Section \ref{sec:algo}, we present the \textsc{Fair Phased Elimination} algorithm: we first discuss parameter estimation in Section \ref{subsec:parameter_estimation}, before presenting a sketch of the algorithm in Section \ref{subsec:algo} (a detailed version of this algorithm is provided in  Appendix \ref{sec:detailed_algo}). Then, in Section \ref{sec:worst_case}, we establish an upper bound on its worst-case regret.
In Section \ref{sec:gap_dependent}, we derive a gap-dependent upper bound on the regret of our algorithm. In Section \ref{sec:lower_bounds}, we establish lower bounds on some action sets for both the worst-case and the gap-dependent regret, showing that these rates are sharp respectively up to a sub-logarithmic factor and an absolute multiplicative constant. Additional discussions on the geometry of bias estimation are postponed to Appendix \ref{sec:kappa}.

\subsection{Notations and additional assumptions} \label{subsec:notations}
We assume that all covariates $x \in \mathcal{X}$ are distinct, which implies that the group $z_x$ of action $x$ is well defined. We also assume that no group is empty, that the set $\{\left({x \atop z_x}\right) \ : \ x \in \cX\}$ spans $\mathbb{R}^{d+1}$ (which guarantees identifiability of the parameters), and that the rewards are bounded: $\max_{x \in \mathcal{X}} \vert x^{\top}\gamma^*\vert \leq 1$. 

When necessary, we underline the dependence of the regret on the parameter $\theta$ by denoting it $R_T^{\theta}$. We denote by $a_x = \left({x \atop z_x}\right)$ the vector describing an action and its group, by $\theta^* = \left({\gamma^* \atop \omega^*}\right) \in \mathbb{R}^{d + 1}$ the unknown parameter, and by $\mathcal{A} = \left\{a_x : x \in \cX\right\}$ the set of actions and of corresponding sensitive attributes. We denote by $\Delta = (\Delta_x)_{x \in \mathcal{X}}$ the vector of gaps $\Delta_x = \max_{x' \in \mathcal{X}}(x' - x)^{\top}\gamma^*$, and by $\cC(\cX) = \left\{\gamma \in \mathbb{R}^d : \forall x \in \cX, \vert x^{\top}\gamma \vert\leq 1\right\}$ the set of admissible parameters. Note that for all $x\in \cC(\cX)$, $\Delta_x\leq 2$. For $i \leq d+1$, let $e_i$ be the $i$-th vector of the canonical  basis of $\mathbb{R}^{d+1}$, and for any matrix $M$, let $M^+$ be a generalized inverse of $M$. We denote by $\mathcal{P}^{\mathcal{X}}$ the set of probability measures on $\cX$, and $\mathcal{M}^{\mathcal{X}} = \{\mu : \cX \mapsto \mathbb{R}_+\}$. For any $\mu \in \mathcal{P}^{\mathcal{X}}$ or $\mu \in \mathcal{M}^{\mathcal{X}}$, we denote $V(\mu) = \sum_{x \in \mathcal{X}}\mu(x)a_x a_x^{\top}$ the covariance matrix corresponding to this allocation. Moreover, for $u \in \mathbb{R}^{d+1}$ (resp. $\mathcal{U} \in \mathbb{R}^{d+1}$), we denote by $\mathcal{P}^{\mathcal{X}}_u$ (resp. $\mathcal{M}^{\mathcal{X}}_u$) the measures $\mu$ in $\mathcal{P}^{\mathcal{X}}$ (resp. in $\mathcal{M}^{\mathcal{X}}$) such that $u \in \Image(V(\mu))$. For $\mathcal{U}\subset \mathbb{R}^{d+1}$, we denote by $\mathcal{P}^{\mathcal{X}}_{\mathcal{U}}$ (resp. $\mathcal{M}^{\mathcal{X}}_{\mathcal{U}}$) the measures $\mu$ such that $\mu \in \mathcal{P}^{\mathcal{X}}_u$ (resp. $\mathcal{M}^{\mathcal{X}}_u$) for all $u \in \mathcal{U}$.

\section{Fair Phased Elimination algorithm}
\label{sec:algo}
The Fair Phased Elimination algorithm belongs to the category of sequential elimination algorithms. Classical sequential elimination algorithms typically proceed by phases, indexed by $l=1,2,\ldots$. At phase $l$, these algorithms consider a set of potentially optimal actions $\cX_l$. The rewards of all actions $x \in \cX_l$ are then estimated with a given precision $O(\epsilon_{l})$, typically chosen as $\epsilon_{l} = 2^{2-l}$, by sampling actions in $\cX_l$. Actions sub-optimal by a gap larger than the precision level are then removed from the set $\cX_{l+1}$ of potentially optimal actions for the phase $l+1$.

As underlined previously, sequential elimination algorithms may suffer linear regret in the biased linear bandit problem if actions allowing to estimate the bias are discarded by the algorithm before the best group is identified. To mitigate this problem, we first estimate the biased evaluations of the potentially optimal actions, using ordinary least squares estimation. We then debias the estimations using an estimator for the bias relying on independent observations, which may be obtained by sampling sub-optimal actions. Before presenting the algorithm, let us discuss the estimation of the evaluations and of the bias parameter.

\subsection{Optimal design for parameter estimation in the biased linear bandit} 
\label{subsec:parameter_estimation}
\paragraph{G-optimal design for biased evaluation estimation}
As in the Phased Elimination algorithm \cite{BanditBook}, we rely on G-optimal design to estimate the biased evaluations $a_x^{\top}\theta^*$ with small error uniformly over a set of actions $\cX_{l}$. More precisely, for a given set of potentially optimal actions $\cX_{l}$, we compute the G-optimal design solution to the problem
\begin{equation} \label{eq:G-opt-design}
    \underset{\pi \in\cP_{\mathcal{X}_l}^{\mathcal{X}_l}}{\minimize}\ \underset{x \in \mathcal{X}_{l}}{\max}\ a_x^{\top}\left(V(\pi)\right)^{+}a_x \ . \ \ \ \ \ \ \  \ \ \ \ \ \ \ \text{(G-optimal design)}
\end{equation}
This can be done using polynomial-time algorithms, relying for example on interior points method \cite{Boyd}, or on mixed integer second-order cone programming \cite{Sagnol2009ComputingOD}. The celebrated General Equivalence theorem of Kiefer \cite{Kiefer} and Pukelsheim \cite{Pukelsheim} states that the value of Equation \eqref{eq:G-opt-design} is bounded by $d + 1$.  Let $\pi^*$ denote any design solution to the G-optimal design problem \eqref{eq:G-opt-design}, and let $\widehat{\theta}$ denote the ordinary least square estimator obtained by sampling each action $x \in \cX_{l}$ exactly $\ceil{n\pi^*(x)}$ times for a given $n>0$. Then, for all $x \in \cX_l$, the General Equivalence theorem implies that the variance of the estimate $a_x^{\top}\widehat{\theta}$ is smaller than $\nicefrac{(d+1)}{n}$. Moreover, the G-optimal design $\pi^*$ can be chosen so that it is supported by at most $\nicefrac{(d+1)(d+2)}{2}$ points, so the total number of samples is at most $n + \nicefrac{(d+1)(d+2)}{2}$.

\paragraph{$\Delta$-optimal design for bias evaluation} In this paragraph, we introduce the $\Delta$-optimal design, which is discussed in greater depth in Appendix \ref{subsec:Delta_des}. To estimate the bias parameter $\omega^*$, we use the estimator $\widehat{\omega} = e_{d+1}^{\top}\widehat{\theta}$, where $\widehat{\theta}$ is the ordinary least square estimator for the full parameter $\theta^*$. Now, if we sample each action $x\in \cX$ exactly $\mu(x)$ time, the variance of $\widehat{\omega}$ is equal to $e_{d+1}^{\top} V(\mu)^+e_{d+1}$. Given the vector of gaps $\Delta$, the design $\mu$ minimizing the regret of this exploration phase, while ensuring that the variance of $\widehat{\omega}$ is smaller than $1$, is solution of the problem
\begin{eqnarray}\label{eq:kappa_equivalent}
&& \underset{\mu \in\cM_{\mathcal{X}}^{e_{d+1}}}{\minimize}\ \sum_x\mu(x)\Delta_x \quad  \text{such that } \quad e_{d+1}^{\top} V(\mu)^+e_{d+1} \leq  1. \quad \text{ ($\Delta$-optimal design)}
\end{eqnarray}
Let us denote $\mu^{\Delta}$ a minimizer of \eqref{eq:kappa_equivalent}, and $\kappa(\Delta)= {\sum}_{x \in \cX}\mu^{\Delta}(x)\Delta_x$. Lemma \ref{lem:calcul:kappa} in Appendix \ref{sec:kappa} explains how to compute the design $\mu^{\Delta}$ in polynomial time by adapting tools from $c$-optimal design. This lemma also shows that the support of $\mu^{\Delta}$ can be chosen to be of cardinality at most $d+1$. Then, choosing each action exactly $\lceil n \mu^{\Delta}(x) \rceil$ times for a given $n>0$ allows us to estimate the bias with variance lower than $n^{-1}$ and a regret no larger than $n\kappa(\Delta) + 2(d + 1)$. Obviously, we do not know the gap vector $\Delta$ beforehand, so we must estimate it as we go.

\subsection{Outline of the Fair Phased Elimination algorithm}
\label{subsec:algo}
The Fair Phased Elimination  algorithm, sketched in Algorithm \ref{alg:FPE:short}, relies on the following key ideas.  First, note that within a group, the order of the true rewards and of the biased evaluations are the same. Hence, within a group, we can use classical algorithms for linear bandits to choose the actions and estimate the biased evaluations with a controlled within-group regret: this is done using \textbf{$\mathbf{G}$-exploration and elimination}. Second, to compare actions belonging to different groups, we independently estimate the bias parameter $\omega^*$, using \textbf{$\mathbf{\Delta}$-exploration and elimination}. Finally, we underline that bias estimation may require to sample very sub-optimal actions. Therefore, it can be overly costly to estimate the bias up to the precision level required to identify the best group. To prevent this, we use a \textbf{stopping criteria}.

\paragraph{G-exploration and elimination}
At each phase $l = 1,2, ...$, we keep two sets of potentially optimal actions belonging to the groups $+1$ and $-1$, denoted respectively $\cX_l^{(+1)}$ and $\cX_l^{(-1)}$. If we have not identified the group containing the best action, we run a \textsc{G-Exp-Elim} routine \ref{alg:G-Explore and Eliminate} on each set $\cX_l^{(z)}$ for $z = 1$ and $z = -1$. This routine samples actions according to a rounded G-optimal design on $\cX_l^{(z)}$, with a total number of observations chosen so that the biased evaluations of all actions in $\cX_l^{(z)}$ are known with an error at most $\epsilon_l$. The set $\cX_{l+1}^{(z)}$ is obtained by removing from $\cX_{l}^{(z)}$ actions whose estimated evaluations are sub-optimal by a gap larger than $3\epsilon_l$, compared to the empirical best action in the group. This allows to ensure that only actions sub-optimal by a gap $\mathcal{O}(\epsilon_l)$ remain in $\cX_{l+1}^{(z)}$, and to estimate the gap vector $\Delta$ with a precision sufficient for $\Delta$-optimal estimation. 

If the group containing the best action has been identified, we discard the other group, and run a \textsc{G-Exp-Elim} routine \ref{alg:G-Explore and Eliminate} on the set of potentially optimal actions in this group.
\floatname{algorithm}{Routine}
\begin{algorithm}[h!] 
\caption{\textsc{ G-Exp-Elim}\,($\mathcal{X},n,\epsilon$)}\label{alg:G-Explore and Eliminate}
\begin{algorithmic}[1]
\State Compute G-optimal design $\pi$ solution of (\ref{eq:G-opt-design}) on $\mathcal{X}$, with $|\supp(\pi)|\leq \nicefrac{(d+1)(d+2)}{2}$
\State Sample $\ceil{n\pi(x)}$ times each action $a_{x}$ for $x\in \mathcal{X}$ \Comment{G-optimal parameter estimation}
\State Compute the ordinary least square estimator $\widehat \theta$
\State $\mathcal{X}^{'}\gets \ac{x\in\cX:\max_{x'\in\cX}(x'-x)^\top \widehat \theta\leq 3\epsilon}$ \Comment{Suboptimal actions elimination}
\State \Return $\widehat \theta$ and $\cX'$
\end{algorithmic}
\end{algorithm}

\paragraph{$\Delta$-exploration and elimination}
If the group of the best action has not been found before phase $l$, we run the \textsc{$\Delta$-Exp-Elim} routine \ref{alg:Delta-Explore and Eliminate}. More precisely, relying on a previous estimate $\widehat{\Delta}^l$ of the gap vector $\Delta$, we compute the $\widehat{\Delta}^l$-optimal design $\widehat{\mu}$. We then estimate the bias using actions sampled according to a rounded version of this design, with a total number of observations chosen so that the error of bias estimation is smaller than $\epsilon_l$, and use it to debias the reward estimation. If the debiased evaluation of the best action of each group are separated by a gap larger than $4\epsilon_l$, we consider that the best group is the one containing the empirical best action in terms of biased evaluation, and we discard the other group.

If we cannot find the best group, we rely on estimates of the bias and of the biased evaluations obtained during the previous round to update the estimate of the gap vector $\widehat{\Delta}^{l+1}$. 

\begin{algorithm}[h!] 
\caption{\textsc{ $\Delta$-Exp-Elim}\,($\mathcal{X},(\cX^{(z)},\widehat \theta^{(z)})_{z\in\ac{-1,1}},\widehat\Delta,n,\epsilon$)}\label{alg:Delta-Explore and Eliminate}
\begin{algorithmic}[1]
\State Compute $\widehat\Delta$-optimal design $\big(\hat\mu,\kappa(\hat\Delta)\big)$ solution of (\ref{eq:kappa_equivalent})  on $\cX$, with $|\supp(\hat\mu)|\leq d+1$
\State Sample $\ceil{n\hat \mu(x)}$ times each action $a_{x}$ for $x\in \mathcal{X}$ \Comment{$\widehat\Delta$-optimal bias estimation}
\State Compute $\homega=e_{d+1}^\top \widehat\theta$, where $\widehat \theta$ is the ordinary least square estimator
\For{$z\in\ac{-1,1}$ and $x\in\cX^{(z)}$}{ $\widehat m_{x}\gets a_{x}^\top \widehat \theta^{(z)} -z\homega$} \Comment{Debiased rewards estimation}
\EndFor
\If{$\exists z\in\ac{-1,1}$ such that $\displaystyle{\max_{x\in\cX^{(z)}}\widehat m_{x} \geq \max_{x\in\cX^{(-z)}}\widehat m_{x} +4\epsilon  }$}{ $\mathcal{Z}\gets \ac{z}$} \Comment{Group elimination}
\Else { $\widehat\Delta_{x}\gets 2\wedge\pa{\max_{x'\in\cX^{(-1)}\cup \cX^{(1)}}\widehat m_{x'}-\widehat m_{x}+4\epsilon}$ \ for all $x\in\cX^{(-1)}\cup \cX^{(1)}$}
\EndIf
\State \Return $\mathcal{Z}$ and $\widehat \Delta$
\end{algorithmic}
\end{algorithm}

\paragraph{Stopping criteria } As underlined previously, the \textsc{$\Delta$-Exp-Elim} routine samples actions that can be very sub-optimal. As a consequence, when the gap between the best two actions of each group is small, finding the best group can be overly costly in terms of regret. To prevent this, if the best group has not been found at stage $l$ fulfilling $\epsilon_{l}\leq \big(\nicefrac{\kappa(\widehat\Delta^{l}) \log(T)}{T}\big)^{1/3}$, the bias estimation is stopped and the empirical best action in $\cX^{(1)}_{l+1}\cup\cX^{(-1)}_{l+1}$ is sampled for the remaining time (see Algorithm \ref{alg:FPE:short})
\floatname{algorithm}{Algorithm}

\begin{algorithm}[h!] 
\caption{\textsc{Fair Phased Elimination} (sketched)} \label{alg:FPE:short}
\begin{algorithmic}[1]
\State {\bf input:} $\delta$, $T$, $\mathcal{X}$, $k=|\mathcal{X}|$, $\epsilon_l=2^{2-l}$ for $l\geq 1$
\State {\bf initialize:} $\mathcal{X}_1^{(+1)} \gets \{x : z_x = 1\}$, $\mathcal{X}_1^{(-1)} \gets \{x : z_x = -1\}$, \\ \quad \quad \quad \quad $\mathcal{Z}_1 \gets \{-1,+1\}$, $\widehat\Delta^{1}\gets \left(2, ..., 2\right)$, $l \gets 0$
\While{the budget is not spent} {$l\gets l + 1$}
\For{$z \in \mathcal{Z}_l$}
\State $\left(\widehat \theta^{(z)},\mathcal{X}_{l+1}^{(z)}\right) \gets \textsc{G-Exp-Elim}\left(\mathcal{X}_{l}^{(z)},{2(d+1)\over \epsilon_l^{2}}\log\left(\frac{kl(l+1)}{ \delta}\right), \epsilon_l\right)$ 
\EndFor
\If{ $\mathcal{Z}_{l} = \{-1,+1\}$}
\If{ $\epsilon_{l}\leq \left(\kappa(\widehat\Delta^{l}) \log(T)/T\right)^{1/3}$} \Comment{Stop bias estimation}
\State Sample best action in $\mathcal{X}_{l+1}^{(-1)} \cup \mathcal{X}_{l+1}^{(+1)}$ for the remaining time
\Else
\State $\pa{\mathcal{Z}_{l+1},\widehat \Delta^{l+1} }\gets \Delta\textsc{-Exp-Elim}\left(\cX,\Big(\mathcal{X}^{(z)}_{l+1},\widehat\theta_{l}^{(z)}\Big)_{z\in\ac{-1,1}},\widehat\Delta^l,{2\over\epsilon_l^{2}}\log\left(\frac{l(l+1)}{ \delta}\right), \epsilon_l\right)$
\EndIf
\EndIf
\EndWhile
\end{algorithmic}
\end{algorithm}

\section{Upper bound on the worst-case regret of \textsc{Fair Phased Elimination}} \label{sec:worst_case}

The regret of the \textsc{Fair Phased Elimination} depends on the difficulty of estimating the bias parameter, captured by $\kappa(\Delta)$. Lemma \ref{lem:kappa2} in Appendix \ref{subsec:Delta_des} shows that for all parameter $\gamma^*\in \cX$, $\kappa(\Delta)$ is upper bounded by $2\kappa_*$, where $\kappa_*$ is the \textit{minimal variance of the bias estimator} given by $$\kappa_*=\underset{\pi \in \mathcal{P}^{\mathcal{X}}_{e_{d+1}}}{\min}\ e_{d+1}^{\top}\left(V(\pi)\right)^{+}e_{d+1}.$$ 
The following theorem provides a bound on the worst case regret depending on $\kappa_*$. Proofs are postponed to Appendix \ref{app:proof_upper_bound_worst_case}.

\begin{thm}\label{thm:upper_bound_worst_case}
For the choice $ \delta = T^{-1}$, there exists two numerical constants $C,C'>0$ such that the following bound on the regret of the \textsc{Fair Phased Elimination} algorithm \ref{alg:FPE} holds
\begin{align*}
    R_T & \leq C\left(\kappa_{*}^{1/3}T^{2/3}\log(T)^{1/3} + (d\vee \kappa_{*})\log(T) + d^2 +  d\kappa_{*}^{-1/3}T^{1/3}\log(kT)\log(T)^{-1/3} \right)\\
    & \leq C'\kappa_{*}^{1/3}T^{2/3}\log(T)^{1/3} \quad\quad \textrm{for}\quad T\geq \frac{\left((d\vee \kappa_*)^{3/2}\log(T)\right)  \vee d^3}{\sqrt{\kappa_*}} \vee \frac{(d\log(kT))^{3}}{(\kappa_*\log(T))^{2}}.
\end{align*}
\end{thm}
In Section \ref{subsec:lower_bounds_worst_case}, we show that the upper bound obtained in Theorem  \ref{thm:upper_bound_worst_case} is sharp in some settings, up to the sub-logarithmic factor $\log(T)^{1/3}$. 

Theorem \ref{thm:upper_bound_worst_case} shows that the worst-case regret of the Fair Phased Elimination  algorithm asymptotically grows as $C\kappa^{1/3}_{*}T^{2/3}\log\left(T\right)^{1/3}$. This worst-case regret rate is higher than the typical rate $Cd\log(T)T^{1/2}$ obtained under unbiased feedback on the rewards (see, e.g., \cite{NIPS2011_e1d5be1c}). This increase in the regret corresponds to the cost of learning from unfair evaluations. It is due to the fact that the algorithm may need to sample actions that are sub-optimal in order to estimate the bias parameter. Note that this rate $\widetilde{\mathcal{O}}(T^{2/3})$ is typical for globally observable bandit problems with partial linear monitoring, and can be obtained by applying results established in \cite{PartialFeedbackCOLT2020} for in the partial linear monitoring setting to the biased linear bandit problem.

By contrast to previous results, Theorem \ref{thm:upper_bound_worst_case} characterizes precisely the dependence of the worst-case regret on the geometry of the action set. The relevant constant $\kappa_*$ is the minimal variance for estimating the bias, which appears when considering the related $c$-optimal design problem. While the connection between G-optimal design and the linear bandit problem has already been exploited, it is to the best of our knowledge the first time that $c$-optimal design is related to a partial monitoring problem.

The constant $\kappa_*$ corresponds to the minimum number of samples required for estimating the bias with a variance equal to $1$ (up to rounding issues). Intuitively, if the actions are very correlated with their sensitive attributes, more samples will be needed to estimate the bias with the same precision. This situation corresponds to cases where $\kappa_*$ is large, and leads to a higher regret. Lemma \ref{lem:upsilon_margin}, illustrated in Figure \ref{fig:margin}, relates $\kappa_*$ to the margin between the two groups of actions.

\begin{lem}\label{lem:upsilon_margin} 
$\kappa_*$ is the largest constant $\kappa\geq 0$ such that, there exists an hyperplane $\mathcal{H}$ containing zero and separating the two groups, and such that, the margin to $\mathcal{H}$ is at least $\nicefrac{\sqrt{\kappa}-1}{\sqrt{\kappa}+1}$ times the maximum distance of all points to the hyperplane (see Figure \ref{fig:margin}). When no such hyperplane exists, then $\kappa_* = 1$.
\end{lem}
\begin{figure}[h!]

\begin{subfigure}[b]{0.48\textwidth}
\centering
\includegraphics[width=0.9\textwidth]{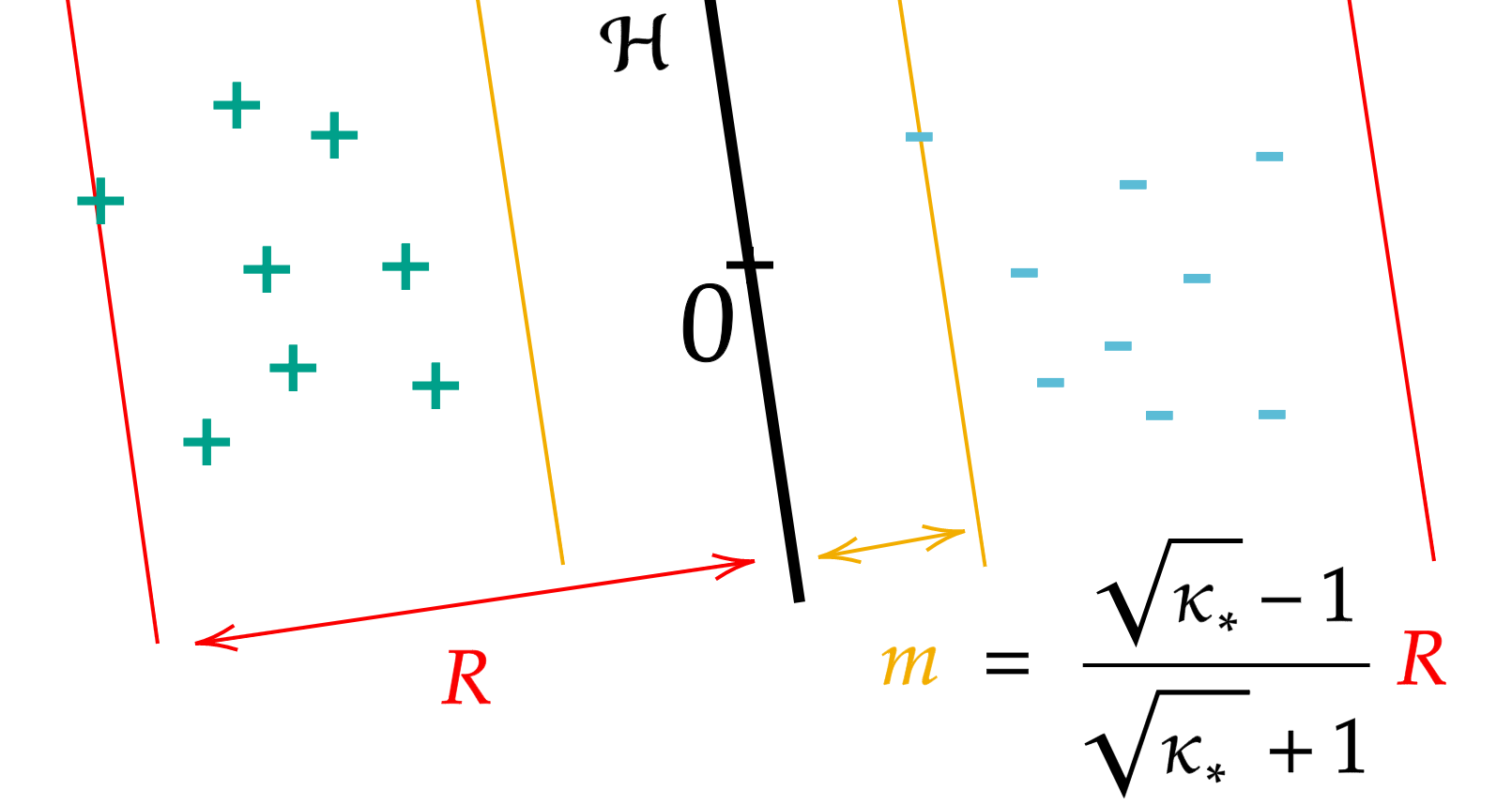}
\caption{The margin $m$ is equal to $\nicefrac{\sqrt{\kappa_*}-1}{\sqrt{\kappa_*}+1}$ times the maximum distance $R$ of any action to the hyperplane.}
\label{subfig:margin}
\end{subfigure}
\hspace{0.5cm}
\begin{subfigure}[b]{0.48\textwidth}
\centering
\includegraphics[width=0.9\textwidth]{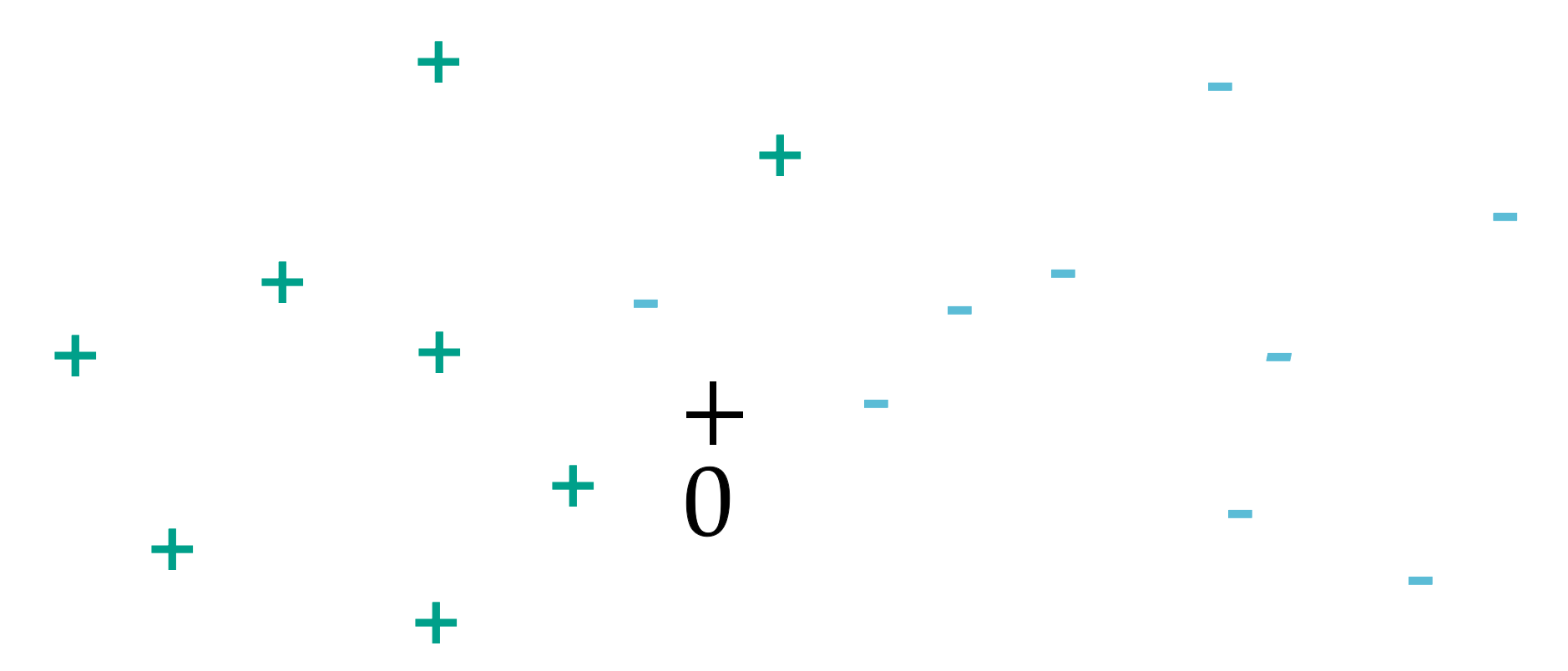}
\caption{$\kappa_* = 1$: the groups cannot be separated by a hyperplane containing $0$.}
\label{subfig:nomargin}
\end{subfigure}
 \caption{Interpretation of $\kappa_*$ in terms of separation of the groups.}
\label{fig:margin}
\vspace{-0.5cm}
\end{figure}

Interestingly, Lemma \ref{lem:upsilon_margin} underlines that under reasonable assumptions, the constant $\kappa_*$ may not depend on the ambient dimension $d$, and it can even be equal to $1$. By contrast, the previous bounds obtained for an Information Directed Sampling algorithm are of order $\alpha^{1/3}d^{1/2}T^{2/3}\log(kT)^{1/2}$, where $\alpha$ is a measure of the complexity of the action set called the worst-case alignment constant. Lemma \ref{lem:upsilon_alpha} in Appendix \ref{sec:kappa} shows that $\alpha$ is equivalent to the minimal variance of the bias estimator $\kappa_*$. Hence, our bound improves over previous results by a factor $d^{1/2}\log(T)^{1/6}(\log(kT)/\log(T))^{1/2}$.

The gaps are not involved in the definition of the minimal variance of bias estimation $\kappa_*$. The reader may have expected to get, instead of $\kappa_*$, the minimax regret for estimating the bias $$\widetilde{\kappa} =  \underset{\gamma\in \cC(\cX), x' \in \cX}{\max} \underset{x \in \cX}{\sum}\widetilde{\mu}(x)(x' - x)^{\top}\gamma,\quad\ \text{where}$$ 
$$ \displaystyle \widetilde{\mu} = \underset{\mu}{\argmin} \underset{x' \in \cX, \gamma \in \cC(\cX)}{\max} \underset{x \in \cX}{\sum}\mu(x)(x' - x)^{\top}\gamma, \ 
\text{such that }  \mu \in \cM^{\cX}_{e_{d+1}}\ 
 \text{and }  \ e_{d+1}^{\top} V(\mu)^+e_{d+1} \leq  1.$$
Next lemma shows that $\kappa_*$ and $\widetilde{\kappa}$ are in equivalent up to  a factor 2. We refer the interested reader to Appendix \ref{sec:kappa}, where further discussions on the geometry of bias estimation are postponed, due to space constraints.
\begin{lem}\label{lem:kappa_tilde}
$\displaystyle \nicefrac{\widetilde{\kappa}}{2} \leq \kappa_* \leq 2\widetilde{\kappa}.$
\end{lem}


\section{Upper bound on the gap-depend regret of \textsc{Fair Phased Elimination}}
\label{sec:gap_dependent}

In this section, we provide an upper bound on the worst-case regret that depends on the gap between the two best actions, and on the gap between the best actions of the two groups. Compared to instance-dependent bounds, established in the linear bandit problem in \cite{pmlr-v54-lattimore17a, AOIDS}, gap-dependent bounds characterize the dependence of the regret on a small number of parameters. They are typically less sharp than instance-dependent bounds, but allow to better highlight the influence of the parameters on the difficulty of the problem. The bound established in the following theorem relates the difficulty of the biased linear bandit to that of bias estimation, and to that of the corresponding $d$-dimensional linear bandit. Proofs are postponed to Appendix \ref{app:proof_upper_bound_worst_case}.

\begin{thm}\label{thm:upper_bound_delta}
Assume that $x^*\in \argmax_{x \in \mathcal{X}} x^{\top}\gamma^*$ is unique. Then, there exists two numerical constants $C,C'>0$ such that, for the choice $ \delta = T^{-1}$, the following bound on the regret of the \textsc{Fair Phased Elimination} algorithm \ref{alg:FPE} holds
\begin{eqnarray*}
         R_T &\leq &  C\left( \left(\frac{d}{\Delta_{\min}} \lor \frac{\kappa\big(\Delta\vee \Delta_{\neq}\vee \varepsilon_{T}\big)}{\Delta_{\neq}^2}\right)\log(T) + d^2  +\frac{d}{\Delta_{\min}} \log\left(k\right) \right)\\
         &\leq&C' \left(\frac{d}{\Delta_{\min}} \lor \frac{\kappa\big(\Delta\vee \Delta_{\neq}\vee \varepsilon_{T}\big)}{\Delta_{\neq}^2}\right)\log(T) \quad \quad\textrm{for } \quad T\geq k\vee e^{d\Delta_{\min}}\
\end{eqnarray*}
where  $\Delta_{\min} = \min_{x \in \mathcal{X}\setminus x^*}\Delta_{x}$, $\Delta_{\neq} = \min_{x \in \mathcal{X} : z_x = -z_{x^*}}\Delta_{x}$, and $ \varepsilon_{T}= (\nicefrac{\kappa_{*} \log(T)}{T})^{1/3}.$
\end{thm}

The term $\nicefrac{d}{\Delta_{\min}} \lor\nicefrac{\kappa(\Delta \vee \Delta_{\neq}\vee\varepsilon_T)}{\Delta_{\neq}^2}$ highlights the two sources of difficulty of the problem. On the one hand, the term $\nicefrac{d}{\Delta_{\min}}$ is unavoidable: even if the algorithm knew beforehand the group containing the best action, it would still need to play a game of $d$-dimensional linear bandits in this group, and suffer, in the worst-case, the corresponding gap-dependent regret \cite{NIPS2011_e1d5be1c}. Note that lower bounds on gap-depend regret of classical linear bandits follow from considering a setting with one near-optimal action with gap $\Delta_{\min}$ in each of the $d$ dimensions. Then, any algorithm needs to explore each dimension up to $\Delta_{\min}^{-2}\log(T)$ times in order to find the best action, but can do so by choosing the near-optimal actions, thus having a regret $\Delta_{\min}^{-1}\log(T)$ in each direction. By contrast, the term $\nicefrac{\kappa(\Delta \vee \Delta_{\neq}\vee\varepsilon_T)}{\Delta_{\neq}^2}$ is characteristic of the biased linear bandit problem: it is due to the fact that the algorithm may need to sample very sub-optimal actions in order to find the group containing the best action. Indeed, to identify this group, one must estimate the bias with a precision $\Delta_{\neq}$, i.e. sample sub-optimal actions with average regret $\kappa(\Delta)$ approximately $\Delta_{\neq}^{-2}\log(T)$ times.

When $\nicefrac{d}{\Delta_{\min}} \leq \nicefrac{\kappa(\Delta \vee \Delta_{\neq}\vee\varepsilon_T)}{\Delta_{\neq}^2}$, the regret corresponds to the regret of this bias estimation phase. In other words, when both groups contain near-optimal actions, the difficulty of the problem is dominated by the price to pay for debiasing the unfair evaluations. Interestingly, when $\nicefrac{d}{\Delta_{\min}} > \nicefrac{\kappa(\Delta \vee \Delta_{\neq}\vee\varepsilon_T)}{\Delta_{\neq}^2}$, the difficulty of the linear bandit with systematic bias is dominated by that of the classical $d$-linear bandit. In this case, the algorithm is able to find the group containing the best action, and the problem reduces to a linear bandit in dimension  $d$. Thus, the linear bandit with systematic bias is a non trivial example of a globally observable game that can be locally observable around the best action.

Finally, we underline that the magnitude of the bias does not appear in the regret: intuitively, no matter its magnitude, the algorithm always need to estimate it up to the same precision (of order $\Delta_{\neq}$) in order to find the best group and to be optimal in terms of gap-depend regret. This indicates that our algorithm is robust against important discriminations in the evaluation mechanism.


\section{Lower bounds on the regret}\label{sec:lower_bounds}

In this section, we derive lower bounds on the worst-case regret and the gap-dependent regret that respectively match the upper bounds established in Theorems \ref{thm:upper_bound_worst_case} and \ref{thm:upper_bound_delta} up to sub-logarithmic factors or numerical constants.

\subsection{Lower bound on the worst-case regret}\label{subsec:lower_bounds_worst_case}
Theorems \ref{thm:upper_bound_worst_case} and  \ref{thm:upper_bound_delta} underline the dependence of the regret on the geometry of the action set. Before stating our result, we begin by introducing the notion of $\kappa_*$-correlated action set.
\begin{defi}[$\kappa_*$-correlated action set]
For $\kappa_* \geq 1$, a set of actions $\mathcal{A}$ is $\kappa_*$-correlated if $\mathcal{A} \in \mathbf{A}_{\kappa_*, d}$, where
\begin{equation*}
\mathbf{A}_{\kappa_*, d} = \left\{
\begin{tabular}{l}
    $\mathcal{A} = \left\{a_1, ..., a_k \right\} \subset \Big(\mathbb{R}^{d} \times \{-1,+1\}\Big)^{k} :$\\
    $k\in \mathbb{N}^{*},\underset{\pi \in \mathcal{P}^{\mathcal{A}}_{e_{d+1}}}{\min}\left\{e_{d+1}^{\top}\Big(\underset{a \in \mathcal{A}}{\displaystyle \sum}\pi(a) a a^{\top} \Big)^+e_{d+1}\right\} \geq \kappa_*$
\end{tabular}
\right\}
\end{equation*}
is the set of actions sets such that the minimal variance of the bias estimator is larger than $\kappa_*$.
\end{defi}
In the following theorem, we establish a lower bound on the regret valid for all $\kappa_*\geq 1$ by designing $\kappa_*$-correlated sets of actions $\cA \in \mathbf{A}_{\kappa_*, d}$, and obtaining lower bounds on the regret of any algorithm on these sets of actions.

\begin{thm}\label{thm:lower_bound_worst_case}
 Let $\kappa_* \geq 1$, $d\geq 2$ and $T\geq 4^3\kappa_* $. There exists an action set $\mathcal{A} \in \mathbf{A}_{\kappa_* , d}$ such that for any algorithm, there exists a bandit problem with parameter $\theta_T \in \mathbb{R}^{d+1}$ such that the regret of this algorithm on the problem characterized by $\theta_T$ satisfies $\displaystyle
    R_T^{\theta_T} \geq \nicefrac{\kappa_* ^{1/3}T^{2/3}}{8e}.$
\end{thm}

Previous lower bounds on the regret of linear bandits with partial monitoring, established in \cite{PartialFeedbackCOLT2020}, state that the regret must be at least $c_{\mathcal{A}}T^{2/3}$ for some parameter $\theta_T \in \mathbb{R}^{d+1}$, where $c_{\mathcal{A}}>0$ is a constant depending (not explicitly) on $\mathcal{A}$. By contrast, Theorem \ref{thm:lower_bound_worst_case} provides an explicit characterization of the dependence of the regret rate on the geometry of the problem, which matches the upper bound of Theorem \ref{thm:upper_bound_worst_case} up to a sub-logarithmic factor. Note that the assumption $d \geq 2$ is necessary here: if $d = 1$, there are at most two potentially optimal actions (namely, $\max\{x : x \in \cX\}$ and $\min\{x : x \in \cX\}$). Then, the problem becomes locally observable, and regret of order $\widetilde{O}(T^{1/2})$ can be achieved \cite{PartialFeedbackCOLT2020}.

\subsection{Lower bound on the gap-dependent regret}\label{subsec:lower_bounds_delta}
We now present a lower bound on the gap-dependent regret. More precisely, for given values of $\Delta_{\min}$ and $\Delta_{\neq}$, we establish a lower bound on the worst case regret among parameters $\theta$ verifying $\Delta_{\min} \leq \min_{x \in \mathcal{X}\setminus x^*}\Delta_{x}$, and $\Delta_{\neq} \leq \min_{x \in \mathcal{X} : z_x = -z_{x^*}}\Delta_{x}$. Before stating formally the result, let us define the corresponding parameter set.  For an action set $\mathcal{A} \in \mathbf{A}_{\kappa_*,d}$, and for $\left(\Delta_{\min}, \Delta_{\neq}\right) \in (0,1)^2$ such that $\Delta_{\min}\leq \Delta_{\neq}$, we denote
\begin{equation*}
\mathbf{\Theta}^{\mathcal{A}}_{\Delta_{\min}, \Delta_{\neq}} = \left\{
\begin{tabular}{l}
    $\theta = \left({\gamma  \atop \omega }\right) :\  \gamma  \in \cC(\cX),\ \exists\ !\ \left({x^* \atop z_{x^*}}\right) \in \argmax_{\left({x \atop z_x}\right) \in \mathcal{A}}\{x^{\top}\gamma \},$\\
    $\forall \left({x' \atop z_{x'}}\right) \in \mathcal{A}\text{ such that } x' \neq x^*,  \left(x^* - x'\right)^{\top}\gamma  \geq \Delta_{\min},$\\
    $\forall \left({x' \atop z_{x'}}\right) \in \mathcal{A}\text{ such that } z_{x'} \neq z_{x^*},  \left(x^* - x'\right)^{\top}\gamma  \geq \Delta_{\neq}$
\end{tabular}
\right\}
\end{equation*}
the set of parameters with minimum gap $\Delta_{\min}$, and minimum between-group-gap $\Delta_{\neq}$. 

The upper bounds established in Theorem \ref{thm:upper_bound_delta} underline the dependence of the gap-dependent regret on the minimal regret $\kappa(\Delta)$ for estimating the bias. Before stating our results, we define a class of problems $\mathbf{\Theta}^{\mathcal{A}}_{\Delta_{\min}, \Delta_{\neq}, \kappa}$ such that $\kappa(\Delta)\leq \kappa$. For a parameter $\gamma \in \cC(\cX)$, let us denote $\Delta(\gamma)_x = \max_{x' \in \cX}(x'-x)^{\top}\gamma$, and $\Delta(\gamma) = \left(\Delta(\gamma)_x\right)_{x \in \cX}$. Moreover, for a given set $\cA$, let us denote 
\begin{equation*}
\mathbf{\Theta}^{\mathcal{A}}_{\Delta_{\min}, \Delta_{\neq}, \kappa} = \mathbf{\Theta}^{\mathcal{A}}_{\Delta_{\min}, \Delta_{\neq}} \cap \left\{\theta = \left({\gamma \atop \omega}\right):\ \gamma \in \cC(\cX),\ \kappa(\Delta(\gamma))\leq \kappa\right\}.
\end{equation*}

\begin{thm}\label{thm:lower_bound_gap}
 For all $\kappa \geq 2$ and all $d \geq 4$, there exists a set of actions $\cA\in \mathbb{R}^{d+1}$ such that for all $(\Delta_{\min},\Delta_{\neq})  \in (0,\nicefrac{1}{8})^2$ with $\Delta_{\min}\leq\Delta_{\neq}$,
 \begin{eqnarray}\label{eq:lb_delta_gap}
 \underset{T \rightarrow\infty}{\liminf} \underset{\theta \in \mathbf{\Theta}^{\mathcal{A}}_{\Delta_{\min}, \Delta_{\neq}, \kappa}}{\sup} \frac{ R_T^{ \theta}}{\log\left(T\right)} \geq \left[\frac{d}{10\Delta_{\min}}\right] \lor \left[
\frac{\kappa+2}{8\Delta_{\neq}^2} \right].
\end{eqnarray}
\end{thm}

Theorem \ref{thm:lower_bound_gap} shows that for some action sets $\cA$, the gap-depend regret of the \textsc{Fair Phased Elimination} algorithm is asymptotically optimal up to a numerical constant. Note that the assumption $d\geq 4$ is necessary in our proof to design an action set $\mathcal{A}$ such that Equation \eqref{eq:lb_delta_gap} holds for all $\Delta_{\min}, \Delta_{\neq} \in (0,\nicefrac{1}{8})$. On the other hand,  as discussed in Appendix \ref{app:lower_bound_gap_2}, for $d \geq 2$, for all $\Delta_{\min}, \Delta_{\neq} \in (0,1/8)$, we can show that there exists action sets $\mathcal{A}$ and $\theta \in \mathbf{\Theta}^{\mathcal{A}}_{\Delta_{\min}, \Delta_{\neq}}$ such that the lower bound in Equation \eqref{eq:lb_delta_gap} still holds, by considering separately the cases $\nicefrac{d}{\Delta_{\min}} >\nicefrac{\kappa}{\Delta_{\neq}^2}$ and $\nicefrac{d}{\Delta_{\min}} \leq\nicefrac{\kappa}{\Delta_{\neq}^2}$.



\section{Conclusion}
In this paper, we addressed the problem of online decision making under biased bandit feedback. We designed a new algorithm based on $\Delta$- and G-optimal design, and obtained worst-case and gap-dependent upper bounds on its regret. We obtained lower bounds on the regret for some problem instances showing that these rates are tight up to sub-logarithmic factors in some settings. These rates highlight two behaviors: on the one hand, the worst case rate  $\mathcal{O}(\kappa_{*}^{1/3}\log(T)^{1/3}T^{2/3})$ highlights the cost induced by the biased feedback, and the need to select sub-optimal actions in order to debias it. On the other hand, the gap-dependent bound shows that for some instance, the problem can be locally observable around the best action: then, the difficulty of the problem is dominated by the difficulty of the corresponding linear bandit problem, and is no more difficult than this problem. When this is not the case, the regret scales as $\kappa(\Delta)\Delta_{\neq}^{-2}\log(T)$, where $\Delta_{\neq}$ is the gap between the best actions of the two groups, and $\kappa(\Delta)$ is the minimum regret for estimating the bias with a given precision. This work paves the way for studying other bandit models with unfair feedback, considering for example continuous, multi-dimensional sensitive attributes.

\bibliographystyle{abbrv}
\bibliography{refFairBandit}
\newpage
\appendix

\section*{Appendix}

The Appendix is organized as follows. We begin  in Section \ref{sec:kappa} by further discussing the interpretation and computation of $\kappa_*$ and $\kappa(\Delta)$, and their relation to the worst-case alignment constant of \cite{PartialFeedbackCOLT2020} and to the problem of optimal estimation of the bias against the worst parameter. Then, we provide in Section \ref{sec:detailed_algo} a detailed version of the \textsc{Fair Phased Elimination} algorithm \ref{alg:FPE:short}. Then, in Section \ref{sec:proofs}, we prove the main results of this paper.

\section{On the geometry of bias estimation} \label{sec:kappa}
The constants $\kappa_*$ and $\kappa(\Delta)$ respectively characterize the difficulty of the worst-case problem, and of the gap-depend problem, and highlight the dependence of the regret on the geometry of the action set. In this section, we begin by discussing  in Section \ref{subsec:c-opt} the interpretation of the constant $\kappa_*$ as the variance of the $e_{d+1}$-optimal design. Using Elfving's characterization of the $e_{d+1}$-optimal design, we then derive an alternative characterization of $\kappa_*$ in terms of separation of the actions of the two groups in Section \ref{subsec:upsilon}

\subsection{Bias estimation as a $\mathbf{e_{d+1}}$-optimal design problem} 
\label{subsec:c-opt} 
Recall that $\kappa_*$ is the minimal variance of the bias estimator related to the problem of $e_{d+1}$-optimal design.

\paragraph{$\mathbf{e_{d+1}}$-optimal design} Optimal design theory addresses the following problem: a scientist must design a set of $n$ experiments $\{x_1, ..., x_n\}\in\mathcal{X}^n$ so as to estimate at best a parameter of interest, where each experiment $x \in \cX$ corresponds to a point $a_x \in \mathbb{R}^{d+1}$. The aim of the scientist is to choose a design, i.e. a function $\mu :\cX \mapsto \mathbb{N}$ indicating the budget $\mu(x)$ to be allocated to each experiment $x \in \cX$. Each experiment $x$ is then repeated exactly $\mu(x)$ times, and the corresponding observations $y_{x,1}, ..., y_{x,\mu(x)}$ are collected for each $x\in\cX$. The law of the observations corresponding to experiment $x$ at point $a_x$ is given by $$y_{x,i} = a_{x}^{\top}\theta^* + \xi_{x,i},$$
where $\xi_{x,i} \sim \mathcal{N}(0,1)$ are independent noise terms, and $\theta^*\in \mathbb{R}^{d + 1}$ is an unknown parameter. The aim of the scientist is to choose the design $\mu$ so as to best estimate (some features of) the parameter $\theta^*$, under a constraint on the total number of experiments $\sum_{x \in \mathcal{X}}\mu(x) \leq n$ for some $n \in \mathbb{N}$.

Different criteria can be used to characterize the optimality of a design $\mu$.  For example, one may need to estimate the full parameter $\theta^*$, in order to predict the outcomes of the experiments $x \in \cX$ with a small uniform error: this leads to the G-optimal design problem (\ref{eq:G-opt-design}). Alternatively, for $c$ a vector in $\mathbb{R}^{d+1}$, one may aim at finding the best design $\mu \in \cN^{\cX}$ for estimating the scalar product $c^{\top}\theta^*$ under a budget constraint $\underset{x \in \mathcal{X}}{\sum}\mu(x) \leq n$, where $\cN^{\cX} = \left\{\mu : \cX \rightarrow \mathbb{N}\right\}$. This problem is known as \textit{$c$-optimal design}. Unbiased linear estimation of $c^{\top}\theta^*$ is possible only when $c$ belongs to the image of $V(\mu)$, and in this case the  best linear unbiased estimator  of the scalar product $c^{\top}\theta^*$ is given by $c^\top \widehat\theta$, where $\widehat\theta$ is the least-square estimator defined as
\begin{equation}
    \widehat{\theta} = V(\mu)^{+}\sum_{x\in \mathcal{X}}a_x \left(\sum_{i \leq \mu(x)}y_{x,i}\right) \ \ \ \text{ for }  \ \ \  V(\mu) = \sum_{x \in \mathcal{X}}\mu(x) a_x a_x^{\top}.
\end{equation}
The variance of the estimator $c^\top \widehat\theta$ is then equal to $c^{\top}V(\mu)^{+}c$.

Exact $c$-optimal design aims at choosing the allocation $\mu \in \cN^{\cX}$ minimizing the variance of $c^\top \widehat\theta$ for a given budget $\sum_x \mu(x) \leq n$, under the constraint that $c\in\Image(V(\mu))$. Let us define  the normalized design $\pi: x \in \cX \mapsto \mu(x)/n$, and let us underline that $\pi$ defines a probability on $\cX$. The variance of $c^{\top}\widehat{\theta}$ is then equal to $n^{-1}c^{\top}V(\pi)^{+}c$. In the limit $n \rightarrow + \infty$, the problem is equivalent to the problem of approximate $c$-optimal design (sometimes simply referred to as $c$-optimal design), that aims at finding a probability measure $\pi \in \mathcal{P}_c^{\cX}:=\{\pi\in\mathcal{P}
^{\cX}:c\in\Image(V(\pi))\}$ solution to the following problem
\begin{equation*}
    \underset{\pi \in \mathcal{P}_c^{\cX}}{\min}\ c^{\top}V(\pi)^{+}c\,. \ \ \ \ \ \ \  \ \ \ \ \ \ \ \text{($c$-optimal design)}
\end{equation*}

Note that when $\{a_x : x \in \cX\}$ spans $\mathbb{R}^{d+1}$, for any $c \in \mathbb{R}^{d+1}$, there exists a design $\pi$ such that $c\in\Image(V(\pi))$, and hence the $c$-optimal design problem admits a solution.

\paragraph{Computation of the $\mathbf{e_{d+1}}$-optimal design}
 Finding an exact optimal allocation $\mu \in \cN^{\cX}$ under the constraint that $\sum_{x \in \mathcal{X}}\mu(x) \leq n$ is unfortunately NP-complete. However, finding an approximate optimal design $\pi \in \mathcal{P}^{\mathcal{X}}_{c}$ can be done in polynomial time \cite{Complexity-c-Optimal}. Several algorithms, including multiplicative algorithms \cite{Fellman} and a simplex method of linear programming  \cite{Simplex-c-Optimal}, have been proposed to iteratively approximate the optimal design. More recently, \cite{pronzato:removing} suggested using screening tests to remove inessential points to accelerate optimization algorithms.
 
Classical results from $e_{d+1}$-optimal design show that there exists a $c$-optimal design supported by at most $d+1$ points (see, e.g., \cite{Pazman,Simplex-c-Optimal} for a proof of this result). The following Lemma indicates how to obtain an exact design by rounding an approximate design supported by at most $d+1$ points.

\begin{lem}\label{lem:c-opt}
For any $\pi\in \cM^{\cX}_{e_{d+1}}$ and any $m>0$, the estimator $e_{d+1}^{\top}\widehat\theta_{\mu}$ computed from the design $\mu: x \mapsto \ceil{m\pi(x)}$ is an unbiased estimator of $e_{d+1}^\top\theta$ and it has a variance at most $m^{-1} e_{d+1}^\top V(\pi)^{+}e_{d+1}$.
\end{lem}

Obviously, similar results also hold for G-optimal design.

\begin{lem}\label{lem:G-opt}
Let $\pi$ be a solution of the G-optimal design problem \eqref{eq:G-opt-design}. Then, for any $m>0$ and any $x\in \mathcal{X}$, the estimator $a_x^{\top}\widehat\theta_{\mu}$ computed from the design $\mu: x \mapsto \ceil{m\pi(x)}$ is an unbiased estimator of the evaluation $a_x^\top\theta$, and it has a variance 
$$a_x^{\top}V(\mu)^+a_x \leq m^{-1}(d+1).$$
\end{lem}

\subsection{Interpretation of $\kappa_*$ in terms of separation of the groups}
\label{subsec:upsilon}
Next theorem, due to Elfving, characterizes solutions to the $c$-optimal design problem.
\begin{thm}[\cite{Elfving}]\label{thm:Elfving}
Let $\mathcal{S} = \text{convex hull}\left\{ +a_x, -a_x : x \in \mathcal{X} \right\}$ be the Elfving's set of $\{a_x : x\in \cX \}\subset \mathbb{R}^{d+1}$, and let $\partial \mathcal{S}$ denote the boundary of $\mathcal{S}$. A design $\pi \in \mathcal{P}_c^{\cX}$ is $c$-optimal for $c \in \mathbb{R}^{d+1}$ if and only if there exists $\zeta \in \{-1,+1\}^{\mathcal{X}}$ and $t>0$ such that 
 \begin{equation*}
tc = \underset{x\in\mathcal{X}}{\sum}\pi(x)\zeta_x a_x \in \partial \mathcal{S}.
\end{equation*}
Moreover, $t^{-2} = c^{\top}\left(V(\pi)\right)^{+}c$ is value of the $c$-optimal design problem.
\end{thm}
Elfving's characterization of the $e_{d+1}$-optimal design allows us to derive the following equivalent characterization of $\kappa_*$.

\begin{lem}\label{lem:upsilon}  $\displaystyle \kappa_* = \max_{u \in \mathbb{R}^{d}} \frac{1}{\max_{x \in \cX} \left(x^{\top}u + z_x\right)^2}$.
\end{lem}

Lemma \ref{lem:upsilon_margin} follows from the characterization in Lemma \ref{lem:upsilon}. When $\kappa_*>1$, the vector $\tilde{u}$ defined as

$\tilde{u} = \argmax_{u \in \mathbb{R}^{d}} \frac{1}{\max_{x \in \cX} \left(x^{\top}u + z_x\right)^2}$ is a normal vector of the separating hyperplane $\mathcal{H}$ in Figure \ref{fig:margin}. Moreover, as shown in the proof of Lemma \ref{lem:upsilon_margin}, the margin is in this case equal to $1-\kappa_*^{-1/2}$, while the maximum distance of all points to the hyperplane is $1+\kappa_*^{-1/2}$.

Lemma \ref{lem:upsilon} also allows us to compare the bound in Theorem \ref{thm:upper_bound_worst_case} with previous results on linear bandit with partial monitoring, expressed in terms of the worst-case alignment constant. 

\subsection{Comparison to the worst-case alignment constant}

Previous work on linear bandit with partial linear monitoring measures the difficulty of the bandit game using the \textit{worst-case alignment constant} $\alpha$, defined as
$$\alpha = \max_{u \in \mathbb{R}^d} \frac{\max_{x, x' \in \cX}((x - x')^{\top}u)^2}{\max_{x \in \cX}(z_x  x^{\top}u+1)^2}.$$
The following Lemma shows that this constant is essentially equivalent to the minimal variance of the bias estimator $\kappa_*$. 

\begin{lem}\label{lem:upsilon_alpha}
$\frac{\kappa_*}{3} \leq \alpha \leq 16\kappa_*$.
\end{lem}

On the one hand, Lemma \ref{lem:upsilon_alpha} shows that $\kappa_*$ and $\alpha$ are essentially equivalent. In particular, Theorem \ref{thm:lower_bound_worst_case} implies that the large $T$ regret is of order $\alpha^{1/3}\log(T)^{1/3}T^{2/3}$. This improves over previous known rates, obtained in \cite{PartialFeedbackCOLT2020}, by a factor $d^{1/2}\log(T)^{1/6}(\log(kT)/\log(T))^{1/2}$.

On the other hand, as underlined, the constant $\kappa_*$ appears when considering the well-studied problem of $c$-optimal design. Therefore, classical results and algorithms for optimal design can be used to characterize and compute this constant.

\subsection{Optimal bias estimation against the worst parameter}
The constant $\kappa_*$ also appears naturally when considering the related problem of optimal bias estimation against the worst parameter.
\paragraph{Regret of $e_{d+1}$-optimal design}
Recall that $\kappa_*$ denotes the \textit{minimal variance of the bias estimator}, i.e. the value of the solution of the $e_{d+1}$-optimal design problem
\begin{equation*}
\kappa_*=\underset{\pi \in \mathcal{P}^{\mathcal{X}}_{e_{d+1}}}{\min}\ e_{d+1}^{\top}\left(V(\pi)\right)^{+}e_{d+1} \,,
\end{equation*}
The $e_{d+1}$-optimal design can be equivalently defined as the solution of the problem 
\begin{align}\label{eq:kappa_equivalent-1}
\text{minimize } & \underset{x \in \cX}{\sum}\mu(x)\quad
\text{such that }  \mu \in \cM^{\cX}_{e_{d+1}}\ 
 \text{and }  \ e_{d+1}^{\top} V(\mu)^+e_{d+1} \leq  \kappa_*.
\end{align} The characterization given in Equation \eqref{eq:kappa_equivalent-1} underlines that the $e_{d+1}$-optimal design provides (up to discretization issues) the minimal number of samples required for estimating $\omega^*$ with a variance $\kappa_*$. Let us denote by $\mu^*$ the optimal design for estimating $\omega^*$ with a variance 1, defined as 
\begin{eqnarray*}\label{eq:mu*}
\mu^*& = & \underset{\mu}{\argmin}  \underset{x \in \cX}{\sum}\mu(x)\quad
\text{such that }  \mu \in \cM^{\cX}_{e_{d+1}}\ 
 \text{and }  \ e_{d+1}^{\top} V(\mu)^+e_{d+1} \leq  1.
\end{eqnarray*}
Note that from the definition of $\kappa_*$, we have $\sum_x \mu^*(x) = \kappa_*$. 

A first (naive) approach to obtain an estimate of the bias parameter $\omega^*$ with precision level $\epsilon>0$ would consist in sampling actions according to $\epsilon^{-2}\mu^*$, rounded according to the procedure defined in Lemma \ref{lem:c-opt}. Let us denote by $\Delta_{x}$ the gap $\Delta_{x}=\max_{x'\in\cX} (x'-x)^{\top}\gamma^*$ between the (non-observed) reward of the best action and the reward of the action $x$. The regret corresponding to this estimation phase would then be 
$$\epsilon^{-2}\underset{x \in \cX}{\sum}\mu^{*}(x)\Delta_x,$$
which can be as large as $\kappa_*\epsilon^{-2}\max_{x}\Delta_x$. Interestingly, we show that the regret corresponding to the $e_{d+1}$-optimal design is equivalent (up to a small multiplicative constant) to the minimax regret.

\paragraph{Optimal worst-case estimation}
The minimax regret corresponds to the regret of the best sampling scheme against the worst admissible parameter $\gamma$. Note that, for a given design $\mu$, this worst-case regret is given by $$ \max_{x' \in \cX, \gamma \in \cC(\cX)}\sum_x \mu(x)(x' - x)^{\top}\gamma,$$ where we recall that $\cC(\cX) = \left\{\gamma \in \mathbb{R}^d : \forall x \in \cX, \vert x^{\top}\gamma \vert\leq 1\right\}$ is the set of admissible parameters. To achieve the lowest regret against the worst parameter, we must use the minimax optimal design $\widetilde{\mu}$ solution to the problem 
\begin{eqnarray*}
\widetilde{\mu} = \underset{\mu}{\argmin} \underset{x' \in \cX, \gamma \in \cC(\cX)}{\max} \underset{x \in \cX}{\sum}\mu(x)(x' - x)^{\top}\gamma \ 
\text{such that }  \mu \in \cM^{\cX}_{e_{d+1}}\ 
 \text{and }  \ e_{d+1}^{\top} V(\mu)^+e_{d+1} \leq  1.
\end{eqnarray*}

Lemma \ref{lem:kappa_tilde} underlines that the regret corresponding to the $e_{d+1}$-optimal design is no larger than twice the minimax regret.

\subsection{On the $\Delta$-optimal design} \label{subsec:Delta_des}
Recall that for a vector of gaps $\Delta = \left(\Delta_x\right)_{x \in \cX}$, $\mu^{\Delta}$  denotes the $\Delta$-optimal design, defined as the solution of the following problem 

\begin{eqnarray*}
\mu^{\Delta} &=& \underset{\mu}{\argmin} \underset{x \in \cX}{\sum}\mu(x)\Delta_x\quad
\text{such that }  \mu \in \cM^{\cX}_{e_{d+1}}\ 
 \text{and }  \ e_{d+1}^{\top} V(\mu)^+e_{d+1} \leq  1. \quad (\Delta\text{-optimal design})
\end{eqnarray*}

If we knew the gaps $\Delta_x$, we could sample the actions according to the $\Delta$-optimal design $\mu^{\Delta}$, and pay the regret $\epsilon^{-2}\kappa(\Delta)$ (up to rounding error) for estimating $\omega^*$ with an error smaller than $\epsilon$, where 
$$\kappa(\Delta) = \underset{x \in \cX}{\sum}\mu^{\Delta}(x)\Delta_x.$$ 

\begin{lem}\label{lem:kappa2} If $\gamma^* \in \cC(\cX)$, then $\kappa(\Delta)\leq 2\kappa_*$
\end{lem}
\begin{proof}
Be definition of $\cC(\cX)$, for all $\gamma^* \in \cC(X)$, all $x, x'\in \cX$, we have
$$(x-x')^{\top}\gamma^* \leq \vert x^{\top}\gamma^* \vert + \vert x'^{\top}\gamma^* \vert \leq 2.$$
Then, 
\begin{eqnarray*}
\kappa(\Delta) &\leq& 2\,\underset{\mu}{\min} \underset{x \in \cX}{\sum}\mu(x)\quad
\text{such that }  \mu \in \cM^{\cX}_{e_{d+1}}\ 
 \text{and }  \ e_{d+1}^{\top} V(\mu)^+e_{d+1} \leq  1.
\end{eqnarray*}
Let $\mu_*$ be the solution of the $e_{d+1}$-optimal design problem 
\begin{eqnarray*}
\underset{\mu}{\minimize }\  e_{d+1}^{\top} V(\mu)^+ e_{d+1} \ \text{such that }  \mu \in \cP^{\cX}_{e_{d+1}}.
\end{eqnarray*}
By definition of $\kappa_*$, we see that $e_{d+1}^{\top} V(\mu_*)^+e_{d+1} = \kappa_*$. This implies that the measure $\kappa_* \times \mu_*$ verifies the constraints $e_{d+1}^{\top} V(\kappa_* \times \mu_*)^+e_{d+1} \leq 1$ and $\kappa_*\mu_* \in \cM^{\cX}_{e_{d+1}}$. Thus,
\begin{eqnarray*}
\kappa(\Delta) &\leq& 2\underset{x \in \cX}{\sum}\kappa_*\mu_*(x) = 2\kappa_*.
\end{eqnarray*}
\end{proof}

\paragraph{On the regret $\kappa(\Delta)$} The function $\kappa$ verifies the following properties.

\begin{lem}\label{lem:prop_kappa}
For two vectors of gaps $\Delta$, $\Delta'$, denote by $\Delta \land \Delta'$ (respectively $\Delta \lor \Delta'$) the vector of gaps given by  $\left(\Delta \land \Delta'\right)_x = \Delta_x \land \Delta'_x$ (respectively  $\left(\Delta \lor \Delta'\right)_x = \Delta_x \lor \Delta'_x$) for all $x \in \cX$. Moreover, denote $\Delta \leq \Delta'$ if $\Delta_x \leq \Delta'_x$ for all $x \in \cX$. Then, the following properties hold :
\begin{enumerate}[label=\roman*)]
    \item for all $c>0$, $\kappa(c\Delta) = c\kappa(\Delta)$;\label{eq:prop_kappa_i}
    \item if $\Delta \leq \Delta'$, then $\kappa(\Delta) \leq \kappa(\Delta')$; \label{eq:prop_kappa_ii}
    \item $\kappa(\Delta \lor \Delta') \geq \kappa(\Delta) \lor \kappa(\Delta');$\label{eq:prop_kappa_iii}
    \item the function $\epsilon  \mapsto \kappa(\Delta \vee  \epsilon)$ is continuous at $0$.\label{eq:prop_kappa_iv}
\end{enumerate}
\end{lem}

\paragraph{Computation of the $\Delta$-optimal design}
In practice, the $\Delta$-optimal design can be computed by adapting algorithms designed for finding the $e_{d+1}$-optimal design. Indeed, the next lemma shows that the computation of the $\Delta$-optimal design amounts to computing an $e_{d+1}$-optimal design for some  rescaled features.

\begin{lem}\label{lem:calcul:kappa}
For any vector $\Delta\in (0,+\infty)^{\cX}$, let $\pi^{\Delta}$ be the $e_{d+1}$-optimal design relative to the set
$\mathcal{A}^{\Delta}=\left\{\Delta_{x}^{-1/2}\left({x \atop z_x}\right): x\in \mathcal{X}  \right\}$ and let $\kappa^\Delta=e_{d+1}^\top V(\pi^{\Delta})^{+}e_{d+1}$ be the $e_{d+1}$-optimal variance relative to $\cA^{\Delta}$.
Then, the $\Delta$-optimal design $\mu^{\Delta}$ is given by
$\mu^{\Delta}(x)= \kappa^{\Delta}\pi^{\Delta}(x)\Delta_{x}^{-1}$ for all $x\in\cX$. In addition, the support of $\mu^{\Delta}$ can be chosen to be of cardinality at most $d+1$.
\end{lem}


\section{Detailed Fair Phased Elimination  algorithm}\label{sec:detailed_algo}

We present the notations used in Algorithm \ref{alg:FPE}. The phases are indexed by $l\in \mathbb{N}^*$. The sets $\cX_l^{(z)}$ for $z\in \{-1, +1\}$ corresponds to actions in group $z$ that are considered as potentially optimal in phase $l$. The variable $\widehat{z^*_l}$ encodes the group determined as optimal: it is $0$ as long as this group has not been determined. The subscript $(z)$ refer to the group $z$ when $z\in \{-1, +1\}$, and otherwise to the estimation of the bias $\omega^*$: for example, the  probability $\pi^{(z)}_l$ for $z\in \{-1, +1\}$ and $l>1$ corresponds to the approximate G-optimal design on $\cX_l^{(z)}$. Then, for $z\in \{-1, +1\}$, allocations $\mu^{(z)}$ (resp. $\mu^{(0)}$) correspond to allocation of samples in the exploration phase Exp$_l^{(z)}$ (resp. Exp$_l^{(0)}$). Similarly, $V^{(z)}_l$ (resp $V^{(0)}_l$) denotes the variance matrix of the estimator $\left({\hgamma_l^{(z)} \atop \homega_l^{(z)}}\right)$ (resp. $\homega_l^{(0)}$) obtained from observations made during phase Exp$_l^{(z)}$ (resp. Exp$_l^{(0)}$). Finally, $\text{Explore}_l^{(z)}$ (resp. $\text{Explore}_l^{(0)}$) is a Boolean variable indicating whether the exploration  at phase $l$ for group $z$ (resp. for the bias parameter) has been performed. It is used in the proofs to ensure that the corresponding estimators are well defined.

\begin{algorithm}[H] 
\caption{Fair Phased Elimination (detailed version)}\label{alg:FPE}
\begin{algorithmic}[1]
\State {\bf Input:} $\delta$, $T$, $k=|\mathcal{X}|$
\State {\bf Initialize:} Recovery $ \gets \emptyset$, $t \gets 0$, $l\gets 1$ $\widehat{z^*_1} \gets 0$,\\
 \ \ \ \ \ \ \ \ \ \ \ \ \ \   $\mathcal{X}_1^{(+1)} \gets \{x : z_x = 1\}$, $\mathcal{X}_1^{(-1)} \gets \{x : z_x = -1\}$, $\widehat{\Delta}_{x}^1\gets 2$ for $x\in\cX$
\While{ $t<T$}
    \State {\bf Initialize: }$\epsilon_l \gets 2^{2-l}$, \ $\widehat{z^*}_{l+1} \gets \widehat{z^*}_l$, \ $\widehat \Delta^{l+1}\gets \widehat \Delta^{l}$, $\text{Explore}_l^{(z)} \gets \text{False}$ for $z \in \{-1,0,+1\}$
    \For{$z \in \{-1, +1\}$ such that $z \neq -\widehat{z^*_l}$} \Comment{G-optimal Exploration and Elimination}
    	\State$\pi^{(z)}_{l} \gets \underset{\pi}{\argmin} \left\{\underset{x \in \mathcal{X}_l^{(z)}}{\max} a_{x}^{\top}V(\pi)^{+}a_{x} : \pi \in  \mathcal{P}^{\mathcal{X}_l^{(z)}}_{\mathcal{X}_l^{(z)}},\  |\supp(\pi)| \leq \frac{(d+1)(d+2)}{2}\right\}$
        \State $\mu^{(z)}_{l}(x) \gets \ceil{ \frac{2(d+1)\pi^{(z)}_{l}(x)}{\epsilon_l^2} \log\left(\frac{kl(l+1)}{ \delta}\right)}$ for all $x \in \mathcal{X}_l^{(z)}$
        \State $n^{(z)}_l \gets  \underset{{x \in \mathcal{X}_l^{(z)}}}{\sum} \mu^{(z)}_{l}(x)$, $\text{Exp}_l^{(z)} \gets \left\{t+1, ...,T \land (t+ n^{(z)}_l )\right\}$
        \If{$t + n^{(z)}_l \leq T$}{}
        		\State $\text{Explore}_l^{(z)} \gets \text{True}$, choose each action $x \in \mathcal{X}_l^{(z)}$ exactly $\mu^{(z)}_{l}(x)$ times
        		\State $V_l^{(z)} \gets \sum_{t \in \text{Exp}_l^{(z)}}a_{x_{t}}a_{x_{t}}^{\top}$, \ \ $\htheta_{l}^{(z)}
		\gets \left(V_l^{(z)}\right)^{+}\sum_{t \in \text{Exp}_l^{(z)}} y_t a_{x_{t}} $
        		\State $\mathcal{X}_{l+1}^{(z)} \gets \left\{x \in \mathcal{X}_{l}^{(z)} : \max_{x' \in\mathcal{X}_{l}^{(z)}}\left(a_{x'} - a_{x}\right)^{\top}\htheta_l^{(z)} \leq 3\epsilon_l \right\}$
        \Else{ for $t \in \text{Exp}_l^{(z)}$, sample empirical best action in $\mathcal{X}_l^{(z)}$} 
    	\EndIf
    	\State $t \gets t+ n^{(z)}_l$
    \EndFor
    \If{$\widehat{z^*_{l}} = 0$} 
        \State compute the $\widehat{\Delta}^l$-optimal design $\widehat{\mu}_{l}$ and the corresponding regret $\kappa(\widehat{\Delta}^l)$
	    \If{$\epsilon_{l}\leq \pa{\kappa(\widehat\Delta^{l}) \log(T)/T}^{1/3}$} \Comment{Recovery phase}
        	\State Recovery $ \gets \{t, ..., T\}$
        	\State sample empirical best action  in $\cX_{l+1}^{(-1)}\cup\cX_{l+1}^{(1)}$ until the end of the budget, $t \gets T$
	    \Else \Comment{$\widehat{\Delta}^l$-optimal Exploration and Elimination}
	    \State $\mu_{l}^{(0)}(x) \gets \ceil{ \frac{2\hat\mu_{l}(x)}{\epsilon_l^2} \log\left(\frac{l(l+1)}{ \delta}\right)}$ for all $ x \in \mathcal{X}$
	    \State $n^{(0)}_l \gets  \underset{{x \in \mathcal{X}}}{\sum} \mu^{(0)}_{l}(x)$,  $\text{Exp}_l^{(0)} \gets \left\{t, ..., T\land (t + n_l^{(0)})\right\}$
        \If{$t + n_l^{(0)} \leq T$}
            \State $\text{Explore}_l^{(0)} \gets \text{True}$,  choose each action $x \in \mathcal{X}$ exactly $\mu^{(0)}_{l}(x)$ times
            \State $V_l^{(0)} \gets \sum_{t\in \text{Exp}_l^{(0)}} a_{x_{t}}a_{x_{t}}^{\top}$, \ \ $\homega_l^{(0)} \gets e_{d+1}^{\top}\left(V_l^{(0)}\right)^{+}\sum_{t\in \text{Exp}_l^{(0)}}y_t a_{x_{t}}$
            \For{$x\in \cX_{l+1}^{(-1)}\cup\cX_{l+1}^{(1)}$} 
            	\State $\widehat m_{l,x} \gets a_{x}^\top \htheta_{l}^{(z_{x})}-z_{x}\widehat \omega^{(0)}_{l}$
            	\State $\widehat \Delta^{l+1}_{x}\gets \pa{\max_{x'\in\cX_{l+1}^{(-1)}\cup\cX_{l+1}^{(1)}}  \widehat m_{l,x'} -\widehat m_{l,x}+4\epsilon_{l}}\wedge 2$
            \EndFor
            \For{ $z \in \{-1,+1\}$}
                \If{$\underset{x \in\mathcal{X}_{l+1}^{(z)}}{\max} \widehat m_{l,x}- 2 \epsilon_{l}\geq \underset{x \in\mathcal{X}_{l+1}^{(-z)}}{\max} \widehat m_{l,x}+ 2\epsilon_l$}
                {$\widehat{z^*}_{l+1} \gets z$}
                \EndIf
            \EndFor
            \Else \ sample empirical best action in $\cX_{l+1}^{(-1)}\cup\cX_{l+1}^{(1)}$ until the end of the budget, $t \gets T$
        \EndIf
        \State $t \gets t+ n^{(0)}_l$
        \EndIf
    \EndIf
    \State $l\gets l+1$
\EndWhile
\end{algorithmic}
\end{algorithm}


\section{Proofs}\label{sec:proofs}
For an event $\mathcal{F}$ such that $\mathbb{P}\left(\mathcal{F}\right)>0$, we denote by $\mathbb{E}_{\vert \mathcal{F}}$ (resp. $\mathbb{P}_{\vert \mathcal{F}}$) the expectation (resp. the probability) conditionally on $\mathcal{F}$.  

\subsection{Proof of Theorem \ref{thm:upper_bound_worst_case}}\label{app:proof_upper_bound_worst_case}

We begin by defining for $z \in \{-1,0,+1\}$
\begin{align*}
L^{(z)}&=\max\ac{l\geq 1: \text{Explore}_l^{(z)} = \text{True}}\nonumber
\end{align*}
the largest integer $l$ such that Explore$^{(z)}_l=$ True. Recall that $\kappa_*$ is the $e_{d+1}$-optimal variance. By definition of the algorithm, for all $l \leq L^{(0)}+1$, $\widehat{\Delta}^{l} \leq 2$, so  $\kappa(\widehat{\Delta}^{l}) \leq 2\kappa_*$. Now, let us also define
\begin{align*}
L_{T}&=\max\ac{l\geq 1: \epsilon_{l}> \pa{2\kappa_{*} \log(T)\over T}^{1/3}}.\nonumber
\end{align*}
Then, if Recovery$\neq \emptyset$, we must have $L^{(0)} \geq L_T$. Moreover, we see that since $\epsilon_{L_T} = 2^{2-L_T}$, we have  $L_T \leq 2 + \frac{\log_2\left(T/(2\kappa_*\log(T))\right)}{3} \leq 3\log_2\left(T\right)$ when $T >1$.

We define a "bad" event $\mathcal{F}$, such that, on $\overline{\mathcal{F}}$, our estimators $\hgamma_l^{(z)}$ and $\homega_l^{(z)}$ are close to the true parameters $\gamma^*$ and $\omega^*$ for all rounds $l$. More precisely, let
\begin{eqnarray}\label{eq:defEl}
\mathcal{F} &=& \underset{l\geq 1}{\bigcup}\mathcal{F}_l,
\end{eqnarray}
where for $l\geq 1$
\begin{eqnarray*}
    \mathcal{F}_l &=& \left\{\exists z\in \{-1,1\} \text{ such that Explore}_l^{(z)} =\text{True, and } x\in \mathcal{X}_l^{(z)} \text{ such that } \ \left\vert  \left({\hgamma_l^{(z)} - \gamma^* \atop \homega_l^{(z)} - \omega^*}\right)^{\top} \left({x \atop z_x}\right)\right \vert \geq \epsilon_l\right\} \nonumber \\
    && \bigcup \left\{\text{Explore}_l^{(0)} =\text{True and }\left\vert  \homega_l^{(0)}- \omega^*\right \vert \geq \epsilon_l\right\}.\nonumber
\end{eqnarray*}
Then, the regret decomposes as
\begin{equation}\label{eq:regret_dec}
    R_T \leq \underset{t\leq T}{\sum}\mathbb{E}_{\vert \overline{\mathcal{F}}}\left[(x^* - x_t)^{\top}\gamma^*\right] + 2T\mathbb{P}\left[\mathcal{F}\right].
\end{equation}
The following lemma relies on concentration of Gaussian variables to bound the probability of the event $\mathcal{F}$. 
\begin{lem}\label{lem:proba_F}
$\mathbb{P}\left(\mathcal{F}\right) \leq 2 \delta.$
\end{lem}
Now, the first term of \eqref{eq:regret_dec} can be decomposed as
\begin{eqnarray*}
    \underset{t\leq T}{\sum}(x^* - x_t)^{\top}\gamma^* & \leq& \underset{z \in \{-1,0,+1\}}{\sum}\overset{L^{(z)}+1}{\underset{l= 1}{\sum}}\underset{t \in \text{Exp}^{(z)}_l}{\sum}(x^* - x_t)^{\top}\gamma^* + \underset{t \in \text{Recovery}}{\sum}(x^* - x_t)^{\top}\gamma^*,
\end{eqnarray*}
where we use as convention that the sum over an empty set is null. Note that for $z \in \{-1,+1\}$, during the phase $\text{Exp}^{(z)}_l$ the algorithm only samples actions from $\mathcal{X}^{(z)}_l$. By contrast, during the phase $\text{Exp}^{(0)}_l$, even actions eliminated from the sets $\mathcal{X}_l^{(z)}$ can be sampled.  Finally, if the algorithm stops during phase Exp$_{L^{(0)}+1}^{(0)}$, but does not have enough budget to complete the last $\widehat{\Delta}^l$-optimal Exploration and Elimination Phase, it samples the remaining actions in the set $\mathcal{X}_{L^{(0)}+2}^{(-1)}\cup\mathcal{X}_{L^{(0)}+2}^{(+1)}$. Hence, the first term of \eqref{eq:regret_dec} can be upper-bounded by
\begin{eqnarray}\label{eq:rt_dec}
    \underset{t\leq T}{\sum}(x^* - x_t)^{\top}\gamma^* &\leq& \underset{z \in \{-1, +1\}}{\sum}\overset{L_T}{\underset{l= 1}{\sum}}\left(\underset{x \in \mathcal{X}_l^{(z)}}{\sum}\mu_{l}^{(z)}(x)\right)\underset{x \in \mathcal{X}_l^{(z)}}{\max}(x^* - x)^{\top}\gamma^*\\
    && +\underset{z \in \{-1, +1\}}{\sum}\overset{L^{(z)}+1}{\underset{l= L_T+1}{\sum}}\underset{t \in \text{Exp}^{(z)}_l}{\sum}(x^* - x_t)^{\top}\gamma^* + \underset{t \in \text{Recovery}}{\sum}(x^* - x_t)^{\top}\gamma^* \nonumber \\
    && + \overset{L^{(0)}}{\underset{l=1}{\sum}}\underset{x \in \mathcal{X}}{\sum}\mu_{l}^{(0)}(x)\Delta_{x} + \mathds{1}\left\{\text{Explore}_{L^{(0)}+1}^{(0)} = \text{False}\right\}\underset{t \in \text{Exp}^{(0)}_{L^{(0)}+1}}{\sum} \underset{x \in \mathcal{X}_{L^{(0)}+2}^{(-1)}\cup \mathcal{X}_{L^{(0)}+2}^{(+1)}}{\max}(x^* - x)^{\top}\gamma^*\nonumber.
\end{eqnarray}

We begin by bounding the sum of the regret corresponding to the Recovery phase and to the phases Exp$_L^{(z)}$ for $z\in\{-1,+1\}$ and $l >L_T$ on the event $\overline{\cF}$.

\paragraph{Bound on }$\displaystyle\underset{z \in \{-1, +1\}}{\sum}\overset{L^{(z)}+1}{\underset{l= L_T+1}{\sum}}\underset{t \in \text{Exp}^{(z)}_l}{\sum}(x^* - x_t)^{\top}\gamma^* + \underset{t \in \text{Recovery}}{\sum}(x^* - x_t)^{\top}\gamma^*$.\\

\begin{lem}\label{lem:discard_subopt} Let $x^*\in \argmax_{x\in \mathcal{X}} x^{\top}\gamma^*$ be an optimal action. Then, on the event $\overline{\mathcal{F}}$ defined in Equation \eqref{eq:defEl}, for $l\geq 1$ such that Explore$_l^{(z_{x^*})}=$ True, 
\begin{eqnarray}\label{eq:subopt_goodz}
    \mathcal{X}_{l+1}^{(z_{x^*})} &\subset& \left\{x \in \mathcal{X}_{1}^{(z_{x^*})} : (x^* - x)^{\top}\gamma^* <10 \epsilon_{l+1} \right\}.
\end{eqnarray} 
Moreover, for $l\geq 1$ such that Explore$_l^{(-z_{x^*})}=$ True,
\begin{eqnarray*}
    \mathcal{X}_{l+1}^{(-z_{x^*})} &\subset& \left\{x \in \mathcal{X}_{1}^{(-z_{x^*})} : (x^* - x)^{\top}\gamma^* <42 \epsilon_{l+1} \right\}.
\end{eqnarray*}
\end{lem}
Recall that if Recovery$\neq \emptyset$, $L^{(0)}\geq L_T$. Then, all actions sampled during the Recovery phase belong to $\mathcal{X}_{l+1}^{(-1)} \cup \mathcal{X}_{l+1}^{(+1)}$ for some $l \geq L_T$. Lemma \ref{lem:discard_subopt} shows that, on $\overline{\mathcal{F}}$, for $l \geq L_T$, the actions in $\mathcal{X}_{l+1}^{(z)}$ are sub-optimal by at most $42\epsilon_{L_{T} +1}$. Then, we get that on the event $\overline{\cF}$,
\begin{eqnarray}\label{eq:wc_EP}
    \underset{z \in \{-1, +1\}}{\sum}\overset{L^{(z)}+1}{\underset{l= L_T+1}{\sum}}\underset{t \in \text{Exp}^{(z)}_l}{\sum}(x^* - x_t)^{\top}\gamma^* + \underset{t \in \text{Recovery}}{\sum}(x^* - x_t)^{\top}\gamma^* &\leq &T \times 42\epsilon_{L_{T+1}}\nonumber\\
    &\leq& 53 \kappa_{*}^{1/3}T^{2/3}\log(T)^{1/3}.
\end{eqnarray}
\bigskip
\paragraph{Bound on }$ \overset{L^{(0)}}{\underset{l=1}{\sum}}\underset{x \in \mathcal{X}}{\sum}\mu_{l}^{(0)}(x)\Delta_{x}  + \mathds{1}\left\{\text{Explore}_{L^{(0)}+1}^{(0)} = \text{False}\right\}\underset{t \in \text{Exp}^{(0)}_{L^{(0)}+1}}{\sum} \underset{x \in \mathcal{X}_{L^{(0)}+2}^{(-1)}\cup \mathcal{X}_{L^{(0)}+2}^{(+1)}}{\max}(x^* - x)^{\top}\gamma^*$.\newline
We begin by bounding $ \mathds{1}\left\{\text{Explore}_{L^{(0)}+1}^{(0)} = \text{False}\right\}\underset{t \in \text{Exp}^{(0)}_{L^{(0)}+1}}{\sum} \underset{x \in \mathcal{X}_{L^{(0)}+2}^{(-1)}\cup \mathcal{X}_{L^{(0)}+2}^{(+1)}}{\max}(x^* - x)^{\top}\gamma^*.$ Recall that $n_{L^{(0)}+1}^{(0)} = \underset{x \in \cX}{\sum}\mu_{L^{(0)}+1}^{(0)}(x)$ is the budget that would be necessary to complete the $\widehat{\Delta}^{l}$-optimal Exploration and Elimination phase at phase $L^{(0)}+1$. On the one hand, Lemma \ref{lem:discard_subopt} implies that on the event $\overline{\cF}$,
\begin{eqnarray*}
\mathds{1}\left\{\text{Explore}_{L^{(0)}+1}^{(0)} = \text{False}\right\}\underset{t \in \text{Exp}^{(0)}_{L^{(0)}+1}}{\sum} \underset{x \in \mathcal{X}_{L^{(0)}+2}^{(-1)}\cup \mathcal{X}_{L^{(0)}+2}^{(+1)}}{\max}(x^* - x)^{\top}\gamma^* \leq 42n_{L^{(0)}+1}^{(0)}\epsilon_{L^{(0)}+2} \leq 21n_{L^{(0)}+1}^{(0)}\epsilon_{L^{(0)}+1}.
\end{eqnarray*}
On the other hand, for all $l\leq L^{(0)}+1$, the definition of $\widehat{\Delta}^l$ implies that $\widehat{\Delta}^l_x \geq \epsilon_l$ for all $x \in \cX$. Therefore, $21n_{L^{(0)}+1}^{(0)}\epsilon_{L^{(0)}+1} \leq 21n_{L^{(0)}+1}^{(0)}\min_x\widehat{\Delta}^{L^{(0)}+1}_x$. This implies that on $\overline{\cF}$,
\begin{eqnarray}\label{eq:terme_de_bord}
\mathds{1}\left\{\text{Explore}_{L^{(0)}+1}^{(0)} = \text{False}\right\}\underset{t \in \text{Exp}^{(0)}_{L^{(0)}+1}}{\sum} \underset{x \in \mathcal{X}_{L^{(0)}+2}^{(-1)}\cup \mathcal{X}_{L^{(0)}+2}^{(+1)}}{\max}(x^* - x)^{\top}\gamma^* \leq 21\underset{x \in \mathcal{X}}{\sum}\mu_{L^{(0)}+1}^{(0)}(x)\widehat{\Delta}^{L^{(0)}+1}_{x}.
\end{eqnarray}

Next, to bound the remaining terms of Equation \eqref{eq:rt_dec}, we bound the regret $\underset{x \in \mathcal{X}}{\sum}\mu_{l}^{(0)}(x)\Delta_x$ of exploration phase $\text{Exp}^{(0)}_l$ using the following lemma.
\begin{lem}\label{lem:time}
For all $l>0$, and $z \in \{-1, +1\}$, we have
\begin{equation*}
\sum_{x \in \mathcal{X}_l^{(z)}}\mu_{l}^{(z)}(x) \leq \frac{2(d+1)}{\epsilon_l^2} \log\left(\frac{kl(l+1)}{ \delta}\right) + \frac{(d+1)(d+2)}{2}.
\end{equation*}
and on $\overline{\mathcal{F}}$, we have 
\begin{equation*}
\sum_{x \in \mathcal{X}}\mu_{l}^{(0)}(x) \Delta_{x} \leq \sum_{x \in \mathcal{X}}\mu_{l}^{(0)}(x) \widehat{\Delta}_{x}^{l}\leq  \frac{2\kappa(\widehat \Delta^{l})}{\epsilon_l^2} \log\left(\frac{l(l+1)}{ \delta}\right) + 2(d+1).
\end{equation*}
\end{lem}

Then, Equation \eqref{eq:terme_de_bord} and Lemma \ref{lem:time} imply that on $\overline{\mathcal{F}}$
\begin{align}\label{eq:borne_r_0_0}
\overset{L^{(0)}}{\underset{l=1}{\sum}}\underset{x \in \mathcal{X}}{\sum}\mu_{l}^{(0)}(x)\Delta_{x} + \mathds{1}\left\{\text{Explore}_{L^{(0)}+1}^{(0)} = \text{False}\right\}\underset{t \in \text{Exp}^{(0)}_{L^{(0)}+1}}{\sum}& \underset{x \in \mathcal{X}_{L^{(0)}+2}^{(-1)}\cup \mathcal{X}_{L^{(0)}+2}^{(+1)}}{\max}(x^* - x)^{\top}\gamma^*\nonumber\\
&\leq  21\overset{L^{(0)}+1}{\underset{l=1}{\sum}}\underset{x \in \mathcal{X}}{\sum}\mu_{l}^{(0)}(x)\widehat{\Delta}^{l}_{x}\nonumber\\
& \leq 42 \overset{L^{(0)}+1}{\underset{l=1}{\sum}} \frac{\kappa(\widehat \Delta^{l})}{\epsilon_l^2} \log\left(\frac{l(l+1)}{ \delta}\right) + 42(d+1) (L^{(0)}+1) 
\end{align}

We rely on the following Lemma to bound $\kappa(\widehat \Delta^{l})$.
\begin{lem}\label{lem:kappa-l}
On $\overline{\mathcal F}$, we have for any $l\geq 1$ and any $\tau>0$
\begin{equation*}
\kappa(\widehat \Delta^{l}) \leq 513 \pa{1+{\epsilon_{l}\over \tau}} \kappa(\Delta \vee \tau).
\end{equation*}
and 
\begin{equation*}
\kappa(\widehat \Delta^{l}) \geq \kappa(\Delta \vee \epsilon_l).
\end{equation*}
\end{lem}

Lemma \ref{lem:time} and Lemma \ref{lem:kappa-l} with $\tau = \epsilon_{L^{(0)}}$ imply that on $\overline{\cF}$,
\begin{eqnarray}\label{eq:borne_r_0.1}
   \overset{L^{(0)}+1}{\underset{l=1}{\sum}} \frac{\kappa(\widehat \Delta^{l})}{\epsilon_l^2} \log\left(\frac{l(l+1)}{ \delta}\right) &\leq &  513\kappa(\Delta \vee \epsilon_{L^{(0)}})\log\left(\frac{(L^{(0)}+1)(L^{(0)}+2)}{ \delta}\right)\left(\underset{l=1}{\overset{L^{(0)}+1}{\sum}}  \frac{1}{\epsilon_l^2} +  \underset{l=1}{\overset{L^{(0)}+1}{\sum}}  \frac{1}{\epsilon_l\epsilon_{L^{(0)}}} \right)\nonumber\\
    &\leq & 513\kappa(\Delta \vee \epsilon_{L^{(0)}})\log\left(\frac{6L^{(0)}}{ \delta}\right)\left( \frac{16}{\epsilon_{L^{(0)}}^2} + \frac{4}{\epsilon_{L^{(0)}}^2}\right)\nonumber\\
     &\leq & 10260\log\left(\frac{6L^{(0)}}{ \delta}\right) \frac{\kappa(\widehat{\Delta}^{L^{(0)}})}{\epsilon_{L^{(0)}}^2}
\end{eqnarray}
where the last line follows from the second claim of  Lemma \ref{lem:kappa-l}. Now, by definition of $L^{(0)}$, $\epsilon_{L^{(0)}}\geq \left(\kappa(\widehat{\Delta}^{L^{(0)}}) \log(T)/T\right)^{1/3}$. Then, Equation \eqref{eq:borne_r_0.1} implies that
\begin{eqnarray}\label{eq:borne_r_0.1.1}
   \overset{L^{(0)}+1}{\underset{l=1}{\sum}} \frac{\kappa(\widehat \Delta^{l})}{\epsilon_l^2} \log\left(\frac{l(l+1)}{ \delta}\right) &\leq & 10260\log\left(\frac{6L^{(0)}}{ \delta}\right)\kappa(\widehat{\Delta}^{L^{(0)}})^{1/3}\log(T)^{-2/3}T^{2/3}.
\end{eqnarray}
Moreover, we observe that during each phase $l$, but the last one, we sample at least 
$$\max_{z\in \ac{-1,1}} \sum_{x\in \cX^{(z)}_{l}}\tau_{l,x}^{(z)}\geq {2 (d+1)\over \delta_{l}^2} \log(kl(l+1)/\delta)$$
actions during the G-optimal explorations, so the number of phases $L^{(0)}$ is never larger than 
\begin{equation*}
\ell_{T} = 1\vee \log_{4}(T).
\end{equation*} Using this remark, together with Equations \eqref{eq:borne_r_0_0} and \eqref{eq:borne_r_0.1.1}, we find that on $\overline{\cF}$
\begin{eqnarray}\label{eq:borne_r_0.2}
  \overset{L^{(0)}}{\underset{l=1}{\sum}}\underset{x \in \mathcal{X}}{\sum}\mu_{l}^{(0)}(x)\widehat{\Delta}^{l}_{x}  &+&  \mathds{1}\left\{\text{Explore}_{L^{(0)}+1}^{(0)} = \text{False}\right\}\underset{t \in \text{Exp}^{(0)}_{L^{(0)}+1}}{\sum} \underset{x \in \mathcal{X}_{L^{(0)}+2}^{(-1)}\cup \mathcal{X}_{L^{(0)}+2}^{(+1)}}{\max}(x^* - x)^{\top}\gamma^*\nonumber\\ &&\leq 2^{19}\log\left(\frac{6L^{(0)}}{ \delta}\right) \kappa(\widehat{\Delta}^{L^{(0)}})T^{2/3}\log(T)^{-2/3} + 42\ell_T.
\end{eqnarray}

\paragraph{Bound on $\underset{z \in \{-1, +1\}}{\sum}\overset{L_T}{\underset{l= 1}{\sum}}\left(\underset{x \in \mathcal{X}_l^{(z)}}{\sum}\mu_{l}^{(z)}(x)\right)\underset{x \in \mathcal{X}_l^{(z)}}{\max}(x^* - x)^{\top}\gamma^*$.}
We bound the remaining term in Equation \eqref{eq:rt_dec} using the first claim in Lemma \ref{lem:time} and Lemma \ref{lem:discard_subopt}. On $\overline{\cF}$,

\begin{eqnarray}\label{eq:bound_regret_1}
   \underset{z \in \{-1, +1\}}{\sum}\overset{L_T}{\underset{l= 1}{\sum}}\left(\underset{x \in \mathcal{X}_l^{(z)}}{\sum}\mu_{l}^{(z)}(x)\right)\underset{x \in \mathcal{X}_l^{(z)}}{\max}(x^* - x)^{\top}\gamma^*
   &\leq&  2\underset{l=1}{\overset{L_T}{\sum}}  \left(\frac{2(d+1)}{\epsilon_l^2} \log\left(\frac{kl(l+1)}{ \delta}\right) + \frac{(d+1)(d+2)}{2}\right) 42 \epsilon_{l} \nonumber\\
    &\leq & \frac{336(d+1)}{\epsilon_{L_T}} \log\left(\frac{kL_T(1+L_T)}{ \delta}\right) + 168(d+1)(d+2)\nonumber\\
    &\leq &  267(d+1)\kappa_{*}^{-1/3}T^{1/3}\log(T)^{-1/3} \log\left(\frac{kL_T(1+L_T)}{ \delta}\right) \nonumber\\
    &&+ 168(d+1)(d+2).
\end{eqnarray}
Combing Equations \eqref{eq:regret_dec}, \eqref{eq:rt_dec}, \eqref{eq:wc_EP}, \eqref{eq:borne_r_0.2}, and \eqref{eq:bound_regret_1}, and using $\delta = T^{-1}$, $\kappa(\widehat{\Delta}^{L^{(0)}})\leq  \kappa_*$ and $L_T\leq 4T/\log(2)$, we get for all $T\geq 1$
\begin{equation*}
    R_T \leq C\left(\kappa_{*}^{1/3}T^{2/3}\log(T)^{1/3} + (d\vee \kappa_{*})\log(T) + d^2 +  d\kappa_{*}^{-1/3}T^{1/3}\log(kT)\log(T)^{-1/3} \right)
\end{equation*}
for some absolute constant $C>0$. Finally, for 
$$T\geq \frac{\left((d\vee \kappa_*)^{3/2}\log(T)\right)  \vee d^3}{\sqrt{\kappa_*}} \vee \frac{(d\log(kT))^{3}}{(\kappa_*\log(T))^{2}},$$ 
we get
\begin{equation*}
    R_T \leq C'\kappa_{*}^{1/3}T^{2/3}\log(T)^{1/3}.
\end{equation*}

\subsection{Proof of Theorem \ref{thm:upper_bound_delta}}

The beginning of the proof of Theorem \ref{thm:upper_bound_delta} follows the same lines as the proof of Theorem \ref{thm:upper_bound_worst_case}. We begin by decomposing the regret as 
\begin{equation}\label{eq:regret_dec_delta}
    R_T \leq \underset{t\leq T}{\sum}\mathbb{E}_{\vert \overline{\mathcal{F}}}\left[(x^* - x_t)^{\top}\gamma^*\right] + 2T\mathbb{P}\left[\mathcal{F}\right].
\end{equation}
where $\mathcal{F}$ is defined in Equation \eqref{eq:defEl}. On the one hand, Lemma \ref{lem:proba_F} implies $T\mathbb{P}\left[\mathcal{F}\right]\leq 2\delta T$. Then, Equation \eqref{eq:regret_dec_delta} implies
\begin{eqnarray}\label{eq:rt_dec_delta}
    R_T &\leq& 4\delta T + \mathbb{E}_{\vert \overline{\mathcal{F}}}\left[\underset{z \in \{-1, +1\}}{\sum}\overset{L^{(z)}+1}{\underset{l\geq 1}{\sum}}\underset{t \in \text{Exp}^{(z)}_l}{\sum}(x^* - x_t)^{\top}\gamma^*\right]  + \mathbb{E}_{\vert \overline{\mathcal{F}}}\left[\underset{t \in \text{Recovery}}{\sum}(x^* - x_t)^{\top}\gamma^*\right]\\
    && + \mathbb{E}_{\vert \overline{\mathcal{F}}}\left[\overset{L^{(0)}}{\underset{l=1}{\sum}}\underset{x \in \mathcal{X}}{\sum}\mu_{l}^{(0)}(x)\Delta_{x}\right] + \mathbb{E}_{\vert \overline{\mathcal{F}}}\left[ \mathds{1}\left\{\text{Explore}_{L^{(0)}+1}^{(0)} = \text{False}\right\}\underset{t \in \text{Exp}^{(0)}_{L^{(0)}+1}}{\sum} \underset{x \in \mathcal{X}_{L^{(0)}+2}^{(-1)}\cup \mathcal{X}_{L^{(0)}+2}^{(+1)}}{\max}(x^* - x)^{\top}\gamma^*\right]\nonumber
\end{eqnarray}
where $\mathcal{F}$ is defined in Equation \eqref{eq:defEl}, and where we used the convention that the sum over an empty set is null. 

\paragraph{Bound on }$\mathds{1}\left\{\text{Explore}_{L^{(0)}+1}^{(0)} = \text{False}\right\}\underset{t \in \text{Exp}^{(0)}_{L^{(0)}+1}}{\sum} \underset{x \in \mathcal{X}_{L^{(z)}}+1}{\max}(x^* - x)^{\top}\gamma^*$.\newline
Similarly to the proof of Theorem \ref{thm:upper_bound_worst_case}, we use Lemma \ref{lem:discard_subopt} and Lemma \ref{lem:kappa-l} to show that on $\cbF$
\begin{eqnarray}\label{eq:terme_bords_delta}
  \mathds{1}\left\{\text{Explore}_{L^{(0)}+1}^{(0)} = \text{False}\right\}\underset{t \in \text{Exp}^{(0)}_{L^{(0)}+1}}{\sum} \underset{x \in \mathcal{X}_{L^{(z)}}+1}{\max}(x^* - x)^{\top}\gamma^* \leq 21 \underset{x \in \mathcal{X}}{\sum}\mu_{L^{(0)+1}}^{(0)}(x)\widehat{\Delta}^{L^{(0)}+1}_{x}.
\end{eqnarray}

\paragraph{Bound on }$\underset{z \in \{-1, +1\}}{\sum}\overset{L^{(z)}+1}{\underset{l\geq 1}{\sum}}\underset{t \in \text{Exp}^{(z)}_l}{\sum}(x^* - x_t)^{\top}\gamma^*.$\newline
Lemma \ref{lem:discard_subopt} shows that for $l \leq L^{(z)}$, the actions in $\mathcal{X}_{l+1}^{(z)}$ are sub-optimal by at most an additional factor at most $21\epsilon_l$. Let us set $l_{\Delta_{\min}} = \lceil-\log_{2}(\Delta_{\min}/21)\rceil$, so that
\begin{equation*}
{\Delta_{\min}\over 42} \leq \epsilon_{l_{\Delta_{\min}}} \leq {\Delta_{\min}\over 21}.
\end{equation*}
For $l \geq l_{\Delta_{\min}}$, we have  $\mathcal{X}_{l+1}^{(-1)}\cup \mathcal{X}_{l+1}^{(+1)} = \{x_{z^*}\}$. Thus, $l^{(-z_{x^*})} \leq l_{\Delta_{\min}}$, and for $l \geq l_{\Delta_{\min}}$, the algorithm selects only $x^*$ during the phase Exp$_l^{(z^*)}$. Then, combining Lemmas \ref{lem:time} and \ref{lem:discard_subopt}, and the fact that $L^{(z)} +1 \leq \ell_T$, we find that, on $\cbF$,
\begin{eqnarray}\label{eq:bound_two_bandits_delta}
    \underset{z \in \{-1, +1\}}{\sum} \overset{L^{(z)}+1}{\underset{l= 1}{\sum}}\underset{t \in \text{Exp}^{(z)}_l}{\sum}(x^* - x_t)^{\top}\gamma^* &\leq & \underset{z \in \{-1, +1\}}{\sum}\overset{l_{\Delta_{\min}}+1 \land \ell_{T}}{\underset{l= 1}{\sum}}\left(\underset{x \in \mathcal{X}_l^{(z)}}{\sum}\mu_{l}^{(z)}(x)\right)\underset{x \in \mathcal{X}_l^{(z)}}{\max}(x^* - x)^{\top}\gamma^*\nonumber \\
    &\leq & 2\underset{l=1}{\overset{l_{\Delta_{\min}}+1\land \ell_{T}}{\sum}}  \left(\frac{2(d+1)}{\epsilon_l^2} \log\left(\frac{kl(l+1)}{ \delta}\right) + \frac{(d+1)(d+2)}{2}\right)42 \epsilon_{l}\nonumber\\
    &\leq & 84(d+1)(d+2) + \epsilon_{l_{\Delta_{\min}}}^{-1}\times 672(d+1) \log\left(\frac{k(1+\ell_{T})(2+\ell_{T})}{ \delta}\right)\nonumber\\
&\leq & 84(d+1)(d+2) +\frac{28224(d+1)}{\Delta_{\min}} \log\left(\frac{k(1+\ell_{T}))(2+\ell_{T}))}{ \delta}\right).
\end{eqnarray}

\paragraph{Bound on }$\displaystyle \underset{t \in \text{Recovery}}{\sum}(x^* - x_t)^{\top}\gamma^* + \overset{L^{(0)}}{\underset{l=1}{\sum}}\underset{x \in \mathcal{X}}{\sum}\mu_{l}^{(0)}(x)\Delta_{x}+ \underset{x \in \mathcal{X}}{\sum}\mu_{L^{(0)}+1}^{(0)}(x)\widehat{\Delta}^{L^{(0)}+1}_{x}.$

We use the following lemma to bound the number of phases necessary to eliminate the sub-optimal group.

\begin{lem}\label{lem:discard_z}
On the event $\overline{\mathcal{F}}$ defined in  Equation \eqref{eq:defEl}, for $l\geq 1$ such that $\epsilon_l \leq \frac{\Delta_{\neq}}{8}$ and Explore$_L^{(0)} = $ True, $\widehat{z^*}_{l+1} = z_{x^*}$.
\end{lem}

Let $l_{\Delta_{\neq}} = \lceil-\log(\Delta_{\neq}/8)/\log(2)\rceil $ be such that 
\begin{equation}\label{eq:delta:neq}
{\Delta_{\neq}\over 16}\leq \epsilon_{l_{\Delta_{\neq}}} \leq {\Delta_{\neq}\over 8}.
\end{equation}
Lemma \ref{lem:discard_z} implies that on $\cbF$, $ L^{(0)} \leq l_{\Delta_{\neq}}$. 

\bigskip
To bound the remaining terms, we consider two cases, corresponding to Recovery$=\emptyset$ and Recovery$\neq \emptyset$.

\bigskip
\underline{\textbf{Case 1:} Recovery$=\emptyset$.}
Our case assumption implies that 
\begin{equation}\label{eq:no_recovery_1}
    \underset{t \in\text{Recovery}}{\sum}(x^* - x_t)^{\top}\gamma^* = 0.
\end{equation}
Lemma \ref{lem:kappa-l} implies that 
\begin{eqnarray*}
  \overset{L^{(0)}}{\underset{l=1}{\sum}}\underset{x \in \mathcal{X}}{\sum}\mu_{l}^{(0)}(x)\Delta_{x}+ \underset{x \in \mathcal{X}}{\sum}\mu_{L^{(0)}+1}^{(0)}(x)\widehat{\Delta}^{L^{(0)}+1}_{x} \leq \overset{L^{(0)}+1}{\underset{l=1}{\sum}}\underset{x \in \mathcal{X}}{\sum}\mu_{l}^{(0)}(x)\widehat{\Delta}^l_{x}.
\end{eqnarray*}
Moreover, $ L^{(0)} \leq l_{\Delta_{\neq}} \land \ell_T$, so on $\cbF$
\begin{eqnarray*}
    \overset{L^{(0)}+1}{\underset{l=1}{\sum}}\underset{x \in \mathcal{X}}{\sum}\mu_{l}^{(0)}(x)\widehat{\Delta}^l_{x}&\leq &\overset{(l_{\Delta_{\neq}}\land\ell_T) +1}{\underset{l=1}{\sum}}\underset{x \in \mathcal{X}}{\sum}\mu_{l}^{(0)}(x)\widehat{\Delta}^l_{x}.\nonumber
\end{eqnarray*}
Using Lemma \ref{lem:time}, we find that on  $\overline{\mathcal{F}}$
\begin{align*}
\overset{(l_{\Delta_{\neq}}\land\ell_T) +1}{\underset{l=1}{\sum}}\underset{x \in \mathcal{X}}{\sum}\mu_{l}^{(0)}(x)\widehat{\Delta}^l_{x} &\leq 
\sum_{l=1}^{(l_{\Delta_{\neq}}\land\ell_T) +1} \frac{2\kappa(\widehat \Delta^{l})}{\epsilon_l^2} \log\left(\frac{l(l+1)}{ \delta}\right) +2(d+1)(\ell_T+1)\\
&\leq  2\log\left(\frac{(\ell_T+1)(\ell_T+2)}{\delta}\right) \sum_{l=1}^{l_{\Delta_{\neq}}+1} \frac{\kappa(\widehat \Delta^{l})}{\epsilon_l^2}+2(d+1)(\ell_T+1).
\end{align*}
Using Lemma \ref{lem:kappa-l} with $\tau = \Delta_{\neq}$ and (\ref{eq:delta:neq}), we have on $\cbF$
\begin{align*}
 \sum_{l=1}^{l_{\Delta_{\neq}}+1} \frac{\kappa(\widehat \Delta^{l})}{\epsilon_l^2} &\leq 513 \kappa(\Delta \vee \Delta_{\neq}) \sum_{l=1}^{l_{\Delta_{\neq}}+1} \pa{\epsilon_{l}^{-2}+\epsilon_{l}^{-1}/\Delta_{\neq}}\\
 & \leq {2^{18} \kappa(\Delta \vee \Delta_{\neq})\over \Delta_{\neq}^2}.
\end{align*}
We obtain on $\cbF$
\begin{eqnarray}\label{eq:rp_delta_2}
\overset{L^{(0)}+1}{\underset{l=1}{\sum}}\underset{x \in \mathcal{X}}{\sum}\mu_{l}^{(0)}(x)\widehat{\Delta}^l_{x}
    &\leq & 2^{19}\log\left(\frac{(\ell_T+1)(\ell_T+2)}{\delta}\right) { \kappa(\Delta \vee \Delta_{\neq})\over \Delta_{\neq}^2}+2(d+1)(\ell_T+1).
\end{eqnarray}
Combining Equations  \eqref{eq:bound_two_bandits_delta}, \eqref{eq:terme_bords_delta}, \eqref{eq:no_recovery_1}, and \eqref{eq:rp_delta_2}, we find that on $\cbF$, when Recovery$=\emptyset$, there exsists an absolute constant $c>0$ such that for $\delta = T^{-1}$, 
\begin{align}\label{eq:fin_case_1}
    \underset{z \in \{-1, +1\}}{\sum}\overset{L^{(z)}+1}{\underset{l\geq 1}{\sum}}\underset{t \in \text{Exp}^{(z)}_l}{\sum}(x^* - x_t)^{\top}\gamma^* + & \underset{t \in \text{Recovery}}{\sum}(x^* - x_t)^{\top}\gamma^*
    +  \overset{L^{(0)}}{\underset{l=1}{\sum}}\underset{x \in \mathcal{X}}{\sum}\mu_{l}^{(0)}(x)\Delta_{x} \\
    + \mathds{1} \big\{\text{Explore}_{L^{(0)}+1}^{(0)} = \text{False}\big\} &\underset{t \in \text{Exp}^{(0)}_{L^{(0)}+1}}{\sum} \underset{x \in \mathcal{X}_{L^{(0)}+2}^{(-1)}\cup \mathcal{X}_{L^{(0)}+2}^{(+1)}}{\max} (x^* - x)^{\top}\gamma^* \nonumber \\
     \leq  c\bigg( d^2 +&\left(\frac{d}{\Delta_{\min}} \vee \frac{\kappa(\Delta\vee\Delta_{\neq})}{\Delta_{\neq}^2}\right)\log(T) + \frac{d}{\Delta_{\min}}\log(k)\bigg).\nonumber
\end{align}

\bigskip

\underline{\textbf{Case 2:} Recovery$\neq\emptyset$.}
In this case, the algorithm enters Recovery at phase $L^{(0)}$, so 
$\text{Explore}_{L^{(0)}+1}^{(0)}=$False and $\text{Exp}^{(0)}_{L^{(0)}+1} = \emptyset$, and
\begin{equation}\label{eq:cas2}
    \mathds{1} \big\{\text{Explore}_{L^{(0)}+1}^{(0)} = \text{False}\big\} \underset{t \in \text{Exp}^{(0)}_{L^{(0)}+1}}{\sum} \underset{x \in \mathcal{X}_{L^{(0)}+2}^{(-1)}\cup \mathcal{X}_{L^{(0)}+2}^{(+1)}}{\max} (x^* - x)^{\top}\gamma^* =0.
\end{equation}

Using Lemma \ref{lem:discard_subopt}, we see that
\begin{equation*}
    \underset{t \in \text{Recovery}}{\sum}(x^* - x_t)^{\top}\gamma^* \leq 21T\epsilon_{L^{(0)}+1}.
\end{equation*}
On the other hand, in the Recovery phase, $\epsilon_{L^{(0)}+1} \leq \left(\kappa(\widehat{\Delta}^{L^{(0)}+1})\log(T)/T)\right)^{1/3}$. Thus, 
\begin{equation*}
    \underset{t \in \text{Recovery}}{\sum}(x^* - x_t)^{\top}\gamma^* \leq \frac{21\kappa(\widehat{\Delta}^{L^{(0)}+1})\log(T)}{\epsilon_{L^{(0)}+1}^2}.
\end{equation*}
Now, Lemma \ref{lem:time} show that
\begin{equation*}
    \overset{L^{(0)}}{\underset{l=1}{\sum}}\underset{x \in \mathcal{X}}{\sum}\mu_{l}^{(0)}(x)\Delta_{x} \leq 4\log(2L^{(0)}\delta^{-1})\overset{L^{(0)}}{\underset{l=1}{\sum}}\frac{\kappa(\widehat{\Delta}^{l})}{\epsilon_{l}^2} + 4dL^{(0)}.
\end{equation*}

Combining these results, and using $L^{(0)}\leq \ell_T$, we see that 
\begin{equation}\label{eq:cas2_bourne_Rec}
    \underset{t \in \text{Recovery}}{\sum}(x^* - x_t)^{\top}\gamma^* +  \overset{L^{(0)}}{\underset{l=1}{\sum}}\underset{x \in \mathcal{X}}{\sum}\mu_{l}^{(0)}(x)\Delta_{x} \leq  4dL^{(0)} + \left(4\log(2\ell_T\delta^{-1}) \vee 21\log(T)\right)\overset{L^{(0)}+1}{\underset{l=1}{\sum}} \frac{\kappa(\widehat{\Delta}^{l})}{\epsilon_{l}^2}.
\end{equation}
Using Lemma \ref{lem:kappa-l} with $\tau = \epsilon_{L^{(0)}}$, we see that

\begin{eqnarray*}
    \overset{L^{(0)}+1}{\underset{l=1}{\sum}} \frac{\kappa(\widehat{\Delta}^{l})}{\epsilon_{l}^2} &\leq& 513\overset{L^{(0)}+1}{\underset{l=1}{\sum}} \frac{\kappa(\Delta \vee \epsilon_{L^{(0)}})}{\epsilon_{l}^2} + 513\overset{L^{(0)}+1}{\underset{l=1}{\sum}} \frac{\kappa(\Delta \vee \epsilon_{L^{(0)}})}{\epsilon_{L^{(0)}} \epsilon_{l}} \nonumber\\
    && \leq 10260\frac{\kappa(\Delta \vee \epsilon_{L^{(0)}})}{\epsilon_{L^{(0)}}^2}.
\end{eqnarray*}

Now, the algorithm enters the Recovery phase before finding the best group, so we must have $L^{(0)}\leq l_{\Delta_{\neq}}$. This implies that 
\begin{eqnarray*}
    \overset{L^{(0)}+1}{\underset{l=1}{\sum}} \frac{\kappa(\widehat{\Delta}^{l})}{\epsilon_{l}^2} &\leq&
   2^{18}\frac{\kappa(\Delta \vee \epsilon_{L^{(0)}})}{\Delta_{\neq}^2}.
\end{eqnarray*}
Finally, note that $L^{(0)} \geq L_T$, so $\epsilon_{L^{(0)}}\leq \epsilon_{L_T} = \varepsilon_T$, and

\begin{eqnarray}\label{eq:cas2last}
    \overset{L^{(0)}+1}{\underset{l=1}{\sum}} \frac{\kappa(\widehat{\Delta}^{l})}{\epsilon_{l}^2} &\leq&
     2^{18}\frac{\kappa(\Delta \vee \varepsilon_{T})}{\Delta_{\neq}^2}.
\end{eqnarray}

Combining Equations  \eqref{eq:bound_two_bandits_delta}, \eqref{eq:cas2}, \eqref{eq:cas2_bourne_Rec}, and \eqref{eq:cas2last}, we find that on $\cbF$, when Recovery$\neq\emptyset$, there exists an absolute constant $c>0$ such that for $\delta = T^{-1}$, 
\begin{align}\label{eq:fin_case_2}
    \underset{z \in \{-1, +1\}}{\sum}\overset{L^{(z)}+1}{\underset{l\geq 1}{\sum}}\underset{t \in \text{Exp}^{(z)}_l}{\sum}(x^* - x_t)^{\top}\gamma^* + & \underset{t \in \text{Recovery}}{\sum}(x^* - x_t)^{\top}\gamma^*
    +  \overset{L^{(0)}}{\underset{l=1}{\sum}}\underset{x \in \mathcal{X}}{\sum}\mu_{l}^{(0)}(x)\Delta_{x} \\
    + \mathds{1} \big\{\text{Explore}_{L^{(0)}+1}^{(0)} = \text{False}\big\} &\underset{t \in \text{Exp}^{(0)}_{L^{(0)}+1}}{\sum} \underset{x \in \mathcal{X}_{L^{(0)}+2}^{(-1)}\cup \mathcal{X}_{L^{(0)}+2}^{(+1)} }{\max} (x^* - x)^{\top}\gamma^* \nonumber \\
    \leq  c\bigg( d^2 +&\left(\frac{d}{\Delta_{\min}} \vee \frac{\kappa(\Delta\vee\varepsilon_T)}{\Delta_{\neq}^2}\right) \log(T) + \frac{d\log(k)}{\Delta_{\min}}\bigg).\nonumber
\end{align}

\bigskip

\paragraph{Conclusion}
We conclude the proof of Theorem \ref{thm:upper_bound_delta} by combining Equations \eqref{eq:rt_dec_delta}, \eqref{eq:fin_case_1} and \eqref{eq:fin_case_2}.


\subsection{Proof of Theorem \ref{thm:lower_bound_worst_case}} \label{sec:lower_bound_worst_case_d_2}
Consider the actions $\mathcal{A}$ defined in the following lemma.

\begin{lem} \label{lem:kappa}
 Let the action set be given by $\mathcal{A} =\left\{\left({x_1 \atop z_{x_1}}\right), ...,\left({x_{d+1} \atop z_{x_{d+1}}}\right) \right\}$, where $\left({x_1 \atop z_{x_1}}\right) = e_1 + e_{d+1}$, $\left({x_i \atop z_{x_i}}\right) = e_i - e_{d+1}$ for $i \in \{2, ..., d\}$, and $\left({x_{d+1} \atop z_{x_{d+1}}}\right) = -\left(1-\frac{2}{\sqrt{\kappa_*}+1} \right)e_1 - e_{d+1}$. It holds that
\begin{equation*}
   \underset{\pi \in \mathcal{P}^{\mathcal{A}}}{\min}\left\{ e_{d+1}^{\top}\left(\underset{\left({x \atop z}\right) \in \mathcal{A}}{\displaystyle \sum}\pi_x \left({x \atop z_x}\right) \left({x \atop z_x}\right)^{\top} \right)^+e_{d+1}\right\} = \kappa.
\end{equation*}
\end{lem}

By Lemma \ref{lem:kappa}, $\mathcal{A}\in \mathbf{A}_{\kappa_*,d}$. We will introduce two bandit problems characterized by two parameters $\theta^{(1)}_T$ and $\theta^{(2)}_T$ - assuming that the noise $\xi_t$ is Gaussian and i.i.d. - and we prove that for any algorithm, the regret for one of those two problems must be of larger order than $\kappa_*^{1/3}T^{2/3}$.

We also consider the following two alternative problems. For a small $1/4>\rho_T>0$ where $\rho_T = T^{-1/3}\kappa_*^{1/3}$ (satisfied since $T > 4^3 \kappa_*$), the two alternative action parameters are defined as:
\begin{eqnarray*}
    \gamma^{(1)}_T & =& \frac{1+\rho_T}{2}e_1 + \frac{1 - \rho_T}{2}e_{2} - \frac{\rho_T}{2}  \left(\sum_{3 \leq j \leq d} e_j\right) \\
    \gamma^{(2)}_T & =& \frac{1-\rho_T}{2}e_1 + \frac{1 + \rho_T}{2}e_{2} +\frac{\rho_T}{2}  \left(\sum_{3\leq j \leq d} e_j\right).
\end{eqnarray*}
On top of this, two bias parameters are defined as $\omega^{(1)}_T = -\frac{\rho_T}{2}$ and $\omega^{(2)}_T = \frac{\rho_T}{2}$. Through this, we define the two bandit problems of the sketch of proof of Lemma \ref{lem:kappa} characterized by $\theta^{(1)}_T = \left({\gamma^{(1)}_T \atop \omega^{(1)}_T}\right)$ and $\theta^{(2)}_T= \left({\gamma^{(2)}_T \atop \omega^{(2)}_T}\right)$ - and where the distribution of the noise $\xi_t$ is supposed to be Gaussian and i.i.d. We refer to these two problems respectively as {\bf Problem 1} and {\bf Problem 2}. We write $R_T^{(1)}$, $\mathbb{P}^{(1)}$ and $\mathbb{E}^{(1)}$ (respectively $R_T^{(2)}$,  $\mathbb{P}^{(2)}$ and $\mathbb{E}^{(2)}$) for the regret, probability and expectation for the first bandit problem, when the parameter is $\theta^{(1)}_T$ (respectively the second bandit problem with $\theta^{(2)}_T$). We also write $\mathbb{P}^{(i)}_j$ for the distribution of a sample received in {\bf Problem i} when sampling action $x_j$ at any given time $t$ - note that by definition of the bandit problems, this distribution does not depend on $t$ and on the past samples given that action $x_j$ is sampled.

The three following facts hold on these two bandit problems:
\begin{enumerate}
\item[{\bf Fact 1}] The parameters $\gamma^{(1)}_T$ and $\gamma^{(2)}_T$ are chosen so that $x_1$ is the unique best action for {\bf Problem 1}, and $x_2$ is the unique best action for {\bf Problem 2}. Choosing any sub-optimal action induces an instantaneous regret of at least $\rho_T$, and choosing the very sub-optimal action $x_{d+1}$ induces an instantaneous regret of at least $1/2$.
\item[{\bf Fact 2}] Because of the chosen bias parameters, the distributions of the evaluations of all actions but $x_{d+1}$ are exactly the same under the two bandit problems characterized by $\theta^{(1)}$ and $\theta^{(2)}_T$ - i.e.~exactly the same data is observed under the two alternative bandit problems defined by the two alternative parameters for all actions but $x_{d+1}$. More precisely, for $i\in \{1,2\}$, in {\bf Problem i} and at any time $t$, when sampling action $x_i$ where $i \leq 2$, we observe a sample distributed according to $\mathcal N(1/2,1)$ - i.e.~$\mathbb{P}^{(i)}_j$ is $\mathcal N(1/2,1)$ - and when sampling action $x_i$ where $2 <i \geq d+1$, we observe a sample distributed according to $\mathcal N(0,1)$ - i.e.~$\mathbb{P}^{(i)}_j$ is $\mathcal N(0,1)$.

\item[{\bf Fact 3}] The distributions of the outcomes of the evaluation of action $x_{d+1}$ differs in the two bandit problems. Set $\alpha = 2/(\sqrt{\kappa_*}+1)$. In {\bf Problem 1}, $\mathbb{P}^{(1)}_{d+1}$ is $\mathcal N(-\frac{1-\alpha - \rho_T\alpha}{2},1)$. In {\bf Problem 2}, $\mathbb{P}^{(2)}_{d+1}$ is $\mathcal N(-\frac{1-\alpha + \rho_T\alpha}{2},1)$. So that the difference between the means of the evaluations of action $x_{d+1}$ in the two bandit problems is $\bar \Delta = \rho_T\alpha = \frac{2\rho_T}{\sqrt{\kappa_*}+1} \leq \frac{2\rho_T}{\sqrt{\kappa_*}}$.
\end{enumerate}

For $i\leq d+1$, we write $N_{i}(T)$ for the number of times that action $x_i$ has been selected before time $T$. In {\bf Problem 1}, choosing the action $x_{d+1}$ leads to an instantaneous regret larger than $\frac{1}{2}$ ({\bf Fact 1}), so that $$R_T^{(1)} \geq \frac{\mathbb{E}^{(1)}\left[N_{x_{d+1}}(T)\right]}{2}.$$
If $\mathbb{E}^{(1)}\left[N_{d+1}(T)\right] \geq \frac{T^{2/3}\kappa_*^{1/3}}{2}$, then Theorem \ref{thm:upper_bound_worst_case} follows immediately; we therefore consider from now on the case when 
\begin{equation}\label{eq:E}
\mathbb{E}^{(1)}\left[N_{d+1}(T)\right] \leq \frac{T^{2/3}\kappa_*^{1/3}}{2}.
\end{equation}

Now, let us define the event
$$F = \left\{N_{1}(T) \geq \frac{T}{2}\kappa_*^{1/3}\right\}.$$
Note that action $x_1$ is optimal for {\bf Problem 1} and that action $x_2$ is optimal for {\bf Problem 2} ({\bf Fact 1}). Since choosing an action that is sub-optimal leads to an instantaneous regret larger than $\rho_T$ ({\bf Fact 1}), we also have 
$$R_T^{(1)} \geq \frac{T\rho_T}{2}\mathbb{P}^{(1)}\left(\overline{F} \right)$$ and $$R_T^{(2)} \geq \frac{T\rho_T}{2}\mathbb{P}^{(2)}\left(F \right).$$

Then, Bretagnolle-Huber inequality (see, e.g., Theorem 14.2 in \cite{BanditBook}) implies that
\begin{equation*}
    R_T^{(1)} + R_T^{(2)} \geq \frac{T\rho_T}{4}\exp\left(-KL\left(\mathbb{P}^{(1)},\mathbb{P}^{(2)} \right)\right).
\end{equation*}
For the choice $\rho_T = T^{-1/3}\kappa_*^{1/3}$, this implies that 
\begin{equation}\label{eq:lb_regret}
    R_T^{(1)} + R_T^{(2)} \geq \frac{T^{2/3}\kappa_*^{1/3}}{4}\exp\left(-KL\left(\mathbb{P}^{(1)},\mathbb{P}^{(2)} \right)\right).
\end{equation}

Now, the Kullback-Leibler divergence between $\mathbb{P}^{(1)}$ and $\mathbb{P}^{(2)}$ can be rewritten as follows (see, e.g., Lemma 15.1 in \cite{BanditBook}) :
\begin{eqnarray*}
    KL(\mathbb{P}^{(1)}, \mathbb{P}^{(2)}) &=& \frac{1}{2}\underset{j \leq d+1}{\sum} \mathbb{E}^{(1)}\left[N_{j}(T)\right]KL(\mathbb{P}^{(1)}_{j}, \mathbb{P}^{(2)}_{j}).
\end{eqnarray*}
By {\bf Fact 2}, we have that for any $j\leq d$, $\mathbb{P}^{(1)}_{j} =\mathbb{P}^{(2)}_{j}$. So that
\begin{eqnarray*}
    KL(\mathbb{P}^{(1)}, \mathbb{P}^{(2)}) &=& \frac{1}{2} \mathbb{E}^{(1)}\left[N_{d+1}(T)\right]KL(\mathbb{P}^{(1)}_{d+1}, \mathbb{P}^{(2)}_{d+1}).
\end{eqnarray*}
By the characterization of $\mathbb{P}^{(1)}_{d+1}, \mathbb{P}^{(2)}_{d+1}$ in {\bf Fact 3}, and recalling that the Kullback-Leibler divergence between two normalized Gaussian distributions is given by the squared distance between their means, we find that
\begin{eqnarray*}
    KL(\mathbb{P}^{(1)}, \mathbb{P}^{(2)}) &=& \frac{1}{2}\mathbb{E}^{(1)}\left[N_{d+1}(T)\right] \bar \Delta^2.
\end{eqnarray*}
Thus, by the definition of $\bar \Delta$ in {\bf Fact 3} and by Equation~\eqref{eq:E}
\begin{equation}\label{eq:lb_kl}
    KL\left(\mathbb{P}^{(1)},\mathbb{P}^{(2)}\right) = \frac{1}{2}\mathbb{E}^{(1)}\left[N_{d+1}(T)\right]\left(\frac{2\rho_T}{\sqrt{\kappa_*}+1}\right)^2 \leq \frac{T^{2/3}\kappa_*^{1/3}}{4} \times\frac{4\rho_T^{2}}{\kappa_*} = 1,
\end{equation}
reminding that $\rho_T = T^{-1/3}\kappa_*^{1/3}$.

Combining Equations~\eqref{eq:lb_regret} and~\eqref{eq:lb_kl} implies that 
\begin{equation*}
    \max\left\{R_T^{(1)}, R_T^{(2)}\right\}\geq \frac{T^{2/3}\kappa^{1/3}}{8}\exp(-1),
\end{equation*}
which concludes the proof of Theorem \ref{thm:lower_bound_worst_case}.

\subsection{Proof of Theorems \ref{thm:lower_bound_gap}} 
Theorems \ref{thm:lower_bound_gap} follows directly from the next Theorem.

\begin{thm}\label{thm:lower_bound}
For all $\kappa_*\geq 1$ and all $d\geq 4$, there exists an action set $\mathcal{A} \in \mathbf{A}_{\kappa_*, d}$, such that for all bandit algorithms, for all $(\Delta_{\min},\Delta_{\neq})  \in (0,\nicefrac{1}{8})^2$ with $\Delta_{\min}\leq\Delta_{\neq}$, and for all budget $T \geq 2$, there exists a problem characterized by $\theta\in \mathbf{\Theta}^{\mathcal{A}}_{\Delta_{\min}, \Delta_{\neq}} $ such that the regret of the algorithm on the problem satisfies
\begin{eqnarray}\label{eq:lb_delta_deux_dep}
R_T^{\theta} &\geq& \left[\frac{d}{10\Delta_{\min}}\log \left(T\right)\left[1 -\frac{\log\left(\frac{8 d \log \left(T\right)}{\Delta_{\min}^2}\right)}{\log\left(T\right)}\right]\right] \lor \left[
\frac{\kappa_*+1}{4\Delta_{\neq}^2} \log \left(T\right)\left[1 - \frac{\log\left(\frac{8\kappa_*\log \left(T\right)}{\Delta_{\neq}^3}\right)}{\log\left(T\right)}\right]\right]\nonumber\\ 
&&\vee \left[\frac{\kappa_*}{4\Delta_{\neq}^2}\left[1\land \log\left(\frac{T\Delta_{\neq}^3}{8\kappa_*}\right)\right]\right].
\end{eqnarray}
Moreover, on this problem, $\kappa(\Delta) \in \left[\nicefrac{\kappa_*}{8}, 2\kappa_*\right]$.
\end{thm}

\begin{remark}
Note that Theorem \ref{thm:lower_bound} allows us to recover a lower bound similar to that of Theorem \ref{thm:lower_bound_worst_case} by choosing $\Delta_{\neq}$ and $\Delta_{\min}$ of the order $\kappa_*^{1/3}T^{-1/3}$, however this bound only holds for $d$ larger than 4.
\end{remark}

We prove Theorem \ref{thm:lower_bound} for the following set of actions $\cA$: $\mathcal{A} =\left\{\left({x_1 \atop z_{x_1}}\right), ...,\left({x_{d+1} \atop z_{x_{d+1}}}\right) \right\}$, where $\left({x_i \atop z_{x_i}}\right) = e_i + e_{d+1}$, for $i \in \{2, ..., \lfloor d/2\rfloor\}$, $\left({x_i \atop z_{x_i}}\right) = e_i - e_{d+1}$ for $i \in \{\lfloor d/2\rfloor +1, ..., d\}$, and $\left({x_{d+1} \atop z_{x_{d+1}}}\right) = -\left(1-\frac{2}{\sqrt{\kappa_*}+1} \right)e_1 - e_{d+1}$. Then, by Lemma \ref{lem:kappa_2}, for this choice of action set, we have $\mathcal{A}\in \mathbf{A}_{\kappa_*,d}$. 

We consider the following set of bandit problems: for $i \in \{1, ..., \lfloor d/2 \rfloor +1\}$  {\bf Problem i} is characterized by the parameter $\theta^{(i)}$, where $\theta^{(i)} = \left({\gamma^{(i)}\atop \omega^{(i)}}\right)$ is defined as:
\begin{eqnarray*}
    \gamma^{(1)} & =& \frac{1+ \Delta_{\neq} - \Delta_{\min}}{2}\left(\sum_{1\leq j \leq \lfloor d/2\rfloor} e_j\right) +\frac{1 - \Delta_{\neq}-\Delta_{\min}}{2} \left(\sum_{\lfloor d/2\rfloor + 1 \leq j \leq d} e_j\right) + \Delta_{\min}e_1 + \Delta_{\min}e_{\lfloor d/2\rfloor +1}\\
    \gamma^{(i)}  &=& \gamma^{(1)} + 2\Delta_{\min}e_{i} + 2\Delta_{\min}e_{\lfloor d/2\rfloor +i} \ \ \ \forall i \in \{2, ..., \lfloor d/2\rfloor\}\\
    \gamma^{(\lfloor d/2\rfloor + 1)}  &=& \frac{1- \Delta_{\neq} - \Delta_{\min}}{2}\left(\sum_{1\leq j \leq \lfloor d/2\rfloor} e_j\right) +\frac{1 + \Delta_{\neq}-\Delta_{\min}}{2} \left(\sum_{\lfloor d/2\rfloor + 1 \leq j \leq d} e_j \right)+ \Delta_{\min}e_1 + \Delta_{\min}e_{\lfloor d/2\rfloor +1},
\end{eqnarray*}
and the bias parameters are defined as $\omega^{(i)} = -\frac{\Delta_{\neq}}{2}$ $\forall i \in \{1, ..., \lfloor d/2\rfloor\}$, and otherwise $\omega^{(\lfloor d/2\rfloor+1)} = \frac{\Delta_{\neq}}{2}$. We write $\mathbb E^{(i)}, \mathbb P^{(i)}, R^{(i)}_T$ for resp.~the probability, expectation, and regret, in {\bf Problem i}. Note that this choice of parameters ensures that $\forall i \in \{1, ..., \lfloor d/2\rfloor +1\}$, $\theta^{(i)} \in \mathbf{\Theta}^{\mathcal{A}}_{\Delta_{\min}, \Delta_{\neq}}$.

Set $\mathcal{A} =\left\{\left({x_1 \atop z_{x_1}}\right), ...,\left({x_{d+1} \atop z_{x_{d+1}}}\right) \right\}$, where $\left({x_i \atop z_{x_i}}\right) = e_i + e_{d+1}$, for $i \in \{2, ..., \lfloor d/2\rfloor\}$, $\left({x_i \atop z_{x_i}}\right) = e_i - e_{d+1}$ for $i \in \{\lfloor d/2\rfloor +1, ..., d\}$, and $\left({x_{d+1} \atop z_{x_{d+1}}}\right) = -\left(1-\frac{2}{\sqrt{\kappa_*}+1} \right)e_1 - e_{d+1}$. Then, Lemma \ref{lem:kappa_2} shows that $\cA \in \mathbf{A}_{\kappa_*, d}$.
\begin{lem} \label{lem:kappa_2}
It holds that
\begin{equation*}
   \underset{\pi \in \mathcal{P}^{\mathcal{A}}_{e_{d+1}}}{\min}\left\{ e_{d+1}^{\top}\left(\underset{\left({x \atop z}\right) \in \mathcal{A}}{\displaystyle \sum}\pi(x) \left({x \atop z_x}\right) \left({x \atop z_x}\right)^{\top} \right)^+e_{d+1}\right\} = \kappa_*.
\end{equation*}
\end{lem}

The following facts hold:
\begin{enumerate}
\item[{\bf Fact 1}] For any $i \in \{1, ..., \lfloor d/2\rfloor +1\}$, action $x_i$ is the unique optimal action in {\bf Problem i}. Since $1/2 \geq \Delta_{\neq} \geq \Delta_{\min}$, sampling any other (sub-optimal) action leads to an instantaneous regret of at least $\Delta_{\min}$. Moreover, choosing an action in the group $-z_i$ leads to an instantaneous regret of at least $\Delta_{\neq}$.
\item[{\bf Fact 2}] In {\bf Problem i} for any $i \in \{1, ..., \lfloor d/2\rfloor +1\}$, action $d+1$ is very sub-optimal and sampling it leads to an instantaneous regret higher than $(1 - 2/(\sqrt{\kappa_*} +1))(1-\Delta_{\neq}+\Delta_{\min}) + (1+\Delta_{\neq}+\Delta_{\min})/2 \geq 1/2$, since $\kappa_* \geq 1$ and $1/2 \geq \Delta_{\neq} \geq \Delta_{\min}$.
\item[{\bf Fact 3}] In {\bf Problem i}, for $i \in \{1, ..., \lfloor d/2\rfloor +1\}$, when sampling action $x_j$ at time, $t$ the distribution of the observation does not depend on $t$ or on the past (except through the choice of $x_j$) and is $\mathbb P_j^{(i)}$. It is characterized as:
\begin{align*}
&\forall i \in\{1, ...,\lfloor d/2\rfloor+1\}, \mathbb P_1^{(i)},  \mathbb P_{\lfloor d/2\rfloor+1}^{(i)}~~~\mathrm{are}~~~\mathcal N((1+\Delta_{\min})/2,1)\\
&\forall i \in\{1, ..., \lfloor d/2\rfloor+1\},\forall j \in \{2, ..., d\}\setminus\{\lfloor d/2\rfloor+1, i, \lfloor d/2\rfloor+i\}, \mathbb P_j^{(i)}~~~\mathrm{is}~~~\mathcal N((1-\Delta_{\min})/2,1),\\
&\forall i \in\{2, \lfloor d/2\rfloor\}, \mathbb P_i^{(i)}~~~\mathrm{is}~~~\mathcal N((1+3\Delta_{\min})/2,1)~~~~\mathbb P_{\lfloor d/2\rfloor+i}^{(i)}~~~\mathrm{is}~~~\mathcal N((1+3\Delta_{\min})/2,1)\\
&\forall i \in\{1, \lfloor d/2\rfloor\}, \mathbb P_{d+1}^{(i)}~~~\mathrm{is}~~~~\mathcal N(-(1-\alpha)(1+\Delta_{\neq}+\Delta_{\min})/2 + \Delta_{\neq}/2,1),\\
&\mathbb P_{d+1}^{(\lfloor d/2\rfloor+1)}~~~~\mathrm{is}~~~~\mathcal N(-(1-\alpha)(1-\Delta_{\neq}+\Delta_{\min})/2 - \Delta_{\neq}/2,1)~~~\mathrm{where}~~\alpha = 2/(\sqrt{\kappa_*} +1).
\end{align*}
So that:
\begin{itemize}
\item[{\bf Fact 3.1}] For any $i \in \{2, ..., \lfloor d/2\rfloor \}$, between {\bf Problem 1} and {\bf Problem i}, the only actions that provide different evaluations when sampled are action $i$ and action $\lfloor d/2\rfloor+i$, and the mean gaps in both cases is $2\Delta_{\min}$.
\item[{\bf Fact 3.2}] Between {\bf Problem 1} and {\bf Problem $\lfloor d/2\rfloor +1$}, the only action that provide different evaluation when sampled is action $d+1$, and the mean gap in this case is $\alpha\Delta_{\neq}$.
\end{itemize} 
\end{enumerate}

For $j \leq d+1$, we write $N_{j}(T)$ for the total number of times action $x_j$ has been selected before time $T$.  Then, for $j \in \{1, ..., \lfloor d/2\rfloor\}$, let $E^{(j)} = \left\{N_{i}(T) \leq T/2  \right\}$. Note that for $i \in \{1, ..., \lfloor d/2\rfloor\}$, in {\bf Problem i} the action $x_i$ is the optimal action. Therefore, for any efficient algorithm, for all $i \in \{1, ..., \lfloor d/2\rfloor\}$ the event $E^{(i)}$ should have a low probability under $\mathbb{P}^{(i)}$. Indeed,  for $i \in \{1, ..., \lfloor d/2\rfloor\}$, the regret of the algorithm under {\bf Problem i} 
can be lower-bounded as follows - see {\bf Facts 1 and 2}:
\begin{eqnarray}\label{eq:lbregret}
    R_T^{(i)}
    &\geq&\underset{j \leq \lfloor d/2\rfloor,\ j \neq i}{\sum} \mathbb{E}^{(i)}\left[N_{j}(T)\right] \Delta_{\min} +  \underset{\lfloor d/2\rfloor + 1\leq  j \leq d}{\sum}\mathbb{E}^{(i)}\left[N_{j}(T)\right]\Delta_{\neq} + \frac{\mathbb{E}^{(i)}\left[N_{d + 1}(T)\right]}{2}.
\end{eqnarray} 
Since $\sum_{j}\mathbb{E}^{(i)}\left[N_j(T)\right] = T$ and $\Delta_{\min}\leq \Delta_{\neq}\leq \frac{1}{2}$, this implies together with {\bf Facts 1}:
\begin{eqnarray*}
     R_T^{(i)} &\geq& \left(T -  \mathbb{E}^{(i)}\left[N_{i}(T)\right]\right)\Delta_{\min}
\end{eqnarray*}
Using the definition of $E^{(i)}$, we find that
\begin{eqnarray}
     R_T^{(i)} &\geq& \frac{T\Delta_{\min}}{2} \mathbb{P}^{(i)}\left(E^{(i)}\right).\label{eq:Ri}
\end{eqnarray}
In particular for {\bf Problem 1}, for any $i \in \{1, ..., \lfloor d/2\rfloor\}$, 
\begin{align}
    R_T^{(1)} &\geq \frac{T\Delta_{\min}}{2} \mathbb{P}^{(1)}\left(\overline{E^{(i)}}\right).\label{eq:R1}
\end{align}
since $E^{(1)} \supset \overline{E^{(i)}}$.

\medskip
 Similarly, let us also define the event $F = \left\{ \underset{i\leq \lfloor d/2\rfloor}{\sum}N_i(T)\geq T/2\right\}$. Then, in {\bf Problem 1}, the group $1$ contains the optimal action, and so for any efficient algorithm, the event $F$ should have a low probability under $\mathbb{P}^{(1)}$. Indeed, Equation \eqref{eq:lbregret} also implies
\begin{eqnarray}\label{eq:R1d}
    R_T^{(1)}&\geq& \left(T- \mathbb{E}^{(1)}\left[\underset{i\leq \lfloor d/2\rfloor}{\sum} N_{i}(T)\right]\right)\Delta_{\neq}\geq \frac{T\Delta_{\neq}}{2}\mathbb{P}^{(1)}\left(\overline{F} \right).
\end{eqnarray}
On the other hand,  for any efficient algorithm, the event $F$ should have high probability under $\mathbb{P}^{(\lfloor d/2\rfloor +1)}$. Indeed,under problem \textbf{Problem }$\mathbf{\lfloor d/2\rfloor +1}$, the regret can be lower-bounded as follows - see {\bf Facts 1 and 2}:
\begin{eqnarray*}
    R_T^{(\lfloor d/2 \rfloor+1)}
    &\geq& \underset{j \leq \lfloor d/2\rfloor}{\sum} \mathbb{E}^{(\lfloor d/2 \rfloor+1)}\left[N_{j}(T)\right] \Delta_{\neq} +  \underset{\lfloor d/2\rfloor + 2\leq  j \leq d }{\sum}\mathbb{E}^{(\lfloor d/2 \rfloor+1)}\left[N_{j}(T)\right]\Delta_{\min} + \frac{\mathbb{E}^{(\lfloor d/2 \rfloor+1)}\left[N_{d + 1}(T)\right]}{2}.
\end{eqnarray*}
which implies that
\begin{eqnarray}\label{eq:Rd}
    R_T^{(\lfloor d/2 \rfloor+1)}
    &\geq& \underset{j \leq \lfloor d/2\rfloor}{\sum} \mathbb{E}^{(\lfloor d/2 \rfloor+1)}\left[N_{j}(T)\right] \Delta_{\neq} \geq \frac{T\Delta_{\neq}}{2}\mathbb{P}^{(\lfloor d/2 \rfloor+1)}\left(F \right).
\end{eqnarray}

\medskip

Now, Bretagnolle-Huber inequality (see, e.g., Theorem 14.2 in \cite{BanditBook}) implies that for all $i \in \{2, ..., \lfloor d/2\rfloor\}$,
\begin{eqnarray}\label{eq:bretagnollei}
    \frac{1}{2}\exp \left( - KL\left(\mathbb{P}^{(1)}, \mathbb{P}^{(i)}\right)\right) & \leq & \mathbb{P}^{(i)}\left(E^{(i)}\right) + \mathbb{P}^{(1)}\left(\overline{E^{(i)}}\right)
\end{eqnarray}
and that
\begin{eqnarray}\label{eq:bretagnolled}
    \frac{1}{2}\exp \left( - KL\left(\mathbb{P}^{(1)}, \mathbb{P}^{( \lfloor d/2\rfloor +1)}\right)\right) & \leq & \mathbb{P}^{( \lfloor d/2\rfloor+1)}\left(F\right) + \mathbb{P}^{(1)}\left(\overline{F}\right).
\end{eqnarray}
On the one hand, Equation \eqref{eq:bretagnollei} implies that for any $i \in \{2, ..., \lfloor d/2\rfloor\}$,
\begin{eqnarray}
    KL\left(\mathbb{P}^{(1)}, \mathbb{P}^{(i)}\right)
    & \geq & - \log \left(2 \mathbb{P}^{(i)}\left(E^{(i)}\right) + 2\mathbb{P}^{(1)}\left(\overline{E^{(i)}}\right)\right)\nonumber\\
    & \geq &  \log\left(T\right) - \log\left(2T\mathbb{P}^{(i)}\left(E^{(i)}\right) + 2T\mathbb{P}^{(1)}\left(\overline{E^{(i)}}\right)\right)\label{eq:BHKli}.
\end{eqnarray}
Combining Equations \eqref{eq:Ri}, \eqref{eq:R1}, and \eqref{eq:BHKli}, we find that
\begin{eqnarray}\label{eq:KLi}
    KL\left(\mathbb{P}^{(1)}, \mathbb{P}^{(i)}\right) & \geq & \log\left(T\right) - \log\left(\frac{4(R_T^{(i)} + R_T^{(1)})}{ \Delta_{\min}}\right).
\end{eqnarray}
On the other hand, Equation \eqref{eq:bretagnolled} implies that 
\begin{eqnarray}
    KL\left(\mathbb{P}^{(1)}, \mathbb{P}^{( \lfloor d/2\rfloor +1)}\right)
    & \geq & - \log \left(2 \mathbb{P}^{(\lfloor d/2\rfloor +1)}\left(F\right) + 2\mathbb{P}^{(1)}\left(\overline{F}\right)\right)\nonumber\\
    & \geq &  \log\left(T\right) - \log\left(2T\mathbb{P}^{(\lfloor d/2\rfloor +1)}\left(F\right) + 2T\mathbb{P}^{(1)}\left(\overline{F}\right)\right)\label{eq:BHKld}.
\end{eqnarray}
Combining Equations \eqref{eq:Ri}, \eqref{eq:R1}, and \eqref{eq:BHKld}, we find that
\begin{eqnarray}\label{eq:KLd}
    KL\left(\mathbb{P}^{(1)}, \mathbb{P}^{(\lfloor d/2\rfloor +1)}\right) & \geq & \log\left(T\right) - \log\left(\frac{4(R_T^{(\lfloor d/2\rfloor +1)} + R_T^{(1)})}{ \Delta_{\neq}}\right).
\end{eqnarray}

Also, note that for all $i \in \{2, ..., \lfloor d/2\rfloor+1\}$, the Kullback-Leibler divergence between $\mathbb{P}^{(1)}$ and $\mathbb{P}^{(i)}$ can be decomposed as follows (see, e.g., Lemma 15.1 in \cite{BanditBook}) :
\begin{eqnarray}\label{eq:KL_decomposition}
    KL(\mathbb{P}^{(1)}, \mathbb{P}^{(i)}) &=& \underset{j \leq d+1}{\sum} \mathbb{E}^{(1)}\left[N_{j}(T)\right]KL(\mathbb{P}^{(1)}_{j}, \mathbb{P}^{(i)}_{j}).
\end{eqnarray}

\paragraph{\textbf{Lower bound in} $d\Delta_{\min}^{-1}\log T $.}\label{par:low_bound_linear}
By design, for $i \in  \{2, ..., \lfloor d/2\rfloor\}$, all actions but $x_i$ and $x_{\lfloor d\rfloor+i}$ have the same distribution under $\mathbb{P}^{(1)}$ and $\mathbb{P}^{(i)}$ - see {\bf Fact 3.1}. Then, Equation \eqref{eq:KL_decomposition} becomes from {\bf Fact 3.1} and from the expression of KL divergence between standard Gaussian distributions:
\begin{eqnarray*}
     KL(\mathbb{P}^{(1)}, \mathbb{P}^{(i)})&=& \frac{4\Delta_{\min}^2}{2}\mathbb{E}^{(1)}\left[N_{i}(T)\right] +  \frac{4\Delta_{\min}^2}{2}  \mathbb{E}^{(1)}\left[N_{\lfloor{d}\rfloor+i}(T)\right].
\end{eqnarray*}
So that, summing over $i \in  \{2, ..., \lfloor d/2\rfloor\}$, and by {\bf Fact 1}:
\begin{eqnarray*}
     \sum_{i \in  \{2, ..., \lfloor d/2\rfloor\}} KL(\mathbb{P}^{(1)}, \mathbb{P}^{(i)})&\leq& 2\Delta_{\min} R_T^{(1)}.
\end{eqnarray*}
So that by Equation~\eqref{eq:KLi} (summing over $i \in  \{2, ..., \lfloor d/2\rfloor\}$):
\begin{eqnarray*}
    2\Delta_{\min} R_T^{(1)} &\geq & \sum_{i \in  \{2, ..., \lfloor d/2\rfloor\}}\left[\log\left(T\right) - \log\left(\frac{4(R_T^{(i)} + R_T^{(1)})}{\Delta_{\min}}\right)\right]\\
    &=& (\lfloor d/2\rfloor - 1)\log\left(T\right) - \sum_{i \in  \{2, ..., \lfloor d/2\rfloor\}}\log\left(\frac{4(R_T^{(i)} + R_T^{(1)})}{\Delta_{\min}}\right).
\end{eqnarray*}
Let us assume that our algorithm satisfies $\max_{i \leq \lfloor d/2\rfloor} R_T^{(i)} \leq  \frac{d\log \left(T\right)}{\Delta_{\min}}$ - otherwise the bound immediately follows for this algorithm. Then
\begin{eqnarray}
    R_T^{(1)} &\geq & \frac{1}{2\Delta_{\min}}(\lfloor d/2\rfloor - 1)\log\left(T\right) - \frac{1}{2\Delta_{\min}}\sum_{i \in  \{2, ..., \lfloor d/2\rfloor\}}\log\left(\frac{8 d \log T}{\Delta_{\min}^2}\right) \nonumber\\
    &\geq & \frac{1}{2\Delta_{\min}}(\lfloor d/2\rfloor - 1)\left[\log\left(T\right) -\log\left(\frac{8 d \log \left(T\right)}{\Delta_{\min}^2}\right)\right].\label{eq:fin_lb_d}
\end{eqnarray}
Sine
$d \geq 4$, we note that $\lfloor d/2\rfloor - 1 \geq d/5$. This concludes the proof for this part of the bound.

\paragraph{\textbf{Lower bound in} $\kappa_*\Delta_{\neq}^{-2}\log T$.}
By design, all actions but $x_{d+1}$ have the same evaluation under {\bf Problem 1} and {\bf Problem $\lfloor d/2\rfloor +1$} - see {\bf Fact 3.2}. Then, by {\bf Fact 3.2} and the expression between the KL divergence of standard Gaussians, Equation \eqref{eq:KL_decomposition} becomes
\begin{eqnarray*}
     KL(\mathbb{P}^{(1)}, \mathbb{P}^{(\lfloor d/2\rfloor +1)})&=& \mathbb{E}^{(1)}\left[N_{d +1}(T)\right]\frac{\left(\alpha\Delta_{\neq}\right)^2}{2} = \frac{1}{2}\mathbb{E}^{(1)}\left[N_{d+1}(T)\right]\left(\frac{2\Delta_{\neq}}{\sqrt{\kappa_*}+1}\right)^2.
\end{eqnarray*}
Combined with equation \eqref{eq:KLd}, this implies that 
\begin{eqnarray}\label{eq:lb_nolog}
    \frac{1}{2}\mathbb{E}^{(1)}\left[N_{d+1}(T)\right]\left(\frac{2\Delta_{\neq}}{\sqrt{\kappa_*}+1}\right)^2 & \geq & \log\left(T\right) - \log\left(\frac{4(R_T^{(\lfloor d/2\rfloor +1)} + R_T^{(1)})}{ \Delta_{\neq}}\right).
\end{eqnarray}
Let us assume that our algorithm satisfies $\max_{i \leq \lfloor d/2\rfloor+1} R_T^{(i)} \leq  \frac{\kappa_*\log \left(T\right)}{\Delta_{\neq}^2}$ - otherwise the bound immediately follows for this algorithm. We then have
\begin{eqnarray*}
    \frac{1}{2}\mathbb{E}^{(1)}\left[N_{d+1}(T)\right]\left(\frac{2\Delta_{\neq}}{\sqrt{\kappa_*}+1}\right)^2 & \geq & \log\left(T\right) - \log\left( \frac{8\kappa_*\log \left(T\right)}{\Delta_{\neq}^3}\right).
\end{eqnarray*}
Using Equation \eqref{eq:lbregret}, we find that
\begin{eqnarray}\label{eq:fin_lb_z}
    R_T^{(1)} & \geq & \frac{\kappa_*+1}{4\Delta_{\neq}^2} \left[\log\left(T\right) - \log\left(\frac{8\kappa_*\log \left(T\right)}{\Delta_{\neq}^3}\right)\right].
\end{eqnarray}

\paragraph{\textbf{Lower bound in} $\kappa_*\Delta_{\neq}^{-2}$.}
Let us assume that our algorithm satisfies $\max_{i \leq \lfloor d/2\rfloor+1} R_T^{(i)} \leq  \frac{\kappa_*}{\Delta_{\neq}^2}$ - otherwise the bound immediately follows for this algorithm. Then, Equation \eqref{eq:lb_nolog} implies
\begin{eqnarray*}
    \frac{1}{2}\mathbb{E}^{(1)}\left[N_{d+1}(T)\right]\left(\frac{2\Delta_{\neq}}{\sqrt{\kappa_*}}\right)^2 & \geq & \log\left(T\right) - \log\left( \frac{8\kappa_*}{\Delta_{\neq}^3}\right).
\end{eqnarray*}
Using again Equation \eqref{eq:lbregret}, we find that
\begin{eqnarray}\label{eq:fin_lb_z_wc}
    R_T^{(1)} & \geq & \frac{\kappa_*+1}{4\Delta_{\neq}^2}\log\left(\frac{T\Delta_{\neq}^3}{8\kappa_*}\right).
\end{eqnarray}
We conclude the proof of Theorem \ref{thm:lower_bound} by combining Equations \eqref{eq:fin_lb_d}, \eqref{eq:fin_lb_z} and \eqref{eq:fin_lb_z_wc}.

\paragraph{Bounds on $\kappa(\Delta)$} Finally, the following lemma allows to express $\kappa(\Delta)$ as a function of $\kappa_*$.
\begin{lem}\label{lem:borne_inf_kappa}
For any $i \in \{1, ..., \lfloor d/2\rfloor +1\}$, the gap vector $\Delta$ verifies $$\kappa(\Delta)= \frac{(1+\sqrt{\kappa_*})^2\Delta_{d+1}}{4}$$
where $\Delta_{d+1} = \max_{i}(x_i - x_{d+1})^{\top}\gamma^{(i)}.$
\end{lem}
On the one hand, since $\kappa_* \geq 1$, we see that $\kappa_* \leq (1+\sqrt{\kappa_*})^2 \leq 4\kappa_*$. On the other hand, $1/2 \leq \Delta_{d+1}\leq 2$, so $\kappa(\Delta) \in \left[\frac{\kappa_*}{8}, 2\kappa_*\right]$.


\subsection{Extension of the gap-dependent lower bounds to $d=2,3$} \label{app:lower_bound_gap_2}

Theorem \ref{thm:lower_bound_gap} can be extended to $d\in \{2,3\}$ by considering separately the cases $\frac{d}{\Delta_{\min}} \geq \frac{\kappa}{\Delta_{\neq}^2}$ and $\frac{d}{\Delta_{\min}} <\frac{\kappa}{\Delta_{\neq}^2}$.

\paragraph{Case 1 : $\frac{d}{\Delta_{\min}} \geq \frac{\kappa}{\Delta_{\neq}^2}$}
Let us consider the set of actions defined by $\mathcal{A} =\left\{\left({x_1 \atop z_{x_1}}\right), ...,\left({x_{d+1} \atop z_{x_{d+1}}}\right) \right\}$, where $\left({x_i \atop z_{x_i}}\right) = e_1 + e_{d+1}$ for $i \in \{1, ..., d\}$, and $\left({x_{d+1} \atop z_{x_{d+1}}}\right) = -\left(1-\frac{2}{\sqrt{\kappa_*}+1} \right)e_1 - e_{d+1}$. Using the same proof as in Lemma \ref{lem:kappa}, we see that 
\begin{equation*}
   \underset{\pi \in \mathcal{P}^{\mathcal{A}}}{\min}\left\{ e_{d+1}^{\top}\left(\underset{\left({x \atop z}\right) \in \mathcal{A}}{\displaystyle \sum}\pi_x \left({x \atop z_x}\right) \left({x \atop z_x}\right)^{\top} \right)^+e_{d+1}\right\} = \kappa.
\end{equation*}

Then, we consider the following problems : for $i\leq d$, {\bf Problem i} is characterized by the parameter $\theta^{(i)}$, where $\theta^{(i)} = \left({\gamma^{(i)}\atop \omega^{(i)}}\right)$ is defined as:
\begin{eqnarray*}
    \gamma^{(1)} & =& \frac{1-\Delta_{\min}}{2}\underset{i \leq d}{\sum}e_i + \Delta_{\min} e_1\\
    \gamma^{(i)} & =& \frac{1-\Delta_{\min}}{2}\underset{i \leq d}{\sum}e_i+\Delta_{\min} e_1+\Delta_{\min} e_i \quad \text{ for i >1}
\end{eqnarray*}
and the bias parameters are defined as $\omega^{(i)} = 0$ for $i\leq d$. The following facts hold:
\begin{enumerate}
\item[{\bf Fact 1}] For any $i \in \{1, ..., d\}$, action $x_i$ is the unique optimal action in {\bf Problem i}. Sampling any other (sub-optimal) action leads to an instantaneous regret of at least $\Delta_{\min}$. 
\item[{\bf Fact 2}] In {\bf Problem i}, for $i \in \{1, ..., d\}$, when sampling action $x_j$ at time, $t$ the distribution of the observation does not depend on $t$ or on the past (except through the choice of $x_j$) and is $\mathbb P_j^{(i)}$. It is characterized as:
\begin{align*}
&\forall i \in\{1, ...,d\}, \mathbb P_1^{(i)} ~~~\mathrm{is}~~~\mathcal N((1+\Delta_{\min})/2,1)\\
&\forall i \in\{1, ...,d\}, \mathbb P_{d+1}^{(1)}~~~\mathrm{is}~~~~\mathcal N(-(1-\frac{2}{\sqrt{\kappa_*}+1})(1+\Delta_{\min})/2,1)\\
&\forall i \in\{2, ...,d\}, \mathbb P_i^{(i)} ~~~\mathrm{is}~~~\mathcal N((1+3\Delta_{\min})/2,1)\\
&\forall i,j \in\{2, ...,d\}, i\neq j : \mathbb P_j^{(i)} ~~~\mathrm{is}~~~\mathcal N((1-\Delta_{\min})/2,1)\\
\end{align*}
So that for any $i \in \{2, ..., d \}$, between {\bf Problem 1} and {\bf Problem i}, the only action that provides different evaluations when sampled is action $i$ , and the mean gap is $2\Delta_{\min}$.
\end{enumerate}

Since $\Delta_{\neq}\leq \frac{1}{8}$, this choice of parameters ensures that $\forall i \in \{1, ..., d\}$, $\theta^{(i)} \in \mathbf{\Theta}^{\mathcal{A}}_{\Delta_{\min}, \Delta_{\neq}, \kappa_*}$. Adapting the proof of Lemma \ref{lem:kappa}, we note that the minimal variance of bias estimation is at least $\kappa_*$.
This proves that $\cA \in \mathbf{\Theta}^{\mathcal{A}}_{\Delta_{\min}, \Delta_{\neq}, \kappa_*}$. Now, the lower bound $$R_T \geq \frac{d-1}{2\Delta_{\min}}\left[\log\left(T\right) -\log\left(\frac{8 d \log \left(T\right)}{\Delta_{\min}^2}\right)\right]$$ follows directly using arguments from the proof of Theorem \ref{thm:lower_bound}.

\paragraph{Case 2 : $\frac{d}{\Delta_{\min}} > \frac{\kappa}{\Delta_{\neq}^2}$}
Let the action set be given by $\mathcal{A} =\left\{\left({x_1 \atop z_{x_1}}\right), ...,\left({x_{d+1} \atop z_{x_{d+1}}}\right) \right\}$, where $\left({x_1 \atop z_{x_1}}\right) = e_1 + e_{d+1}$, $\left({x_i \atop z_{x_i}}\right) = e_i - e_{d+1}$ for $i \in \{2, ..., d\}$, and $\left({x_{d+1} \atop z_{x_{d+1}}}\right) = -\left(1-\frac{2}{\sqrt{\kappa_*}+1} \right)e_1 - e_{d+1}$. By Lemma \ref{lem:kappa}, $\mathcal{A}\in \mathbf{A}_{\kappa_*,d}$. We consider two bandit problems characterized by two parameters $\theta^{(1)}$ and $\theta^{(2)}$, defined as:
\begin{eqnarray*}
    \gamma^{(1)} & =& \frac{1+\Delta_{\neq}}{2}e_1 + \frac{1 - \Delta_{\neq}}{2}e_{2} - \frac{\Delta_{\neq}}{2} e_3 \\
    \gamma^{(2)} & =& \frac{1-\Delta_{\neq}}{2}e_1 + \frac{1 + \Delta_{\neq}}{2}e_{2} + \frac{\Delta_{\neq}}{2} e_3.
\end{eqnarray*}
On top of this, two bias parameters are defined as $\omega^{(1)} = -\frac{\Delta_{\neq}}{2}$ and $\omega^{(2)} = \frac{\Delta_{\neq}}{2}$.

The following facts hold:
\begin{enumerate}
\item[{\bf Fact 1}] For any $i \in \{1, 2\}$, action $x_i$ is the unique optimal action in {\bf Problem i}. Since $1/2 \geq \Delta_{\neq}$, sampling any other (sub-optimal) action leads to an instantaneous regret of at least $\Delta_{\neq}$.
\item[{\bf Fact 2}] In {\bf Problem i}, for $i \in \{1, ..., d\}$, when sampling action $x_j$ at time, $t$ the distribution of the observation does not depend on $t$ or on the past (except through the choice of $x_j$) and is $\mathbb P_j^{(i)}$. It is characterized as:
\begin{align*}
&\forall i \in\{1, 2\}, \forall j \in\{1, 2\}, \mathbb P_j^{(i)} ~~~\mathrm{is}~~~\mathcal N(1/2,1)\\
&\forall i \in\{1, 2\}, \mathbb P_{3}^{(1)}~~~\mathrm{is}~~~~\mathcal N(0,1)\\
&\mathbb P_{d+1}^{(1)} ~~~\mathrm{is}~~~\mathcal N\left(\left(1-\frac{2}{\sqrt{\kappa_*}+1} \right)\left(\frac{1+\Delta_{\neq}}{2}\right) + \frac{\Delta_{\neq}}{2},1\right)\\
&\mathbb P_{d+1}^{(2)} ~~~\mathrm{is}~~~\mathcal N\left(\left(1-\frac{2}{\sqrt{\kappa_*}+1} \right)\left(\frac{1-\Delta_{\neq}}{2}\right) - \frac{\Delta_{\neq}}{2},1\right)
\end{align*}
So that, between {\bf Problem 1} and {\bf Problem 2}, the only action that provides different evaluations when sampled is action $1$, and the mean gaps in both cases is $\frac{2\Delta_{\neq}}{\sqrt{\kappa_*}+1}$.
\end{enumerate}
Note that the minimum gap for these parameters is $\Delta_{\neq}\geq \Delta_{\min}$. Thus, this choice of parameters ensures that $\forall i \in \{1, ..., d\}$, $\theta^{(i)} \in \mathbf{\Theta}^{\mathcal{A}}_{\Delta_{\min}, \Delta_{\neq}, \kappa_*}$. Adapting the proof of Lemma \ref{lem:kappa}, we note that the minimal variance of bias estimation is at least $\kappa_*$.This proves that $\cA \in \mathbf{\Theta}^{\mathcal{A}}_{\Delta_{\min}, \Delta_{\neq}, \kappa_*}$. Then, the lower bound $$R_T \geq \frac{\kappa_*+1}{4\Delta_{\neq}^2} \left[\log\left(T\right) - \log\left(\frac{8\kappa_*\log \left(T\right)}{\Delta_{\neq}^3}\right)\right].$$ follows directly using arguments from the proof of Theorem \ref{thm:lower_bound}.
\subsection{Auxiliary Lemmas}


\subsubsection{Proof of Lemma \ref{lem:upsilon_margin}}

Lemma \ref{lem:upsilon_margin} follows from the characterization of $\kappa_*$ given in Lemma \ref{lem:upsilon}.  We begin by proving the first statement. Assume that $\kappa_*>1$ (otherwise the first statement is void). Note that for all $u \in \mathbb{R}^d$, $\underset{\lambda \rightarrow +\infty}{\text{lim}} (\max_{x \in \cX} \left(x^{\top}(\lambda u) + z_x\right)^2)^{-1} = 0,$ so the minimum over $u \in \mathbb{R}^d$ of $(\max_{x \in \cX} \left(x^{\top}(\lambda u) + z_x\right)^2)^{-1}$ is attained for some vector $\tilde{u}\in \mathbb{R}^d$. Since $\kappa_* >1$, $\tilde{u}$ is not null. Moreover, $\max_{x\in\cX}(1 + z_x x^{\top}\tilde{u})^2<1$, so $\max_{x\in\cX}z_x x^{\top}\tilde{u} <0$. Thus, for all $x\in \cX$,  $x^{\top}\tilde{u}$ and $z_x$ are of opposite sign, and $ x^{\top}\tilde{u} \neq 0$. This implies that the hyperplane containing 0 with normal vector $\tilde{u}$ contains no action, and separates the two groups. Moreover,
$$\kappa_*^{-1/2} = \max_{x \in \cX} \vert z_xx^{\top}\tilde{u} + 1\vert.$$

We denote $x^{(1)}\in \argmax_{x \in \cX} z_{z}x^{\top}\tilde{u}$, and $x^{(2)}\in \argmin_{x \in \cX} z_{z}x^{\top}\tilde{u}$. Let us show that $(z_{x^{(1)}}{x^{(1)}}^{\top}\tilde{u}+1) = -\left(1 + z_{x^{(2)}} {x^{(2)}}^{\top}\tilde{u}\right)$, i.e that $z_{x^{(1)}}{x^{(1)}}^{\top}\tilde{u}+ z_{x^{(2)}} {x^{(2)}}^{\top}\tilde{u} = -2$. Indeed, note that $$ \kappa_*^{-1/2} =( z_{x^{(1)}}{x^{(1)}}^{\top}\tilde{u}+1) \vee -(1 + z_{x^{(2)}} {x^{(2)}}^{\top}\tilde{u}).$$
Then, for $u' = \frac{-2}{\left(z_{x^{(1)}}{x^{(1)}}+ z_{x^{(2)}} {x^{(2)}}\right)^{\top}\tilde{u}}\tilde{u}$, we see that $$z_{x^{(1)}}{x^{(1)}}^{\top}u' +1 =  - \left(1 + z_{x^{(2)}}{x^{(2)}}^{\top}u'\right) = \max_{x \in \cX}\vert z_x x^{\top}u' +1 \vert.$$

By contradiction, let us first assume that $z_{x^{(1)}}{x^{(1)}}^{\top}\tilde{u}+ z_{x^{(2)}} {x^{(2)}}^{\top}\tilde{u} < -2$. Then, 
$$\max_{x \in \cX}\vert z_x x^{\top}u' +1 \vert = z_{x^{(1)}}{x^{(1)}}^{\top}u' +1 < z_{x^{(1)}}{x^{(1)}}^{\top}\tilde{u} +1 = \kappa_*^{-1/2}$$
which contradicts the definition of $\kappa_*$.

Similarly, if we assume that  $z_{x^{(1)}}{x^{(1)}}^{\top}\tilde{u}+ z_{x^{(2)}} {x^{(2)}}^{\top}\tilde{u} > -2$, then 
$$\max_{x \in \cX}\vert z_x x^{\top}u' +1 \vert =  -(z_{x^{(2)}}{x^{(2)}}^{\top}u' +1 )< -(z_{x^{(2)}}{x^{(2)}}^{\top}\tilde{u} +1) = \kappa_*^{-1/2}$$
which contradicts again the definition of $\kappa_*$. Therefore, $$(z_{x^{(1)}}{x^{(1)}}^{\top}\tilde{u}+1) = -\left(1 + z_{x^{(2)}} {x^{(2)}}^{\top}\tilde{u}\right) = \kappa_*^{-1/2}.$$ Then, the hyperplane containing $0$ with normal vector $\tilde u$ separates the actions of the two groups. Moreover, the margin is $-z_{x^{(1)}}{x^{(1)}}^{\top}\tilde{u} = 1-\kappa_*^{-1/2}$, while the maximum distance of all points is $-z_{x^{(2)}}{x^{(2)}}^{\top}\tilde{u} = 1+\kappa_*^{-1/2}$. Thus, there exists $\tilde{u}$ such that the hyperplane containing $0$ with normal vector $\tilde{u}$ separates the actions of the two groups, with margin equal to $\frac{\sqrt{\kappa_*}-1}{\sqrt{\kappa_*}+1}$ times the maximum distance of all points to the hyperplane.

Conversely, assume that there exists $\kappa>\kappa_*$ such that there exists $u\in \mathbb{R}^d$ such that the hyperplane containing $0$ with normal vector $u$ separates the actions of the two groups, with margin equal to $\frac{\sqrt{\kappa}-1}{\sqrt{\kappa}+1} = \frac{1-\kappa^{-1/2}}{1+\kappa^{-1/2}}$ times the maximum distance of all points to the hyperplane, denoted hereafter $d$. Since the hyperplane separates the points, we can assume without loss of generality that for all $x \in \cX$, $z_x x^{\top}u < 0$. Similarly, up to a renormalization, we can assume without loss of generality that $d = 1+ \kappa^{-1/2}$. Then, 
\begin{eqnarray*}
  \max_{x\in \cX}\vert z_x x^{\top}u + 1\vert &=&  (\max_{x \in \cX}z_x x^{\top}u +1) \vee -(\min_{x \in \cX}z_x x^{\top}u +1) \\
  &=& \left(-\frac{1-\kappa^{-1/2}}{1+\kappa^{-1/2}} \times (1+ \kappa^{-1/2}) + 1\right) \vee -(1-\kappa^{-1/2} - 1) = \kappa^{-1/2}<\kappa_*^{-1/2}
\end{eqnarray*}
which contradicts the definition of $\kappa_*$. This concludes the proof of the first statement.

\bigskip
To prove the second statement, let us assume that no separating hyperplane containing zero exists. Then, for all $u\in \mathbb{R}^d$, there exists $x\in \cX$ such that $z_x x^{\top}u\geq 0$. This implies that $\min_{u\in \mathbb{R}^d}\max_{x\in\cX}(z_x x^{\top}u+1) \geq 1$, so $\kappa_* \leq 1$. Choosing $u = 0$, we see that $\kappa_* \geq 1$, which implies that $\kappa_* = 1$.


\subsubsection{Proof of Lemma \ref{lem:kappa_tilde}}
Since for all $\gamma \in \cX$ and all $x \in \cX$, $\vert x^{\top}\gamma \vert\leq 1$, it is easy to see that the gaps are bounded by $2$, and that $\widetilde{\kappa} \leq 2 \kappa_*$. 

Let us now show that $\widetilde{\kappa} \geq \nicefrac{\kappa_*}{2}$.
\begin{eqnarray*}
  \left(x^{(1)}, x^{(2)}, \widetilde{\gamma}\right)&\in& \underset{(x,x') \in \cX, \gamma \in \cC(\cX)}{\argmax} (x - x')^{\top}\gamma\\
  \overline{x} &= & \frac{1}{2}(x^{(1)}+ x^{(2)})\\
  \widetilde{n} &=& \sum_{x\in\cX}\widetilde{\mu}(x)\\
  \text{and }\ \ \ \widetilde{x} &=& \frac{1}{\widetilde{n}}\sum_{x\in\cX}\widetilde{\mu}(x)x.
\end{eqnarray*}

Recall that $\kappa_*$ can equivalently be defined as the budget necessary to estimate the bias with a variance smaller than $1$. Therefore, we have 
\begin{equation}\label{eq:tilde_n}
    \widetilde{n} \geq \kappa_*.
\end{equation}
Let us define $\Delta_{\max}$ as $\Delta_{\max} = (x^{(1)}-x^{(2)})^{\top}\widetilde \gamma = \underset{(x,x') \in \cX, \gamma \in \cC(\cX)}{\max} (x - x')^{\top}\gamma$. By definition of $\widetilde{\kappa}$ and $\widetilde \mu$,
\begin{eqnarray*}
  \widetilde{\kappa} &\geq& \underset{x \in \cX}{\sum}\widetilde{\mu}(x)(x^{(1)} - x)^{\top}\widetilde{\gamma}\\
  & = & \widetilde{n}(x^{(1)} - \widetilde{x})^{\top}\widetilde{\gamma}.
\end{eqnarray*}
Using Equation \eqref{eq:tilde_n}, we find that
\begin{eqnarray}\label{eq:un_sens}
  \frac{\widetilde{\kappa}}{\kappa_*} &\geq& (x^{(1)} - \overline{x})^{\top}\widetilde{\gamma} + (\overline{x} - \widetilde{x})^{\top}\widetilde{\gamma}\nonumber\\
  &=& \frac{\Delta_{\max}}{2} + (\overline{x} - \widetilde{x})^{\top}\widetilde{\gamma}.
\end{eqnarray}
Now, since $\widetilde{\gamma} \in \cC(\cX)$, we also have $-\widetilde{\gamma} \in \cC(\cX)$, and therefore 
\begin{eqnarray*}
  \widetilde{\kappa} &\geq& \underset{x \in \cX}{\sum}\widetilde{\mu}(x)(x^{(2)} - x)^{\top}(-\widetilde{\gamma})\\
  &=&\widetilde{n}(\widetilde{x}-x^{(2)})^{\top}\widetilde{\gamma}
\end{eqnarray*}
Using again Equation \eqref{eq:tilde_n}, we find that
\begin{eqnarray}\label{eq:autre_sens}
  \frac{\widetilde{\kappa}}{\kappa_*}&\geq&(\widetilde{x}-\overline{x})^{\top}\widetilde{\gamma} +  (\overline{x} - x^{(2)})^{\top}\widetilde{\gamma}\nonumber\\
   &=& (\widetilde{x}-\overline{x})^{\top}\widetilde{\gamma} + \frac{\Delta_{\max}}{2}.
\end{eqnarray}
Combining Equations \eqref{eq:un_sens} and \eqref{eq:autre_sens}, we find that
\begin{eqnarray*}
  \frac{\widetilde{\kappa}}{\kappa_*} &\geq& \frac{\Delta_{\max}}{2} + \vert (\overline{x} - \widetilde{x})^{\top}\widetilde{\gamma}\vert.
\end{eqnarray*}
This implies in particular that $\widetilde\kappa\geq \frac{\Delta_{\max}\kappa_*}{2}$.

To conclude the proof of the Lemma, we show that $\Delta_{\max} \geq 1$. By contradiction, assume that $\Delta_{\max} < 1$. 

For all non-zero vector $u \in \mathbb{R}^d$, let us denote $x_u = \argmax_{x \in \cX}\vert x^{\top}u \vert$. Since $\cX$ spans $\mathbb{R}^d$, we necessarily have $\vert x_u^{\top}u \vert >0$, so we can define the normalized vector $\tilde{u} = u/\vert x_u^{\top}u \vert$ such that $\tilde{u}$ belongs to the set $\cC(\cX)$. Finally, denote $x_u^{(1)}, x_u^{(2)} \in \argmax_{x, x' \in \cX} (x_u^{(1)} - x_u^{(2)})^{\top} \tilde{u}$. Note that by definition of $\Delta_{\max}$, we always have $(x_u^{(1)} - x_u^{(2)})^{\top} \tilde{u} \leq \Delta_{\max} < 1$.

\underline{Case 1 : $x_u^{\top}\tilde{u}>0$} Then, by definition of $x_u$ and $x_u^{(1)}$, we see that ${x_u^{(1)}}^{\top}\tilde{u} = x_u^{\top}\tilde{u} = 1$. Then, $(x_u^{(1)} - x_u^{(2)})^{\top} \tilde{u} < 1$ implies that $1 - {x_u^{(2)}}^{\top} \tilde{u} < 1$, so ${x_u^{(2)}}^{\top}\tilde{u} >0$, and in particular ${x_u^{(2)}}^{\top}u >0$. 

\underline{Case 2 : $x_u^{\top}u<0$} Then, by definition of $x_u$ and $x_u^{(2)}$, we see that  ${x_u^{(2)}}^{\top}\tilde{u} = x_u^{\top}u = - 1$. Then $(x_u^{(1)} - x_u^{(2)})^{\top} \tilde{u} < 1$ implies that ${x_u^{(1)}}^{\top} \tilde{u} + 1 < 1$, so ${x_u^{(1)}}^{\top}\tilde{u} <0$, and in particular ${x_u^{(1)}}^{\top}u <0$. 

\medskip

Putting together Case 1 and Case 2, we see that ${x_u^{(1)}}^{\top}u$ and ${x_u^{(2)}}^{\top}u$ are of the same sign and are not null. By definition of $x_u^{(1)}$ and $x_u^{(2)}$, we conclude that for all $x \in \cX$, the sign of $x^{\top}u$ is the same, and that $x^{\top}u$ is not $0$. Since this is true for all non-zero vector $u$, this implies in particular that no hyperplane containing the origin can separate the actions, which contradicts the assumption that $\cX$ spans $\mathbb{R}^d$.


\subsubsection{Proof of Lemmas \ref{lem:c-opt} and \ref{lem:G-opt}}

We begin by proving Lemma \ref{lem:G-opt}. Recall that $\pi$ is a G-optimal design for the set $\{a_x : x\in \cX\}$, and that $\mu$ is defined as  $\mu(x) = \lceil m\pi(x)\rceil$ for all $x \in \cX$.

We first observe that $V(\pi)=A_{\pi}^\top A_{\pi}$, where $A_{\pi}$ is the matrix with lines given by $[\sqrt{\pi(x)} a_{x}^\top]_{x\in\cX}$. 
Since the supports of $\mu$ and $\pi$ are the same, we get that $\Image(A_{\pi}^\top)=\Image(A_{\mu}^\top)$. As a consequence
$$\Image(V(\pi))=\Image(A_{\pi}^\top)=\Image(A_{\mu}^\top)=\Image(V(\mu)),$$
and $x\in \Image(V(\mu))$ for all $x \in \cX$. This ensures that $a_x^\top\widehat\theta_{\mu}$ is an unbiased estimator of $a_x^\top\theta^*$. 

Furthermore  $V(\mu) \succcurlyeq mV(\pi)$, so the variance $a_x^\top V(\mu)^+ a_x$ of $a_x^\top\widehat\theta_{\mu}$ is upper-bounded by $a_x^\top V(\mu)^+ a_x \leq m^{-1} a_x^\top V(\pi)^{+}a_x$. Now, the General Equivalence Theorem of Kiefer and Pukelshein shows that $\max_{x\in \cX}a_x^\top V(\pi)^{+}a_x \leq d+1$. Thus, $a_x^\top V(\pi)^{+}a_x \leq m^{-1}(d+1)$.

We now prove Lemma \ref{lem:c-opt}. Recall that $\pi \in \cM^{\cX}_{e_{d+1}}$ is such that $e_{d+1}\in \Image V(\pi)$, and that $\mu$ is defined as  $\mu(x) = \lceil m\pi(x)\rceil$ for all $x \in \cX$. Using similar arguments, we can show that  $e_{d+1}\in \Image(V(\mu))$, which ensures that $e_{d+1}^\top\widehat\theta_{\mu}$ is an unbiased estimator of $e_{d+1}^\top\theta^*$. The second part of the Lemma follows directly using that $V(\mu) \succcurlyeq mV(\pi)$.

\subsubsection{Proof of Lemma \ref{lem:upsilon}}

Elfving's set $\mathcal{S}$ for estimating the bias in the biased linear bandit problem is given by 
$$\mathcal{S} = convex \ hull \left\{ \left({x \atop z_x}\right), \left({-x \atop -z_x}\right) : x \in \cX \right\},$$
or equivalently by 
$$\mathcal{S} = convex \ hull \left\{ \pm \left({z_xx \atop 1}\right) : x \in \cX \right\}.$$

Now, Theorem \ref{thm:Elfving} indicates that $\kappa_*^{-\nicefrac{1}{2}} e_{d+1}$ belongs to a supporting hyperplane of $\mathcal{S}$. We first show that when $\cA$ spans $\mathbb{R}^{d+1}$, any normal vector $w \in \mathbb{R}^{d+1}$ to this hyperplane is such that $w^{\top}e_{d+1}\neq 0$. 

By contradiction, let us assume that $\kappa_*^{-\nicefrac{1}{2}} e_{d+1}$ belongs to some supporting hyperplane $\mathcal{H}$ of $\mathcal{S}$ parametrized as $\mathcal{H} = \left\{a \in \mathbb{R}^{d+1} : a^{\top}w = b \right\}$, where the normal vector $w$ is of the form  $w = \left({u \atop 0}\right)$. Then, $\kappa_*^{-\nicefrac{1}{2}} e_{d+1} \in \mathcal{H}$, so  $\kappa_*^{-\nicefrac{1}{2}} e_{d+1}^{\top}w = b$, and thus $b = 0$. Now, $\mathcal{H}$ is a supporting hyperplane of $\mathcal{S}$, so for all $a \in \mathcal{S}$ we see that $a^{\top}w \leq b$. In particular, for all $x \in \cX$, $x^{\top}u\leq 0$ and $-x^{\top}u \leq 0$, so $x^{\top}u = 0$. This implies that $\cX$ is supported by an hyperplane in $\mathbb{R}^d$ with normal vector $u$, which contradicts our assumption that $\cA$ spans $\mathbb{R}^{d+1}$. Thus, the supporting hyperplane of $\mathcal{S}$ containing $\kappa_*^{-1/2}e_{d+1}$ has a normal vector $w \in \mathbb{R}^{d+1}$ such that $w^{\top}e_{d+1}\neq 0$. In particular, we can parameterize this hyperplane as $\mathcal{H}_{u, b} = \left\{ a \in \mathbb{R}^{d+1} :  a^{\top}\left({u \atop 1}\right) = b\right\}$ for some $b \in \mathbb{R}$ and $u \in \mathbb{R}^{d}$.

Now, if $\mathcal{H}_{u, b}$ is a supporting hyperplane of $\mathcal{S}$, then, by definition, $\mathcal{S}$ is contained in the half space $\left\{ a \in \mathbb{R}^{d+1} :  a^{\top}\left({u \atop 1}\right) \leq b\right\}$. In particular, for all $x \in \cX$, one must have
$z_x x^{\top}u + 1 \leq b $ and $-z_x x^{\top}u - 1 \leq b$ : therefore, for all $x \in \cX$, $\vert z_x x^{\top}u + 1 \vert \leq b $. Moreover, $\mathcal{H}_{u, b}$ is a supporting hyperplane of $\mathcal{S}$, so there exists an extreme point $a \in \mathcal{S}$ such that $a \in \mathcal{H}_{u, b}$. Note that $\mathcal{S}$ is the convex hull of $\left\{ \pm \left({z_xx \atop 1}\right) : x \in \cX \right\}$, so the extreme points of $\mathcal{S}$ are in $\left\{ \pm \left({z_xx \atop 1}\right) : x \in \cX \right\}$. In particular, this implies that $b = \max \left\{\vert z_x x^{\top}u + 1 \vert : x \in \cX\right\}$. Thus, the supporting hyperplane of $\mathcal{S}$ containing $\kappa_*^{-1/2}e_{d+1}$ is necessarily of the form $\mathcal{H}_{u, \max \left\{\vert z_x x^{\top}u + 1 \vert : x \in \cX\right\}}$.

On the one hand, $\kappa_*^{-\nicefrac{1}{2}}$ belongs to the boundary of $\mathcal{S}$ and therefore to a supporting hyperplane $\mathcal{H}_{u, \max \left\{\vert z_x x^{\top}u + 1 \vert : x \in \cX\right\}}$ of $\mathcal{S}$. Then, there exists $u \in \mathbb{R}^d$ such that $\kappa_*^{-\nicefrac{1}{2}} = \max \left\{\vert z_x x^{\top}u + 1 \vert : x \in \cX\right\}$. 

On the other hand, it is easy to verify that for all $u \in \mathbb{R}^d$, $\mathcal{H}_{u, \max \left\{\vert z_x x^{\top}u + 1 \vert : x \in \cX\right\}}$ is a supporting hyperplane of $\mathcal{S}$. Now, $\kappa_*^{-\nicefrac{1}{2}}e_{d+1}$ belongs to $\mathcal{S}$, so $\kappa_*^{-\nicefrac{1}{2}}e_{d+1}^{\top} \left({u \atop 1}\right) \leq \max \left\{\vert z_x x^{\top}u + 1 \vert : x \in \cX\right\}$.

These two results imply that 
$$\kappa_*^{-\nicefrac{1}{2}} = \min_{u \in \mathbb{R}^d} \max_{x \in \cX} \vert z_x x^{\top}u + 1 \vert$$
which proves the Lemma.

\subsubsection{Proof of Lemma \ref{lem:upsilon_alpha}}
We prove that $2(\sqrt{\kappa_*}-1)^2 \vee 1 \leq \alpha \leq 8(\kappa_* + 1)$. Lemma \ref{lem:upsilon_alpha} follows directly by noticing that $\alpha\geq 1$ and $\kappa_* \geq 1$.

Let us begin by proving that $2(\sqrt{\kappa_*}-1)^2\leq \alpha$ for $\kappa_* >1$ (otherwise this inequality is automatically verified). Note that for all $u \in \mathbb{R}^d$, $\underset{\lambda \rightarrow +\infty}{\text{lim}} \frac{1}{\max_{x \in \cX} \left(x^{\top}(\lambda u) + z_x\right)^2} = 0,$ so the minimum over $u \in \mathbb{R}^d$ of $\frac{1}{\max_{x \in \cX} \left(x^{\top}u + z_x\right)^2} = 0$ is attained for some vector $\tilde{u}\in \mathbb{R}^d$. Let us also denote $\tilde{x}\in \argmax_{x \in \cX} (z_x x^{\top}\tilde{u} + 1)^2$, such that 
$$\kappa_* =  \frac{1}{\left(z_{\tilde{x}}\tilde{x}^{\top}\tilde{u} + 1\right)^2}.$$ 

With these notations, we see that for all $x \in \cX$, $$(z_x x^{\top}\tilde{u}+1)^{2} \leq (z_{\tilde{x}}\tilde{x}^{\top}\tilde{u}+1)^{2} = \kappa_*^{-1} < 1.$$ This implies that for all $x \in \cX$,
$$z_x x^{\top}\tilde{u} \leq -1 + \kappa_*^{-1/2} < 0.$$

Now, let us denote $x^{(1)}, x^{(2)} \in \argmax_{x, x' \in \cX} (x-x')^{\top}\tilde{u}$. By definition of $\alpha$, we see that
$$\alpha \geq \frac{\left((x^{(1)}-x^{(2)})^{\top}\tilde{u}\right)^2}{\left(z_{\tilde{x}}\tilde{x}^{\top}\tilde{u} + 1\right)^2} = \left((x^{(1)}-x^{(2)})^{\top}\tilde{u}\right)^2 \times \kappa_*.$$ Since $z_x x^{\top}\tilde{u} < 0$ for all $x \in \cX$, and since no group is empty, we can conclude that there exists $x, x' \in \cX$ such that $x^{\top}\tilde{u} >0$ and $x'^{\top}\tilde{u} <0$. In particular, by definition of $x^{(1)}$ and $x^{(2)}$, we see that $(x^{(1)})^{\top}\tilde{u} >0$ and $(x^{(2)})^{\top}\tilde{u} <0$. Then, 
$$\left((x^{(1)}-x^{(2)})^{\top}\tilde{u}\right)^2 \geq \left((x^{(1)})^{\top}\tilde{u}\right)^2 + \left((x^{(2)})^{\top}\tilde{u}\right)^2 \geq 2(1-\kappa_*^{-1/2})^2.$$

This implies that 
$$\alpha \geq 2(1-\kappa_*^{-1/2})^2 \times \kappa_* = 2(\sqrt{\kappa_*}-1)^2.$$

Let us now prove that $\alpha \geq 1$. 
Note that by assumption, $\cX$ spans $\mathbb{R}^d$, and in particular there exists $\tilde{u} \in \mathbb{R}^d$ and $x, x' \in \cX$ such that $\max_{x\in \cX}x^{\top}\tilde{u}  >0$ and $\min_{x\in \cX}x^{\top}\tilde{u}  \leq0$. Thus, $\max_{x, x' \in \cX}((x - x')^{\top}\tilde{u})^2 \geq \max_{x \in \cX} (x^{\top}\tilde{u})^2$. For any $\lambda >0$, choosing $u = \lambda \tilde{u}$ in the definition of $\alpha$ implies that
$$\alpha \geq \frac{\lambda^2 \max_{x \in \cX} (x^{\top}u)^2}{\max_{x \in \cX}(\lambda z_x  x^{\top}u+1)^2}.$$

Letting $\lambda$ go to infinity, we find that $\alpha \geq 1$.

Finally, we prove that $\alpha\leq 8(\kappa_* + 1)$. For all $u \in \mathbb{R}^d$, we see that $$\frac{\max_{x, x' \in \cX}((x - x')^{\top}u)^2}{\max_{x \in \cX}(z_x  x^{\top}u+1)^2} \leq \frac{4\max_{x\in \cX}(z_x x^{\top}u)^2}{\max_{x \in \cX}(z_x  x^{\top}u+1)^2}.$$
Now, we see that $$\frac{\max_{x\in \cX}(z_x x^{\top}u)^2}{\max_{x \in \cX}(z_x  x^{\top}u+1)^2} \leq \frac{2\max_{x\in \cX}(z_xx^{\top}u+1)^2 + 2}{\max_{x \in \cX}(z_x  x^{\top}u+1)^2} \leq 2 + \frac{2}{\max_{x \in \cX}(z_x  x^{\top}u+1)^2}.$$
This in turn implies that for all $u \in \mathbb{R}^d$,
$$\frac{\max_{x, x' \in \cX}((x - x')^{\top}u)^2}{\max_{x \in \cX}(z_x  x^{\top}u+1)^2} \leq 8(1+\kappa_*),$$ which finally implies that $\alpha \leq 8(1+\kappa_*).$


\subsubsection{Proof of Lemma \ref{lem:prop_kappa}}
\noindent\textbf{Proof of Claim \ref{eq:prop_kappa_i}} 
The proof of the first claim is immediate by definition of $\kappa$. Indeed, let $\widetilde \cM = \left\{\mu \in \cM^{\cX}_{e_{d+1}} : e_{d+1}^{\top}V(\mu)^+e_{d+1} \leq 1\right\}$ be the set of measures $\mu$ admissible for estimating $\omega^*$ with a precision level $1$. Then,
\begin{eqnarray*}
  \kappa(c\Delta) &=&\min_{\mu \in \widetilde \cM}\sum_x \mu(x)c\Delta_x = c\min_{\mu \in \widetilde \cM}\sum_x \mu(x)\Delta_x = c\kappa(\Delta).
\end{eqnarray*}

\noindent\textbf{Proof of Claim \ref{eq:prop_kappa_ii}} 
The proof of the second claim is also straightforward. If $\Delta \leq \Delta'$, then for all $\mu \in \widetilde \cM$, $\sum_x \mu(x)\Delta_x \leq \sum_x \mu(x)\Delta'_x$. Recall that $\mu^{\Delta'} = \argmin_{\mu \in \widetilde \cM}\sum_x \mu(x)\Delta'_x$. Then,
\begin{eqnarray*}
  \kappa(\Delta') &=& \sum_x \mu^{\Delta'}(x)\Delta'_x \geq \sum_x \mu^{\Delta'}(x)\Delta_x \geq \min_{\mu \in \widetilde \cM}\sum_x \mu(x)\Delta_x = \kappa(\Delta).
\end{eqnarray*}

\noindent\textbf{Proof of Claim \ref{eq:prop_kappa_iii}} 
To prove the third claim, note that
\begin{eqnarray*}  
\kappa(\Delta \lor \Delta') &=& \min_{\mu \in \widetilde{\cM}}\sum_x \mu(x)\left(\Delta_x \lor \Delta_x\right)\\
  &\geq &  \min_{\mu \in \widetilde \cM} \left(\sum_x \mu(x)\Delta_x \lor \sum_x \mu(x)\Delta'_x \right) \\
  &\geq & \left(\min_{\mu \in \widetilde \cM} \sum_x \mu(x)\Delta_x\right) \lor \left(\min_{\mu \in \widetilde \cM}\sum_x \mu(x)\Delta'_x\right)\\
  &\geq &\kappa(\Delta)\lor \kappa(\Delta').
\end{eqnarray*}

\noindent\textbf{Proof of Claim \ref{eq:prop_kappa_iv}} Recall that 
$$\kappa(\Delta) =\min_{\mu \in \widetilde \cM}\sum_x \mu(x)\Delta_x.$$
Let us define a sequence $(\mu_n)_{n\in \mathbb{N}} \in \widetilde \cM^\mathbb{N}$ such that 
$\sum_x \mu_n(x)\Delta_x\underset{n\rightarrow\infty}{\rightarrow} \kappa(\Delta),$
and let us denote $\kappa_n = \sum_x \mu_n(x)\Delta_x$. According to Claim  \ref{eq:prop_kappa_ii}, 
we have 
$$\kappa(\Delta) \leq \kappa(\Delta \vee \epsilon) = \min_{\mu \in \widetilde{\cM}}\sum_x \mu(x)\left(\Delta_x \vee \epsilon \right)
  \leq   \sum_x \mu_n(x)\Delta_x + \epsilon  \sum_x \mu_n(x).$$
It follows that  for all $n$,
$$\kappa(\Delta) \leq \liminf_{\epsilon \to 0^+}\kappa(\Delta \vee \epsilon) \leq  \limsup_{\epsilon \to 0^+}\kappa(\Delta \vee \epsilon)\leq \kappa_{n}.$$
Letting $n$ go to infinity, we get that $\lim_{\epsilon\to 0^+} \kappa(\Delta\vee \epsilon)=\kappa(\Delta)$.


\subsubsection{Proof of Lemma \ref{lem:calcul:kappa}}

Setting $\mu\cdot\Delta=(\mu(x)\Delta_{x})_{x\in\cX}$ and
 $$V_{\Delta}(\lambda)=\sum_{x\in\cX} \lambda_{x} \left({\Delta_{x}^{-1/2}x \atop \Delta_{x}^{-1/2}z_x}\right) \left({\Delta_{x}^{-1/2}x \atop \Delta_{x}^{-1/2}z_x}\right)^{\top},$$
 we observe that 
 $V_{\Delta}(\mu\cdot\Delta)=V(\mu)$. Hence, 
 $$\kappa(\Delta)=\min_{\substack{\mu \in \mathcal{M}^+\\ e_{d+1}^{\top}V_{\Delta}(\mu\cdot \Delta)^{+}e_{d+1}\leq 1}}\ \sum_{x\in\cX} (\mu\cdot\Delta)_{x}.$$
 We observe that $e_{d+1}\in \Image(V(\mu))$ is equivalent to $e_{d+1}\in \Image(V_{\Delta}(\mu\cdot \Delta))$.
 Hence, $\mu^{\Delta}\cdot \Delta = \lambda^{\Delta}$ where 
 $$\lambda^{\Delta }\in\argmin_{\substack{\lambda\in\R_{+}^{\cX} \\ e_{d+1}\in \Image(V_{\Delta}(\lambda))\\ e_{d+1}^{\top}V_{\Delta}(\lambda)^{+}e_{d+1}\leq 1}}\ \sum_{x\in\cX} \lambda_{x}.$$
The conclusion then follows by noticing that by homogeneity, $\lambda^{\Delta}=\kappa^{\Delta} \pi^{\Delta}$.


\subsubsection{Proof of Lemma \ref{lem:proba_F}}
Lemma \ref{lem:proba_F} follows directly from Lemmas \ref{lem:control_gamma} and \ref{lem:control_omega}.

\begin{lem}\label{lem:control_gamma}
\begin{equation*}
\mathbb{P}\left(\exists l \geq 1, z\in \{-1,1\} \text{ such that Explore}_l^{(z)} = \text{True}  \text{, and } x\in \mathcal{X}_l^{(z)} \text{ such that } \ \left\vert  \left({\hgamma_l^{(z)} - \gamma^* \atop \homega_l^{(z)} - \omega^*}\right)^{\top} \left({x \atop z_x}\right)\right \vert \geq \epsilon_l\right) \leq  \delta.
\end{equation*}
\end{lem}

\begin{lem}\label{lem:control_omega}
\begin{equation*}
\mathbb{P}\left(\exists l\geq 1 \text{ such that Explore}_l^{(0)} = \text{True and } \left\vert  \homega_l^{(0)}- \omega^*\right \vert \geq \epsilon_l\right) \leq  \delta.
\end{equation*}
\end{lem}

\subsubsection{Proof of Lemma \ref{lem:discard_subopt}}
To prove Lemma \ref{lem:discard_subopt}, we rely on the following key lemma. This lemma proves that on $\overline{\mathcal{F}}$, i.e. when the error bounds hold, the algorithm never eliminates the best action or the best group.

\begin{lem}\label{lem:E}
 On the event $\overline{\mathcal{F}}$, for all $x^*\in \argmax_{x\in \mathcal{X}} x^{\top}\gamma^*$ and all $l$ such that  Explore$_l^{(z_{x^*})} = $ True, $x^* \in \mathcal{X}_{l+1}^{(z_{x^*})}$. Moreover, on the event $\overline{\mathcal{F}}$, for all $l$ such that Explore$_l^{(0)} = $ True, there exists  $x^*\in \argmax_{x\in \mathcal{X}} x^{\top}\gamma^*$ such that $\widehat{z^*}_{l+1} \neq -z_{x^*}$.
\end{lem}

Let $l\geq 1$ be such that Explore$_l^{(z_{x^*})}=$ True. Then, on $\overline{\mathcal{F}}$, $x^* \in \mathcal{X}_{l+1}^{(z_{x^*})}$ by Lemma $\ref{lem:E}$. Moreover, for all $x \in \mathcal{X}_{l+1}^{(z_{x^*})}$, by definition of $\mathcal{X}_{l+1}^{(z_{x^*)}}$, we have that on $\overline{\mathcal{F}}$
\begin{equation*}
    \left(\left({x^* \atop z_{x^*}}\right) - \left({x \atop z_{x^*}}\right)\right)^{\top}\left({\hgamma_l^{(z)}\atop\homega_l^{(z)}} \right)\leq 3\epsilon_l.
\end{equation*}
which implies that
\begin{equation*}
    \left(\left({x^* \atop z_{x^*}}\right) - \left({x \atop z_{x^*}}\right)\right)^{\top}\left({\gamma^*\atop\omega^*} \right)  \leq 3\epsilon_l + \left\vert\left({x^* \atop z_{x^*}}\right)^{\top}\left({\hgamma_l^{(z)} - \gamma^*\atop\homega_l^{(z)}- \omega^*} \right)\right\vert + \left\vert\left({x \atop z_{x^*}}\right)^{\top}\left({\hgamma_l^{(z)} - \gamma^*\atop\homega_l^{(z)}- \omega^*} \right)\right\vert.
\end{equation*}
Thus, on the event $\overline{\mathcal{F}}$, for all $x \in \mathcal{X}_{l+1}^{(z_{x^*})}$
\begin{equation*}
    \left(x^*-x\right)^{\top}\gamma^* < 5\epsilon_l\,,
\end{equation*}
which proves Equation \eqref{eq:subopt_goodz}.
To prove the second claim of Lemma \ref{lem:discard_subopt}, assume that for all $x'\in \argmax_{x\in \mathcal{X}} x^{\top}\gamma^*$, $z_{x'} = z_{x^*}$ (when this does not hold, the second claim follows from Equation \eqref{eq:subopt_goodz}). Now, let $l\geq 1$ be such that Explore$_l^{(-z_{x^*})}=$ True. By Lemma $\ref{lem:E}$, on $\overline{\mathcal{F}}$, $x^* \in \mathcal{X}_{l}^{(z_{x^*})}$ and $\widehat{z^*_l} = 0$. Then, the algorithm is unable to determine the group containing the best set during the phase Exp$_{l-1}^{(0)}$, so there must exist $x' \in \mathcal{X}_{l}^{(-z_{x^*})}$ such that 
\begin{equation*}
    \left({x^* \atop z_{x^*}}\right)^{\top}\left({\hgamma_{l-1}^{(z_{x^*})} \atop \homega_{l-1}^{(z_{x^*})}}\right)\leq \left({x' \atop -z_{x^*}}\right)^{\top}\left({\hgamma_{l-1}^{(-z_{x^*})} \atop \homega_{l-1}^{(-z_{x^*})}}\right) + 2z_{x^*}\homega_{l-1}^{(0)}+ 4\epsilon_{l-1}.
\end{equation*}
It follows that
\begin{equation*}
     \left({x^* - x' \atop 2 z_{x^*}}\right)^{\top}\left({\gamma^* \atop \omega^*}\right) \leq \left({x^* \atop z_{x^*}}\right)^{\top}\left({\gamma^* - \hgamma_{l-1}^{(z_{x^*})} \atop \omega^* - \homega_{l-1}^{(z_{x^*})}}\right)  + \left({x' \atop -z_{x^*}}\right)^{\top}\left({\hgamma_{l-1}^{(-z_{x^*})}-\gamma^* \atop \homega_{l-1}^{(-z_{x^*})}-\omega^*}\right)+ 2z_{x^*}\homega_{l-1}^{(0)}+ 4\epsilon_{l-1}.
\end{equation*}
On $\overline{\mathcal{F}}$, this implies that 
\begin{equation*}
    \left({x^* - x' \atop 2 z_{x^*}}\right)^{\top}\left({\gamma^* \atop \omega^*}\right) < 2z_{x^*}\homega_{l-1}^{(0)}+ 6\epsilon_{l-1}
\end{equation*}
so
\begin{equation}\label{eq:trans-gr}
    \left(x^* - x'\right)^{\top}\gamma^* \leq 2z_{x^*}\left(\homega^{(0)}_{l-1} - \omega^* \right)+ 6\epsilon_{l-1} < 8\epsilon_{l-1} = 16\epsilon_{l}.
\end{equation}
Moreover, for all $x\in\cX_{l+1}^{(-z_{x^*})}$ we have $(a_{x'}-a_{x})^\top \widehat\theta_{l}^{(-z_{x^*})} \leq 3\epsilon_{l}$, so following the same lines as for the first claim, we get $(x'-x)^{\top}\gamma^*<5\epsilon_{l}$. Combining this bound with (\ref{eq:trans-gr}), we get 
$$\max_{x\in\cX_{l+1}^{(-z_{x^*})}}(x^*-x)^\top \gamma^* < 21 \epsilon_{l}.$$
This concludes the proof of Lemma \ref{lem:discard_subopt}.


\subsubsection{Proof of Lemma \ref{lem:time}}
For $z \in \{-1, +1\}$ and $l>0$,
\begin{equation*}
    \sum_{x} \mu^{(z)}_{l}(x) \leq \sum_{x} \frac{2(d+1)\pi^{(z)}_{l}(x)}{\epsilon_l^2} \log\left(\frac{kl(l+1)}{ \delta}\right) + \vert\supp(\pi^{(z)}_l) \vert.
\end{equation*}
Now, $\supp(\pi^{(z)}_l)\leq \frac{(d+1)(d+2)}{2}$ and $\sum_x \pi^{(z)}_l(x) = 1$, so 
\begin{equation*}
    \sum_{x} \mu^{(z)}_{l}(x) \leq \frac{2(d+1)}{\epsilon_l^2} \log\left(\frac{kl(l+1)}{ \delta}\right) + \frac{(d+1)(d+2)}{2}
\end{equation*}
which proves the first claim of Lemma \ref{lem:time}. 

To prove the second claim, we bound the regret for bias estimation at stage $l$ as follows. On $\overline{\mathcal F}$, we have $\Delta_{x}\leq \widehat \Delta_{x}^l$ for all $x\in\cX$ and $l\geq 1$, so
\begin{align*}
\sum_{x\in \cX} \mu_{l}^{(0)}(x) \Delta_{x} & \ \leq\ \sum_{x\in \cX} \mu_{l}^{(0)}(x) \widehat{\Delta}_{x}^l.
\end{align*}
Recall that $\hat\mu_{l}$ is the $\widehat{\Delta}^l$-optimal design, and that for all $x \in \cX$, $\mu^{(0)}_{l}(x) = \lceil\frac{2\hat\mu_{l}(x)}{\epsilon_l^2}\log\left(\frac{l(l+1)}{\delta}\right) \rceil$. Since $\widehat{\Delta}_{x}^l\leq 2$ for all $x\in \cX$, we have
\begin{equation*}
    \sum_{x\in\cX} \mu^{(0)}_{l}(x)\widehat{\Delta}_{x}^l \leq \sum_{x\in\cX} \frac{2\hat \mu_{l}(x)}{\epsilon_l^2} \log\left(\frac{l(l+1)}{ \delta}\right)\widehat{\Delta}_{x}^l + 2 \vert\supp(\mu^{(0)}_l) \vert
\end{equation*}
and $|\supp(\mu^{(0)}_l)|\leq d+1$, so 
\begin{equation*}
    \sum_{x} \mu^{(0)}_{l}(x) \Delta_{x} \leq  \frac{2}{\epsilon_l^2} \log\left(\frac{l(l+1)}{ \delta}\right)\sum_{x\in \cX} \hat \mu_{l}(x) \widehat{\Delta}_{x}^l+ 2(d+1).
\end{equation*}
By definition of $\hat \mu_{l}(x)$, we have that
\begin{align*}
\sum_{x\in \cX} \widehat \mu_{l}(x) \widehat\Delta_{x}^l =\kappa(\widehat \Delta^{l}).
\end{align*}
It follows that, on $\overline{\mathcal F}$, 
$$ \sum_{x} \mu^{(0)}_{l}(x) \Delta_{x} \leq \sum_{x} \mu^{(0)}_{l}(x) \widehat{\Delta}^l_{x}\leq \frac{2}{\epsilon_l^2} \log\left(\frac{l(l+1)}{ \delta}\right) \kappa(\widehat \Delta^{l})+2(d+1).$$

%

\subsubsection{Proof of Lemma \ref{lem:kappa-l}}

For the first claim, we rely on the next lemma.
\begin{lem}\label{lem:kappa-l:tech}
Let us set $\ell_{x}=\max\ac{l\geq 1: x\in\cX_{l}^{(-1)}\cup\cX_{l}^{(1)}}$. 
On $\overline{\mathcal F}$, we have for any $l\geq 1$
\begin{enumerate}
\item $\widehat \Delta^l_{x} \leq \Delta_{x} +16\epsilon_{l}$ for all $x\in\cX_{l}^{(-1)}\cup\cX_{l}^{(1)}$ (i.e. for all $x$ such that $l\leq \ell_{x}$);
\item if $\Delta_{x}\geq 21\epsilon_{l}$ then $\ell_{x} \leq l$;
\item $\epsilon_{\ell_{x}}<\Delta_{x}$ for all $x\in\cX$.
\end{enumerate}
\end{lem}

Lemma \ref{lem:kappa-l} relies on the following remarks : if $\Delta, \Delta'$ are such that $\Delta_{x}\leq \Delta'_{x}$ for all $x\in\cX$, then by Lemma \ref{lem:prop_kappa} \eqref{eq:prop_kappa_ii}, $\kappa(\Delta)\leq \kappa(\Delta')$. Let us now prove that for all $l \geq 1$ and all $x\in\cX$, $\widehat \Delta_{x}^l \leq 513(\Delta \lor \epsilon_l)$.

\noindent{\bf Case $\epsilon_{l}\geq \Delta_{x}$.} On $\cbF$, we have $l\leq \ell_{x}-1$ according to the third claim of Lemma \ref{lem:kappa-l:tech}. So, on $\cbF$,
$$\widehat \Delta^l_{x} \leq \Delta_{x} +16\epsilon_{l} \leq 17 (\Delta_{x} \vee\epsilon_{l}).$$

\noindent{\bf Case $\epsilon_{l}< \Delta_{x}$.} Then, on $\cbF$, we have $32 \epsilon_{l+5} < \Delta_{x}$ and so $l+5\geq \ell_{x}$ according to the second claim of Lemma~\ref{lem:kappa-l:tech}. Hence, on $\cbF$,  according to  Lemma~\ref{lem:kappa-l:tech}, we have
\begin{align*}
\widehat \Delta_{x}^l &\leq \max_{k=0,\ldots,5} \widehat \Delta_{x}^{\ell_{x}-k} \leq \Delta_{x}+16 \epsilon_{\ell_{x}-5}\\
&\leq \Delta_{x}+512 \epsilon_{\ell_{x}} \leq 513 \Delta_{x}.
\end{align*}
Thus, for all $l \geq 1$ and all $x\in\cX$, 
\begin{equation*}
    \widehat \Delta_{x}^l \leq 513(\Delta \lor \epsilon_l).
\end{equation*} 

Now, let $\widetilde \cM = \left\{\mu \in \cM^{\cX}_{e_{d+1}} : e_{d+1}^{\top}V(\mu)^+e_{d+1} \geq 1\right\}$ the measures $\mu$ admissible for estimating $\omega^*$ with a precision level $1$. Note that for all $a, b, c>0$, 
\begin{equation}\label{eq:ineq}
    (1+ab^{-1})(c\vee b) = (c + cab^{-1})\vee(a + b)\geq c \vee(a+b)\geq c\vee a.
\end{equation} 
Using Equation \eqref{eq:ineq} with $a= \Delta_{x}$, $b = \tau$ and $c = \epsilon$, we see that
\begin{eqnarray*}
  \kappa(\Delta \vee \epsilon) &=&\min_{\mu \in \widetilde \cM}\sum_x \mu(x)(\Delta_x\vee \epsilon) \leq (1+\epsilon/\tau)  \min_{\mu \in \widetilde \cM}\sum_x \mu(x)(\Delta_x \vee \tau) =  (1+\epsilon/\tau)\kappa(\Delta \vee \tau).
\end{eqnarray*}
Using Lemma \ref{lem:prop_kappa} together with $\widehat \Delta_{x}^l \leq 513(\Delta \lor \epsilon_l)$, we find that
\begin{eqnarray*}
  \kappa(\widehat \Delta_{x}^l) \leq 513\kappa(\Delta \lor \epsilon_l) \leq 513(1+\epsilon_l/\tau)\kappa(\Delta \vee \tau).
\end{eqnarray*}
This proves the first claim of Lemma \ref{lem:kappa-l}.
\bigskip

To prove the second claim, we use Lemma \ref{lem:prop_kappa} and the fact that for all $x$, $\widehat{\Delta}_{x}^l \geq \epsilon_l$. Moreover, on $\overline{\mathcal{F}}$, $\widehat \Delta_{x}^l \geq \Delta_x$ for all $x\in \cX$. Then, $\kappa(\widehat \Delta)\geq \kappa(\epsilon_l \lor \Delta)$ by Lemma \ref{lem:prop_kappa} \eqref{eq:prop_kappa_iii}.


\subsubsection{Proof of Lemmas \ref{lem:discard_z} }
To prove Lemma \ref{lem:discard_z}, let us consider $l$ such that $\epsilon_l \leq \frac{\Delta_{\neq}}{8}$. According to Lemma \ref{lem:E}, on $\overline{\mathcal{F}}$ we know that $\widehat{z^*_{l}} \neq -z_{x^*}$. 
When $\widehat{z^*}_{l} = z_{x^*}$, then we also have $\widehat{z^*}_{l+1} = z_{x^*}$ and the conclusion follows immediately.
Let us consider now the case where $\widehat{z^*_{l}}=0$. By definition of $\Delta_{\neq}$, for all $x' \in \mathcal{X}_{l+1}^{(-z_{x^*})}$,
\begin{equation*}
    \left(x^* - x'\right)^{\top}\gamma^* \geq\Delta_{\neq}.
\end{equation*}
This implies that 
\begin{eqnarray*}
        \left({x^* \atop z_{x^*}}\right)^{\top}\left({\hgamma_l^{(z_{x^*})} \atop \homega_l^{(z_{x^*})}}\right) - z_{x^*}\homega_l^{(0)} &\geq&\underset{x \in\mathcal{X}_{l+1}^{(-z_{x^*})}}{\max} \left({x \atop -z_{x^*}}\right)^{\top}\left({\hgamma_l^{(-z_{x^*})} \atop \homega_l^{(-z_{x^*})}}\right) +z_{x^*}\homega_l^{(0)}\\
        &&  + \left({x^* \atop z_{x^*}}\right)^{\top}\left({\hgamma_l^{(z_{x^*})}-\gamma^* \atop \homega_l^{(z_{x^*})}-\omega^*}\right)  + \underset{x \in\mathcal{X}_{l+1}^{(-z_{x^*})}}{\min} \left({x \atop -z_{x^*}}\right)^{\top}\left({\gamma^*-\hgamma_l^{(-z_{x^*})} \atop \omega^*-\homega_l^{(-z_{x^*})}}\right)  \\
        && + \Delta_{\neq}+ 2z_{x^*}\left(\omega^*-\homega_l^{(0)}\right).
\end{eqnarray*}
On $\overline{\mathcal{F}}$, it follows that
\begin{eqnarray*}
        \left({x^* \atop z_{x^*}}\right)^{\top}\left({\hgamma_l^{(z_{x^*})} \atop \homega_l^{(z_{x^*})}}\right) - z_{x^*}\homega_l^{(0)} -2\epsilon_{l}&\geq&\underset{x \in\mathcal{X}_{l+1}^{(-z_{x^*})}}{\max} \left({x \atop -z_{x^*}}\right)^{\top}\left({\hgamma_l^{(-z_{x^*})} \atop \homega_l^{(-z_{x^*})}}\right) +z_{x^*}\homega_l^{(0)} - 6\epsilon_l + \Delta_{\neq}.
\end{eqnarray*}
When $\Delta_{\neq} \geq 8\epsilon_l$, this implies that $\widehat{z^*_{l+1}} = z_{x^*}$.



\subsubsection{Proof of Lemmas \ref{lem:kappa_2} and \ref{lem:kappa}}
We prove Lemma \ref{lem:kappa_2}. The proof of Lemma \ref{lem:kappa} follows by noticing that the two actions sets are equal up to a permutation of the direction of some basis vectors. To prove Lemma \ref{lem:kappa}, we rely on Elfving's characterization of $c$-optimal design, given in Theorem \ref{thm:Elfving}. Theorem \ref{thm:Elfving} shows that for $\pi\in \mathcal{P}^{\{1,.., d+1\}}$ to be $e_{d+1}$-optimal, there must exist $t>0$ and $\zeta \in \{-1,+1\}^{d+1}$ such  that 

\begin{eqnarray*}
\underset{1 \leq i \leq d+1}{\sum}\pi_{i} &=& 1\\
0 &=& \pi_{1}\zeta_{1} - (1- \frac{2}{\sqrt{\kappa_*}+1})\pi_{{d+1}}\zeta_{{d+1}}\\
\forall i \in \{2, ..., d\}, \ 0 &=& \pi_{i}\zeta_{i}\\
t &=& \underset{1 \leq i \leq \lfloor d/2\rfloor }{\sum}\pi_{i}\zeta_{i} -\underset{\lfloor d/2\rfloor +1 \leq i \leq d+1}{\sum}\pi_{i}\zeta_{i}.
\end{eqnarray*}
Solving this system, we find that $t^{-2} = \kappa_*$. Note  that the unicity of the solution for the corresponding probability measure $\pi$ guarantees that $te_{d+1}$ belongs to the boundary of $\mathcal{S}$.

\subsubsection{Proof of Lemma \ref{lem:borne_inf_kappa}}
For a given parameter $\gamma^*$, let us denote by $\Delta_i$ the gap corresponding to the action $i$. To compute $\kappa(\Delta)$, we could want to rely on Lemma \ref{lem:calcul:kappa} to find the $\Delta$-optimal design, corresponding to the $e_{d+1}$-optimal design on the rescaled features $\Delta_x^{-1/2}\left({x \atop z_x}\right)$. Theorem \ref{thm:Elfving} indeed allows us to compute such a design, as seen in the proof of Lemma \ref{lem:kappa_2}. Unfortunately, we cannot rescale the features using the true gaps, since $\Delta_{x^*}=0$. To circumvent this problem, we rely on the following reasoning :
\begin{enumerate}
    \item We use Lemma \ref{lem:calcul:kappa} and Theorem \ref{thm:Elfving} to compute the design $\mu^{\Delta \vee \epsilon}$ for $\epsilon \in (0, \Delta_{\min})$; and the corresponding regret $\kappa(\Delta \vee \epsilon)$;
    \item We find the value of $\kappa(\Delta)$ by noticing that $\epsilon \mapsto \kappa(\Delta \vee \epsilon)$ is continuous at 0.
\end{enumerate}

For $\epsilon\in (0,\Delta_{\min})$, define $\overline{\Delta} = \Delta \vee \epsilon$, and $\overline{x} = \overline{\Delta}_x^{-1/2}x$. Let $\overline{\pi}$ denote the $e_{d+1}$-optimal design for the rescaled features $\overline{x}$, and let $\overline{\kappa_*}$ denote its variance. Then, Lemma \ref{lem:calcul:kappa} ensures that $\kappa(\overline{\Delta}) = \overline{\kappa_*}$.

Now, Theorem \ref{thm:Elfving} shows that there exists $\zeta \in \{-1,+1\}^{d+1}$ such  that 
\begin{eqnarray*}
\underset{1 \leq i \leq d+1}{\sum}\overline{\pi}_{i} &=& 1\\
0 &=& \overline{\pi}_{1}\zeta_{1}\overline{\Delta}_1^{-1/2} - (1- \frac{2}{\sqrt{\kappa_*}+1})\overline{\pi}_{{d+1}}\zeta_{{d+1}}\overline{\Delta}_{d+1}^{-1/2}\\
\forall i \in \{2, ..., d\}, \ 0 &=& \overline{\pi}_{i}\zeta_{i}\overline{\Delta}^{-1/2}_{i}\\
\overline{\kappa_*}^{-1/2} &=& \underset{1 \leq i \leq \lfloor d/2\rfloor }{\sum}\overline{\pi}_{i}\zeta_{i}\overline{\Delta}^{-1/2}_{i} -\underset{\lfloor d/2\rfloor +1 \leq i \leq d+1}{\sum}\overline{\pi}_{i}\zeta_{i}\overline{\Delta}^{-1/2}_{i}
\end{eqnarray*}
and $\overline{\kappa_*}^{-1/2}e_{d+1}$ belongs to the boundary of $\mathcal{S}$.
Solving this system, we find that $$\kappa(\overline{\Delta})^{-1/2} = \overline{\kappa_*}^{-1/2} = \frac{\left(\frac{2}{\sqrt{\kappa_*}+1}\right)\overline{\Delta}^{-1/2}_{d+1}}{1 + \left(1-\frac{2}{\sqrt{\kappa_*}+1}\right)\overline{\Delta}^{-1/2}_{d+1}\overline{\Delta}_1^{1/2}}.$$ As in Lemma \ref{lem:kappa_2}, the unicity of the solution for the corresponding probability measure $\overline{\pi}$ guarantees that $\overline{\kappa_*}^{-1/2}e_{d+1}$ belongs to the boundary of the Elfving's set. Now, $\epsilon \leq \Delta_{\min}$, so 
$$\kappa(\overline{\Delta})^{-1/2}= \kappa(\Delta \vee \epsilon)^{-1/2}= \frac{\left(\frac{2}{\sqrt{\kappa_*}+1}\right){\Delta}^{-1/2}_{d+1}}{1 + \left(1-\frac{2}{\sqrt{\kappa_*}+1}\right){\Delta}^{-1/2}_{d+1}{\epsilon}^{1/2}}.$$

The fourth claim of Lemma \ref{lem:prop_kappa} ensures that $ \kappa(\Delta \vee \epsilon) \underset{\epsilon \rightarrow 0}{\rightarrow} \kappa(\Delta)$. Therefore, 
$$\kappa(\Delta) = \underset{\epsilon \rightarrow 0}{\text{lim}} \left(\frac{\left(\frac{2}{\sqrt{\kappa_*}+1}\right){\Delta}^{-1/2}_{d+1}}{1 + \left(1-\frac{2}{\sqrt{\kappa_*}+1}\right){\Delta}^{-1/2}_{d+1}{\epsilon}^{1/2}}\right)^{-2} = \frac{(\sqrt{\kappa_*}+1)^2{\Delta}_{d+1}}{4}.$$


\subsubsection{Proof of Lemma \ref{lem:control_gamma}}
Recall that $\xi_t = y_t - x_{t}^{\top}\gamma^* - z_{x_t}\omega^*$. For $l \geq 0$ and $z\in \{-1,+1\}$, when  $\text{ Explore}_l^{(z)} = \text{True}$, the least square estimator $\left({\hgamma_l^{(z)} \atop \homega_l^{(z)}}\right)$ is given by
\begin{eqnarray*}
    \left({\hgamma_l^{(z)} \atop \homega_l^{(z)}}\right) &=& \left(V^{(z)}_l\right)^{+} \sum_{t \in \text{Exp}_l^{(z)}}\left(\left({x_t \atop z_{x_t}}\right)^{\top}\left({\gamma^* \atop \omega^*}\right)+ \xi_t\right)\left({x_t \atop z_{x_t}}\right)\\
    &=& \left(V^{(z)}_l\right)^{+} \left(V^{(z)}_l\right)   \left({\gamma^* \atop \omega^*}\right)  +\left(V^{(z)}_l\right)^{+}  \sum_{t \in \text{Exp}_l^{(z)}} \xi_t\left({x_t \atop z_{x_t}}\right),
\end{eqnarray*}
where $\left(V^{(z)}_l\right)^{+}$ is a generalized inverse of $V^{(z)}_l$. 
Since $V^{(z)}_l\left(V^{(z)}_l\right)^+V^{(z)}_l=V^{(z)}_l$, multiplying the left and right hand side of the last equation by $V^{(z)}_l$, we find that
\begin{eqnarray}\label{eq:def_error_gamma}
    V^{(z)}_l\left({\hgamma_l^{(z)} - \gamma^* \atop \homega_l^{(z)} - \omega^*}\right) &=&  V^{(z)}_l\left(V^{(z)}_l\right)^{+}\sum_{t \in \text{Exp}_l^{(z)}} \xi_t \left({x_t \atop z_{x_t}}\right).
\end{eqnarray}
By Lemma~\ref{lem:G-opt}, for all $x \in \mathcal{X}_{l}^{(z)}$, $\left({x \atop z_x}\right) \in \Image\left(V^{(z)}_l\right)$,  so 
 \begin{equation}\label{eq:Image_inversible}
     V^{(z)}_l\left(V^{(z)}_l\right)^{+}\left({x \atop z_x}\right) = \left({x \atop z_x}\right).
 \end{equation} Then, 
 \begin{eqnarray*}
\left({\hgamma_l^{(z)} - \gamma^* \atop \homega_l^{(z)} - \omega^*}\right)^{\top} \left({x \atop z_x}\right) &=&  \left({\hgamma_l^{(z)} - \gamma^* \atop \homega_l^{(z)} - \omega^*}\right)^{\top}V^{(z)}_l \left(V^{(z)}_l\right)^{+} \left({x \atop z_x}\right) \\
 &=&  \sum_{t \in \text{Exp}_l^{(z)}} \left({x_t \atop z_{x_t}}\right)^{\top} \left(V^{(z)}_l\right)^{+}V^{(z)}_l \left(V^{(z)}_l\right)^{+}\left({x \atop z_x}\right) \xi_t\\
  &=&  \sum_{t \in \text{Exp}_l^{(z)}} \left({x_t \atop z_{x_t}}\right)^{\top} \left(V^{(z)}_l\right)^{+}\left({x \atop z_x}\right) \xi_t,
\end{eqnarray*}
where the first and third lines follow from Equation \eqref{eq:Image_inversible}, and the second line follows from Equation \eqref{eq:def_error_gamma}. By definition of our algorithm, conditionally on $\mathcal{X}_l^{(z)}$ and $\text{ Explore}_l^{(z)} = \text{True}$, the variables $\left(\xi_t \right)_{t \in \text{Exp}^{(z)}_l}$ are independent centered normal gaussian variables. Then, 

\begin{eqnarray*}
    &\mathbb{P}_{\vert \mathcal{X}_l^{(z)}, \text{ Explore}_l^{(z)} = \text{True}}\left(\left\vert \left({\hgamma_l^{(z)} - \gamma^* \atop \homega_l^{(z)} - \omega^*}\right)^{\top} \left({x \atop z_x}\right)\right \vert \geq \sqrt{2\sum_{t \in \text{Exp}_l^{(z)}}\left(\left({x_t \atop z_{x_t}}\right)^{\top} \left(V^{(z)}_l\right)^{+}\left({x \atop z_x}\right)\right) ^2\log\left(\frac{kl(l+1)}{ \delta}\right)}\right) \leq \frac{ \delta}{kl(l+1)}.
\end{eqnarray*}
Expanding $\left(\left({x_t \atop z_{x_t}}\right)^{\top} \left(V^{(z)}_l\right)^{+}\left({x \atop z_x}\right)\right) ^2 = \left({x \atop z_x}\right)^{\top}\left(V^{(z)}_l\right)^{+}\left({x_t \atop z_{x_t}}\right)\left({x_t \atop z_{x_t}}\right)^{\top}\left(V^{(z)}_l\right)^{+}\left({x \atop z_x}\right)$, and using the definition of $V^{(z)}_l$, we find that
\begin{eqnarray*}
     &\mathbb{P}_{\vert \mathcal{X}_l^{(z)}, \text{ Explore}_l^{(z)} = \text{True}}\left(\left\vert \left({\hgamma_l^{(z)} - \gamma^* \atop \homega_l^{(z)} - \omega^*}\right)^{\top} \left({x \atop z_x}\right)\right \vert \geq \sqrt{2\left({x \atop z_x}\right)^{\top}\left(V^{(z)}_l\right)^{+} V^{(z)}_l\left(V^{(z)}_l\right)^{+}\left({x \atop z_x}\right)\log\left(\frac{kl(l+1)}{ \delta}\right)}\right) \leq \frac{ \delta}{kl(l+1)}
\end{eqnarray*}
which in turn implies (using Equation \eqref{eq:Image_inversible})
\begin{eqnarray*}
&\mathbb{P}_{\vert \mathcal{X}_l^{(z)}, \text{ Explore}_l^{(z)} = \text{True}}\left(\left\vert \left({\hgamma_l^{(z)} - \gamma^* \atop \homega_l^{(z)} - \omega^*}\right)^{\top} \left({x \atop z_x}\right)\right \vert \geq \sqrt{2\left \Vert \left({x \atop z_x}\right)\right \Vert_{\left(V^{(z)}_l\right)^{+}}^2\log\left(\frac{kl(l+1)}{ \delta}\right)}\right) \leq \frac{ \delta}{kl(l+1)}
\end{eqnarray*}
Now, using Lemma \ref{lem:G-opt} and the definition of $\mu_{l}^{z}$, we see that for all $x \in \cX_{l}^{(z)}$, 
\begin{equation*}
    \left({x \atop z_{x}}\right)^{\top}\left(V^{(z)}_l\right)^{+} \left({x \atop z_{x}}\right)\leq \frac{\epsilon_l^2}{2\log\left(kl(l+1)/ \delta\right)}.
\end{equation*}
Finally, for all $x \in \mathcal{X}_l^{(z)}$,
\begin{eqnarray*}
\lefteqn{\mathbb{P}_{\vert \mathcal{X}_l^{(z)}, \text{ Explore}_l^{(z)} = \text{True}}\left(\left\vert \left({\hgamma_l^{(z)} - \gamma^* \atop \homega_l^{(z)} - \omega^*}\right)^{\top} \left({x \atop z_x}\right)\right \vert \geq \epsilon_{l}\right)}\\
&\leq &\mathbb{P}_{\vert \mathcal{X}_l^{(z)}, \text{ Explore}_l^{(z)} = \text{True}}\left(\left\vert \left({\hgamma_l^{(z)} - \gamma^* \atop \homega_l^{(z)} - \omega^*}\right)^{\top} \left({x \atop z_x}\right)\right \vert \geq \sqrt{2\left \Vert \left({x \atop z_x}\right)\right \Vert_{\left(V^{(z)}_l\right)^{+}}^2\log\left(\frac{kl(l+1)}{ \delta}\right)}\right) \leq \frac{ \delta}{kl(l+1)}\,.
\end{eqnarray*}
Integrating out the conditioning on the value of $\mathcal{X}^{(z)}_l$ and $\text{ Explore}_l^{(z)}$ and using a union bound yields the desire result.


\subsubsection{Proof of Lemma \ref{lem:control_omega}}
The proof is similar to that of Lemma \ref{lem:control_gamma}. If $\text{Explore}_l^{(0)} = \text{True}$, then $\homega_l$ is defined as

\begin{eqnarray*}
    \homega_l^{(0)}&=& e_{d+1}^{\top}\left(V^{(0)}_l\right)^{+} \sum_{t \in \text{Exp}_l^{(0)}}\left(\left({x_t \atop z_{x_t}}\right)^{\top}\left({\gamma^* \atop \omega^*}\right)+ \xi_t\right)\left({x_t \atop z_{x_t}}\right).
\end{eqnarray*}
Since $\left({x \atop z_x}\right)_{x \in \mathcal{X}}$ spans $\mathbb{R}^{d+1}$, $\mu$ is finite and $e_{d+1} \in \Image\left(V(\hat\mu_{l})\right)$. Then, according to Lemma~\ref{lem:c-opt}, for every round $l$, we have $e_{d+1} \in \Image\left(V^{(0)}_l\right)$, so $V^{(0)}_l\left(V^{(0)}_l\right)^{+}e_{d+1}= e_{d+1}$. This implies that
\begin{eqnarray*}
    \homega_l^{(0)}- \omega^* &=&   \sum_{t \in \text{Exp}_l^{(0)}}e_{d+1}^{\top}\left(V^{(0)}_l\right)^{+}\left({x_t \atop z_{x_t}}\right)\xi_t.
\end{eqnarray*}
By definition of our algorithm, conditionally on $\text{Explore}_l^{(0)} = \text{True}$, the variables $\left(\xi_t \right)_{t \in \text{Exp}_l^{(0)}}$ are independent centered normal gaussian variables. Then, 

\begin{equation*}
\mathbb{P}_{\vert \text{Explore}_l^{(0)} = \text{True}}\left(\left \vert \homega_l^{(0)}- \omega^* \right \vert \geq \sqrt{2\sum_{t \in \text{Exp}_l^{(z)}} \left(e_{d+1}^{\top}\left(V^{(0)}_l\right)^{+} \left({x_t \atop z_{x_t}}\right)\right)^2\log\left(\frac{l(l+1)}{ \delta} \right)}\right) \leq \frac{ \delta}{l(l+1)}.
\end{equation*}
Using again $V^{(0)}_l\left(V^{(0)}_l\right)^{+}e_{d+1}= e_{d+1}$ and the definition of $V^{(0)}_l$, we find that 
\begin{equation}\label{eq:bound_omega}
\mathbb{P}_{\vert \text{Explore}_l^{(0)} = \text{True}}\left(\left \vert \homega_l^{(0)}- \omega^* \right \vert \geq \sqrt{2 e_{d+1}^{\top}\left(V^{(0)}_l\right)^{+}e_{d+1}\log\left(\frac{l(l+1)}{ \delta} \right)}\right) \leq  \frac{\delta}{l(l+1)}.
\end{equation}
Now, Lemma \ref{lem:c-opt} and the definition of $\mu_l^{(0)}$ imply that
\begin{equation*}
    e_{d+1}^\top\left(V^{(0)}_l\right)^{+} e_{d+1} \leq \frac{\epsilon_l^2}{2\log\left(l(l+1)/ \delta\right)}.
\end{equation*}
Finally, Equation \eqref{eq:bound_omega} implies that
\begin{equation*}
\mathbb{P}_{\vert \text{Explore}_l^{(0)} = \text{True}}\left(\left \vert \homega_l^{(0)}- \omega^* \right \vert \geq \epsilon_l\right) \leq \frac{ \delta}{l(l+1)}.
\end{equation*}
Using a union bound over the phases $\text{Exp}^{(0)}_l$ yields the result.

\subsubsection{Proof of Lemma \ref{lem:E}}

To prove Lemma \ref{lem:E}, we begin by showing that it is enough to prove that for $l\geq 1$, 
\begin{eqnarray}\label{eq:distinction_cas}
  \mathcal{F}_l \supset&& \left\{\exists x^* \in \argmax_{x \in \mathcal{X}} x^\top \gamma^* : \text{Explore}_l^{(z_{x^*})} = \text{ True and }x^* \notin \mathcal{X}_{l+1}^{(z_{x^*})}\right\} \\
  & \bigcup&  \Bigg\{\bigcap_{l'\leq l}\overline{\left\{\exists x^* \in \argmax_{x \in \mathcal{X}} x^\top \gamma^* : \text{Explore}_{l'}^{(z_{x^*})} = \text{ True and }x^* \notin \mathcal{X}_{l'+1}^{(z_{x^*})}\right\}}\nonumber\\
  && \ \ \  \bigcap\left\{\text{Explore}_l^{(0)} = \text{ True and } \forall  x^* \in \argmax_{x \in \mathcal{X}} x^\top \gamma^*,  \widehat{z^*}_{l+1} = - z_{x^*}\right\}\ \ \ \Bigg\}.  \nonumber
\end{eqnarray}
Indeed, denoting $\mathcal{F}_l^{(1)} = \left\{\exists x^* \in \argmax_{x \in \mathcal{X}}x^\top \gamma^* :\text{Explore}_l^{(z_{x^*})} = \text{ True and }x^* \notin \mathcal{X}_{l+1}^{(z_{x^*})}\right\}$ and\newline \noindent $\mathcal{F}_l^{(2)} = \left\{\text{Explore}_l^{(0)} = \text{ True and }\forall  x^* \in \argmax_{x \in \mathcal{X}}x^\top \gamma^*, \widehat{z^*}_{l+1} = -z_{x^*}\right\}$, we see that Equation \eqref{eq:distinction_cas} would then be rewritten as
\begin{eqnarray*}
  \mathcal{F}_l &\supset& \mathcal{F}_l^{(1)} \bigcup \left\{\bigcap_{l'\leq l}\overline{\mathcal{F}_{l'}^{(1)}} \bigcap\mathcal{F}_l^{(2)}\right\}
\end{eqnarray*}
which implies
\begin{eqnarray*}
  \bigcup_{l\geq 1}\mathcal{F}_l \ \ \supset \ \ \  \bigcup_{l\geq 1}\left\{\mathcal{F}_l^{(1)} \bigcup \left\{\left\{\bigcap_{l'\leq l}\overline{\mathcal{F}_{l'}^{(1)}} \bigcap\mathcal{F}_l^{(2)}\right\}\bigcup_{l'\leq l} \mathcal{F}_{l'}^{(1)}\right\}\right\} \ \ \ \ \supset \ \ \  \bigcup_{l\geq 1}\left\{\mathcal{F}_l^{(1)} \cup \mathcal{F}_l^{(2)}\right\}.
\end{eqnarray*}
Then, Equation \eqref{eq:distinction_cas}  would imply that
\begin{eqnarray*}
    \overline{\mathcal{F}} \ = \ \overline{\bigcup_{l\geq 1}\mathcal{F}_l} \ \subset\ \overline{\bigcup_{l\geq 1}\left\{\mathcal{F}_l^{(1)} \bigcup \mathcal{F}_l^{(2)}\right\}} \ = \ \bigcap_{l\geq 1}\left\{\overline{\mathcal{F}_l^{(1)}} \bigcap \overline{\mathcal{F}_l^{(2)}}\right\},
\end{eqnarray*}
thus proving Lemma \ref{lem:E}. To prove Equation \eqref{eq:distinction_cas}, we show that both $\mathcal{F}^{(1)}_l$ and  $\bigcap_{l'\leq l}\overline{\mathcal{F}_{l'}^{(1)}} \bigcap\mathcal{F}_l^{(2)}$ imply $\mathcal{F}_l$.

\smallskip
\underline{\textbf{If $\mathcal{F}^{(1)}_l$ is true:}} then $\exists x^* \in \argmax_{x \in \mathcal{X}}:$ Explore$_l^{(z_{x^*})} = $ True and $x^* \notin \mathcal{X}_{l+1}^{(z_{x^*})}$.
\newline
Without loss of generality, assume that $l>1$ is the smallest integer such that Explore$_l^{(z_{x^*})} = $ True and $x^* \notin \mathcal{X}_{l+1}^{(z_{x^*})}$. Then, necessarily $x^* \in \mathcal{X}_{l}^{(z_{x^*})}$ (because either $l=1$, or Explore$_{l-1}^{(z_{x^*})} = $ True). Now, because $x^* \in \mathcal{X}^{(z_{x^*})}_{l}\setminus \mathcal{X}^{(z_{x^*})}_{l+1}$, there exists $x \in \mathcal{X}^{(z_{x^*})}_{l}$ such that
\begin{equation*}
    (x - x^*)^{\top}\hgamma_l^{(z_{x^*})}\geq 3\epsilon_l
\end{equation*} 
and in particular
\begin{equation*}
    x^{\top}\hgamma_l^{(z_{x^*})} - \epsilon_l  > (x^*)^{\top}\hgamma_l^{(z_{x^*})} + \epsilon_l.
\end{equation*} 
Recall that by definition of $x^*$, $(\gamma^*)^{\top}(x^* - x) \geq 0$. This in turn implies that 
\begin{equation*}
     \left({x \atop z_{x^*}} \right)^{\top}\left({\hgamma_l^{(z_{x^*})} - \gamma^* \atop \homega_l^{(z_{x^*})} - \omega^*} \right) - \epsilon_l > \left({x^* \atop z_{x^*}} \right)^{\top}\left({\hgamma_l^{(z_{x^*})} - \gamma^* \atop \homega_l^{(z_{x^*})} - \omega^*} \right) + \epsilon_l.
\end{equation*} 
The last equation implies that either $\left({x \atop z_{x}}\right)^{\top}\left({\gamma_l^{(z)} - \gamma^* \atop \homega_l^{(z)} - \omega^*}\right) > \epsilon_l$ or $\left({x^* \atop z_{x^*}}\right)^{\top}\left({\gamma_l^{(z)} - \gamma^* \atop \homega_l^{(z)} - \omega^*}\right) < - \epsilon_l$, which in turn implies $\mathcal{F}_l$.

\underline{\textbf{If $\bigcap_{l'\leq l}\overline{\mathcal{F}_{l'}^{(1)}} \bigcap\mathcal{F}_l^{(2)}$ is true:}} then $\text{Explore}_l^{(0)} = \text{ True and }\forall  x^* \in \argmax_{x \in \mathcal{X}} x^\top \gamma^*,  \widehat{z^*_{l+1}} = - z_{x^*}$. Moreover, for all $l' \leq l$, $\text{Explore}_{l'}^{(z_{x^*})} = \text{ False or }x^* \in \mathcal{X}_{l'+1}^{(z_{x^*})}$.
\newline
Note that this case can only hold if all optimal actions $x^*$ belong to the same group $z_{x^*}$. Without loss of generality, assume that $l>1$ is the smallest integer such that $\text{Explore}_l^{(0)} = \text{ True and }\widehat{z^*_{l+1}} = -z_{x^*}$, and for all $l' \leq l$, $\text{Explore}_{l'}^{(z_{x^*})} = \text{ False or }x^* \in \mathcal{X}_{l'+1}^{(z_{x^*})}$. Note that because $\text{Explore}_l^{(0)} = $ True, necessarily $\text{Explore}_{l'}^{(z_{x^*})} =$ True for all $l' \leq l$, and in particular $x^* \in \mathcal{X}^{(z_{x^*})}_{l+1}$.

Then, there exists $x \in\mathcal{X}_{l+1}^{(-z_{x^*})}$ such that
\begin{equation*}
    \left({x \atop -z_{x^*}}\right)^{\top}\left({\hgamma_l^{(-z_{x^*})} \atop \homega_l^{(-z_{x^*})}}\right) - \left({x^* \atop z_{x^*}}\right)^{\top}\left({\hgamma_l^{(z_{x^*})} \atop \homega_l^{(z_{x^*})}}\right) + 2z_{x^*}\homega_l^{(0)}\geq 4\epsilon_l.
\end{equation*} 
Recall that all optimal actions $x^*$ are in the same group $z_{x^*}$, so $(\gamma^*)^{\top}(x^* - x) >0$. This in turn implies that 
\begin{equation*}
    \left({x \atop -z_{x^*}}\right)^{\top}\left({\hgamma_l^{(-z_{x^*})} - \gamma^* \atop \homega_l^{(-z_{x^*})} - \omega^*}\right) - \left({x^* \atop z_{x^*}}\right)^{\top}\left({\hgamma_l^{(z_{x^*})} - \gamma^* \atop \homega_l^{(z_{x^*})}- \omega^*}\right) + 2z_{x^*}(\homega_l^{(0)}- \omega^*) \geq 4\epsilon_l.
\end{equation*} 
The last equation implies that either $\left({x \atop -z_{x^*}}\right)^{\top}\left({\hgamma_l^{(-z_{x^*})} - \gamma^* \atop \homega_l^{(-z_{x^*})} - \omega^*}\right) \geq \epsilon_l$, or $\left({x^* \atop z_{x^*}}\right)^{\top}\left({\hgamma_l^{(z_{x^*})} - \gamma^* \atop \homega_l^{(z_{x^*})}- \omega^*}\right) \leq -\epsilon_l$, or $z_{x^*}(\homega_l^{(0)}- \omega^*) \geq \epsilon_l$, which in turn implies $\mathcal{F}_l$.


\subsubsection{Proof of Lemma \ref{lem:kappa-l:tech}}
The first claim holds for $l=1$. For $l\geq 1$, for any $x\in\cX_{l+1}^{(-1)}\cup\cX_{l+1}^{(1)}$, we have $\widehat \Delta^{l+1}_{x} \leq \Delta_{x} +8\epsilon_{l}$ on $\overline{\mathcal F}$ according to the definition of $\widehat \Delta^{l+1}$ and $\mathcal{F}$. The first claim then follows.

For the second claim, Lemma \ref{lem:discard_subopt} gives that, on $\overline{\mathcal F}$, $\Delta_{x}<21\epsilon_{l}$ for any $x\in\cX_{l+1}^{(-1)}\cup\cX_{l+1}^{(1)}$. 
So $\Delta_{x}\geq 21\epsilon_{l}$ implies $x\notin \cX_{l+1}^{(-1)}\cup\cX_{l+1}^{(1)}$ and hence $l\geq \ell_{x}$ on $\overline{\mathcal F}$.

For the third claim, we notice that
$$\max_{x'\in \cX^{(z_{x})}_{\ell_{x}}} (a_{x'}-a_{x})^\top \widehat\theta^{(z_{x})}_{\ell_{x}} > 3\epsilon_{\ell_{x}},$$
since $x\notin \cX_{\ell_{x}+1}$. Since the left-hand side is smaller than $\Delta_{x}+2\epsilon_{\ell_{x}}$ on $\overline{\mathcal F}$, we get
$\Delta_{x} > \epsilon_{\ell_{x}}$.

\end{document}